\documentclass[11pt]{article}
 \usepackage{amsmath, amsthm, amssymb, bbm, setspace,bigints}
 \usepackage[margin=1 in]{geometry}
\usepackage{caption}
\usepackage{subcaption}
\usepackage[toc,page]{appendix}
\usepackage{cite}         
\usepackage[pdftex]{graphicx}
\usepackage{pdfpages}
\usepackage{epstopdf}
\usepackage{booktabs}
\usepackage{authblk}

\usepackage{amsthm,bbm,amssymb}
\usepackage{amsmath}
\usepackage{natbib}
\usepackage[colorlinks,citecolor=blue,urlcolor=blue,filecolor=blue,backref=page]{hyperref}
\usepackage{graphicx}

\usepackage[utf8]{inputenc}
\usepackage{amsmath,amssymb,natbib,xcolor,multicol,url,mhequ}
\usepackage{graphicx,bbm,xr}
\usepackage{booktabs,epstopdf,color}
\usepackage[space]{grffile}
\usepackage{lineno}
\usepackage{multirow}

\def\prime{\mathrm{\scriptscriptstyle{T}}}

\def\T{{ \mathrm{\scriptscriptstyle T} }}

\def\mb{\mathbb}

\def\ind{\mathbbm{1}}

\def\T{\mathrm{\scriptscriptstyle{T}}}

\renewcommand{\r}{\mathrm{r}}

\DeclareMathOperator*{\argmin}{arg\,min}

\def\mb{\mathbb}

\def\T{{\mathrm{\scriptscriptstyle T} }}

\def\ind{\mathbbm{1}}
\newcommand \bbP{\mathbb{P}}
\newcommand \bbE{\mathbb{E}}
\newcommand{\be}{\begin{equs}}
\newcommand{\ee}{\end{equs}}

\renewcommand{\k}{\mathrm{k}}
\renewcommand{\r}{\mathrm{r}}

\newcommand{\trd}[1]{\left[#1\right]}
\newcommand{\set}[1]{\left\{#1\right\}}
\def \ind{\mathbbm{1}}
\newcommand{\p}[1]{\left(#1\right)}

\numberwithin{equation}{section}
\theoremstyle{plain}
\newtheorem{theorem}{Theorem}[section]
\newtheorem{lemma}{Lemma}[section]
\newtheorem{definition}{Definition}[section]

\newtheorem{proposition}{Proposition}[section]
\newtheorem{remark}{Remark}[section]

\usepackage{lipsum}
\usepackage{amsfonts,amsmath}
\usepackage{tikz-cd}
\usepackage{epstopdf}
\usepackage{algorithm} 
\usepackage{algorithmic}  
\usepackage[algo2e,ruled,vlined]{algorithm2e}

\def\T{{ \mathrm{\scriptscriptstyle T} }}
\def \bX{\mathbf{X}}
\newcommand\bx{\mathbf{x}}
\def \mb{\mathbb}
\def\D{\mathrm{D}}

\title{Statistical optimality and stability of tangent transform algorithms in logit models 
}
\author[1]{Indrajit Ghosh\thanks{indrajit@stat.tamu.edu}}
\author[1]{Anirban Bhattacharya\thanks{anirbanb@stat.tamu.edu}}
\author[1]{Debdeep Pati\thanks{debdeep@stat.tamu.edu}}

\affil[1]{Department of Statistics,  Texas A\&M University, College Station, Texas, 77843, USA}

\begin{document}
\maketitle

\begin{abstract}
A systematic approach to finding variational approximation in an otherwise intractable non-conjugate model is to exploit the general principle of convex duality by minorizing the marginal likelihood that renders the problem tractable. While such approaches are  popular in the context of variational inference in non-conjugate Bayesian models,  theoretical guarantees on statistical optimality and algorithmic convergence are lacking.  Focusing on logistic regression models,  we provide mild conditions on the data generating process to derive non-asymptotic upper bounds to the risk incurred by the variational optima.  We demonstrate that these assumptions can be completely relaxed if one considers a slight variation of the algorithm by raising the likelihood to a fractional power. Next, we utilize the theory of dynamical systems to provide convergence guarantees for such algorithms  in logistic and multinomial logit regression.   In particular, we establish local asymptotic stability of the algorithm without any assumptions on the data-generating process. We explore a special case involving a semi-orthogonal design under which a global convergence is obtained. The theory is further illustrated using several numerical studies. 

\end{abstract}


\section{Introduction}

Variational Inference (VI) has gained substantial momentum in recent years as an efficient way of performing approximate Bayesian inference. VI seeks to minimize a divergence measure between a tractable family of probability distributions and the posterior distribution, utilizing optimization based techniques to arrive at a minima.  In many high dimensional examples where sampling based techniques such as the Markov chain Monte Carlo require expert vigilance and care for scalability, VI provides a viable answer with relatively lower computational cost. Some notable application areas include graphical models \citep{wainwright2008graphical,jordan1999introduction}, hidden markov models \citep{mackay1997ensemble}, latent class models \citep{blei2003latent}, neural networks \citep{graves2011practical} to name a few. Refer to Chapter 10 in \cite{bishop2006pattern} and \cite{blei2017variational} for excellent reviews on the topic. 

The empirical success of VI has prompted researchers to investigate their theoretical properties. Two distinct directions of research seem to have emerged over the last few years. One line of research concerns the statistical aspects of variational estimators \citep{alquier2016properties,pati2017statistical,yang2017alpha,cherief2018consistency,alquier2020concentration,zhang2020convergence,wang2019frequentist,wang2019variational} in a general setting, delineating sufficient conditions on the data generation mechanism and the variational family under which the variational estimators have optimal first or second-order statistical properties. 
Motivated by the robustness properties of a fractional likelihood \citep{bhattacharya2019bayesian,alquier2020concentration}, \cite{yang2017alpha} proposed a simple modification to the variational objective function, deemed as the $\alpha$-Variational Bayes ($\alpha$-VB), that only requires the variational family to be sufficiently flexible and the prior density to be appropriately concentrated around the true parameter to obtain optimal risk bounds.  

The other line of research studies convergence of the algorithms employed to arrive at the variational optimizer.  In this aspect, the coordinate ascent variational inference (CAVI) algorithm for mean-field VI (refer to Chapter 10 of \cite{bishop2006pattern}) has arguably received the most attention due to its simplicity and generality. 
An early result on algorithmic convergence (and lack thereof) of CAVI in Gaussian mixture models appears in \cite{wang2006convergence}. \cite{zhang2020theoretical,mukherjee2018mean} analyzed CAVI for stochastic block models, a popular model for networks belonging to cEXP. \cite{yin2020theoretical} obtained convergence of cluster labels in a stochastic block model by considering a structured variational family which was not possible using mean field VI. \cite{ghorbani2018instability} noted instability of naive mean-field VI in latent Dirichlet allocation and provided a remedy by optimizing a different type of free energy (TAP) instead of the standard variational objective. 
\cite{locatello2018boosting,campbell2019universal} analyzed convergence of a more flexible class of boosting algorithms which aim to approximate the target class by a mixture of Gaussians rather than a single Gaussian or a product distribution.  

Our goal in this article is to explore a popular class of variational approximations outside cEXP, called the {\em tangent-transform approach} \citep{jaakkola1997variational,jaakkola2000bayesian}. The tangent transform approach is an example of a structured variational approximation, lying on the spectrum between the two extremes given by the restrictive mean-field inference and the highly flexible variational boosting. In this specific instance, the structure exploited is convex duality \citep{jordan1999introduction,wainwright2003variational,wainwright2005new} to minorize the log-likelihood function and provide sharp bounds for the log-partition functions in exponential family models. Assume $p(x, \theta)$ is an exponential family on a discrete space $\mathcal{X}$ indexed by parameter $\theta \in \Theta$,
\begin{eqnarray*}
p(x; \theta) = \exp\{\langle \theta, t(x) \rangle - B(\theta)\}, \quad  B(\theta) = \log \Big[\sum_{x \in \mathcal{X}} \exp \{ \langle \theta, t(x)\rangle\}\Big]. 
\end{eqnarray*}
The
log-partition function $B(\theta)$, a convex function of $\theta$, plays a critical role in computing summary measures of $p(x; \theta)$. 
\cite{jaakkola2000bayesian} exploits the dual representation of the log-partition function in terms of its Fenchel-Legendre conjugate 
$
B(\theta) = \sup_{\mu \in \mathcal{M}} [ \langle \theta, t(x) \rangle - \{-H(\mu)\}] ,
$
where $H(\mu)$ is negative entropy of the distribution parameterized by 
$\mu$ and $\mathcal{M}$ is the marginal polytope. 

Ideas related to the tangent-transform have found widespread applications ranging from approximate inference in graphical models \citep{jordan1999introduction}, low-rank approximations \citep{srebro2003weighted}, inference in large scale generalized linear models \citep{nickisch2009convex}, non-conjugate latent Gaussian models \citep{emtiyaz2013fast} to more recently in sparse kernel machines \citep{shi2019integrating}, hierarchical relevance determination \citep{hirose2020hierarchical}, online prediction \citep{konagayoshi2019minimax} among others. \cite{jaakkola2000bayesian} exploits convex duality to minorize the marginal likelihood by introducing a variational parameter that allows the minorant to be arbitrarily close to the marginal likelihood.  Logistic and multinomial logit regression models are notable examples where a clever use of this idea results in a straightforward Expectation-Maximization (EM) algorithm to compute the variational parameters.  

In this article, we investigate both the statistical and algorithmic aspects of the tangent transform algorithm in logit models. Despite its widespread usage, statistical properties of the point estimate of the regression coefficients resulting from a tangent transform algorithm has not been previously studied. One possible reason is that unlike mean-field VI, where the global objective is to minimize the Kullback--Leibler (or another) divergence between a product distribution and the posterior distribution, the tangent transform algorithm is defined {\em locally}, without a clear global objective function that is being minimized. A key observation underlying our statistical analysis expresses any stationary point of the EM algorithm as a minimzer of a suitably chosen {\em global} variational objective function. This observation allows us to extend some previously developed variational risk bounds for mean-field VI \citep{yang2017alpha,pati2017statistical} to the present setting with some non-trivial modifications. 
We show that with minimal assumptions on the data generating process and the prior density on the regression coefficients, the variational risk bound is minimax optimal (up to logarithmic terms).  Moreover, the assumption on the data generating process can be completely relaxed by raising the standard logistic likelihood by a fractional exponent \citep{bhattacharya2019bayesian}.

Next, we  investigate the convergence of the EM algorithm to the fixed point of the EM iterations. 
There has been some previous efforts to shed more light into the EM sequence of tangent-transform algorithms. 
\cite{hunter2004tutorial} studied connections between minimization-majorization (MM) in case of logistic likelihood to argue convergence of the updates. \cite{durante2019conditionally} drew a connection with the P\'{o}lya-Gamma data augmentation technique \citep{polson2013bayesian} to provide a probabilistic interpretation of the EM updates and showed that the optimal evidence lower bound of the tangent transformation approach coincides with the same obtained in a bonafide variational inference with a suitably defined conditionally conjugate exponential family. However, statistical and algorithmic convergence guarantees of the tangent transform itself continue to be an open problem. It may appear on the surface that the EM algorithm underlying tangent-transforms can be analyzed using the general sufficient conditions for convergence of the EM (we refer to the recent article \citep{balakrishnan2017statistical} and the references therein for more on this topic), a careful inspection however reveals that these general-purpose conditions pose significant difficulty to verify for the present EM iterates and demand stringent conditions on the design matrix and other data generating parameters. Our approach, on the other hand, is to directly analyze the EM sequence without resorting to any high-level results. 

By viewing the EM updates as iterations in a discrete time autonomous dynamical system, we show that the EM updates converge to the desired fixed point under suitable initialization, a phenomenon known as local asymptotic stability.  While stability is typically a weaker statement compared to global convergence as it only ensures convergence if the system is initialized in a neighborhood around the fixed point, our stability result is essentially assumption-free -- it does not require any assumption on the design matrix, on the sparsity of the coefficients, and on the dimension $p$ and the sample size $n$. Although the notion of such convergence is local, to the best of our knowledge, this is the first assumption-free result on the stability of a variational algorithm. The main technical contribution is to show that the spectral radius of the Jacobian matrix of the linearized operator of  the EM sequence is strictly smaller than one at the fixed point. In the special case when the design matrix is semi-orthogonal, we show that the EM sequence is globally convergent with an exponential rate of convergence (logarithmic run time) independent of the initialization. We also provide a straightforward extension of this result to the case of  multinomial logit model.

\section{Tangent Transformation Approach}\label{sec:tt}
Denote the data by $X$ and the likelihood conditioned on parameter $\theta \in \Theta$ by $p(X \mid \theta)$, where $\Theta$ is the parameter space.  For a prior density $p(\cdot)$ on $\Theta$, the goal of VI is to approximate the posterior $p(\theta \mid X ) \propto p(\theta) \, p(X \mid \theta)$ by a member of a tractable family $\Gamma$ of densities on $\Theta$ with respect to the Kullback--Leibler (KL) divergence.\footnote{The KL divergence between densities $f$ and $g$, denoted $\mbox{D}(f  \, \| \ g)$, is $D(f  \, \| \ g) := \int f \log (f/g) d\mu$, where $\mu$ is a common dominating measure.}
Notationally, VI seeks to find 
\begin{align}\label{eq:KL}
\hat{q} = \argmin_{q \in \Gamma} \mbox{D}(q \, \| \, p(\cdot \mid  X)), 
\end{align}
which is equivalent to maximizing the evidence lower bound (ELBO), $\mathcal{L}(q) = \int q(\theta)\log\{p(X,\theta)/q(\theta)\}\,d\theta$ with respect to $q \in \Gamma$. Using a component-wise product structure on $\Gamma$ popularly known as the mean field family \citep{parisi1988statistical}, closed-form updates of a coordinate ascent algorithm (CAVI) can be generally derived in conditionally conjugate exponential families \citep{blei2017variational}. However,  many non-conjugate models such as logistic regression, multinomial logit regression, graphical and, topic models, do not lead to closed-form CAVI updates, necessitating various specialized techniques \citep{jordan1999introduction,blei2017variational}. One such approach is to introduce variational parameters to minorize the log-marginal likelihood by a tractable family, which when combined with an appropriate prior enjoys conjugate inference. For Bayesian logistic regression models, \cite{jaakkola2000bayesian} introduced a tangent-transform of the logistic function using convex duality. By a standard result in convex analysis \citep{rockafellar1970convex}, a convex function $f(\cdot)$ on $\mb R^d$ can be represented via a {\em conjugate} or {\em dual} function $f^*$ as,  
\begin{eqnarray}\label{eq:convdual}
f(x) = \max_{\lambda} \{\langle \lambda, x\rangle - f^*(\lambda)\}, \quad f^*(\lambda) = \max_{x} \{\langle \lambda, x\rangle - f(x)\}.
\end{eqnarray}
One simple example of \eqref{eq:convdual} is $ x^2 = \max_{\lambda}\{\lambda x - \lambda^2/4\}$ with equality at $x=\lambda/2$. Similarly, for a concave $f(\cdot)$ we have $f(x) = \min_{\lambda} \{\langle \lambda, x\rangle - f^*(\lambda)\}$ with the dual  being $f^*(\lambda) = \min_{x} \{\langle \lambda, x\rangle - f(x)\}$. 
Geometrically, the evaluation of a convex function at any point $x$ can be viewed as the maxima of the uncountable collection of hyperplanes $\langle \lambda, x \rangle - f^*(\lambda)$ indexed by $\lambda \in \mb R^d$. 

The usage of duality is not restricted to linear approximations, i.e., hyperplanes. In fact, \cite{jaakkola2000bayesian} used a quadratic bound for the logistic function that induces conjugacy with Gaussian priors. 
In the following subsection, we discuss the salient features of the tangent transform approach. 

\subsection{Convex minorant construction for logistic likelihood}
We discuss a slightly general version of the tangent transform approach where we raise the usual logistic likelihood to a power $\alpha \in (0, 1]$ before combining with the prior. Variational Bayes procedures with fractional likelihoods have been recently considered in \cite{yang2017alpha, alquier2020concentration, alquier2016properties}. The case $\alpha = 1$ recovers the usual tangent transform. 

Assuming we observe binary responses $y_i$ corresponding to fixed covariates 
$\bx_i \in \mathbb{R}^p\ (i=1, \ldots, n)$, consider the usual logistic regression model,
\begin{eqnarray}\label{sec_TTA_eqn0}
y_i \mid  \bx_i,\beta &\sim& \mbox{Bernoulli}(p_i), \quad 
p_i = \frac{1}{1+\exp(-\bx^{\T}_i \beta)} \quad (i = 1, \ldots, n). 
\end{eqnarray}
Denote by $\bX$ the $n \times p$ covariate matrix with $i$th row $\bx_i^{\T}\ (i =1, 2, \ldots, n)$. Consider a Gaussian prior $\beta \sim \mbox{N}_p(\mu_{\beta}, \Sigma_{\beta})$, denoted by $\pi(\beta)$.

Denoting $y=(y_1,y_2,\ldots,y_n)^{\T}$, call the joint density of $(y,\beta)$ given $\bX$ by $p(y,\beta \mid \bX)$. For a fixed $\alpha\in(0,1]$, define the fractional likelihood \citep{walker2001bayesian} by $p^{\alpha}(y \mid \bX,\beta) = \{p(y\mid \bX,\beta)\}^{\alpha}$ and denote with a slight abuse of notation, $p^{\alpha}(y,\beta \mid \bX) =p^{\alpha}(y\mid \bX,\beta)\pi(\beta)$, 
\begin{equation} \label{sec_TTA_eqn1}
p^{\alpha}(y, \beta \mid  \bX) \propto \exp{\Big[ \alpha\, y^{\T}\bX\beta - \alpha\, \sum^n_{i=1}\log\big(1+ e^{\bx_i^{\T}\beta}) -\frac{1}{2}{(\beta - \mu_{\beta})}^{\T} \Sigma^{-1}_{\beta} (\beta - \mu_{\beta}) \Big]}. 
\end{equation}
\cite{jaakkola2000bayesian} begins with the following quadratic duality result that holds for all $x \in \mb{R}$:
\begin{align*} 
& -\log\{1+\exp{(x)}\} = \max_{t \in \mb{R}} [A(t)x^2 - x/2 + C(t)], \\
&  A(t) = -{\tanh(t/2)}/{4t}, \quad C(t) = {t}/{2} - \log\{1+ \exp(t)\} + {t \tanh(t/2)}/{4}.
\end{align*}
We can then bound $\log p^{\alpha}(y,\beta \mid  \bX)$ from below by $\log {p}_{\l}^{\alpha}(y,\beta \mid \bX, \xi)$, where 
\begin{align}\label{sec_TTA_eqn2}
    \log{p}_{\l}^{\alpha}(y, \beta \mid \bX, \xi) = 
    &-\frac{1}{2}\beta^{\T}\left[\Sigma^{-1}_{\beta} - 2 \alpha \bX^{\T}\text{diag}\{A(\xi)\}\bX\right]\beta + \Big\{ {\alpha\Big(y-\frac{1}{2}\ind_n \Big)}^{\T}\bX + \mu^{\T}_{\beta}\Sigma^{-1}_{\beta}\Big\} \beta \nonumber \\ 
    &- \mu^{\T}_{\beta}\Sigma^{-1}_{\beta}\mu_{\beta}
    + \alpha \ind^{\T}_n C(\xi)  + \mbox{Const.}.
\end{align}
In the above display, $\xi=(\xi_1, \ldots, \xi_n)^{\T}$ collectively denotes all variational parameters, with $\xi_i$ appearing from applying the previous duality result for $-\log\{1 + \exp(\bx_i^\T \beta)\}$. 
Also, $\text{diag}\{A(\xi)\}$ is a $n \times n$ diagonal matrix with diagonal entries $\{A(\xi_1), A(\xi_2), \ldots, A(\xi_n)\}$ and $C(\xi) = \{C(\xi_1), \ldots, C(\xi_n)\}^{\T}$. 

Since ${p}_{\l}^{\alpha}(y,\beta \mid \bX, \xi)$ serves as a lower bound to $ p^{\alpha}(y,\beta \mid  \bX)$ for any $\xi \in \mb{R}^n$, similar to \cite{jaakkola2000bayesian} we use an empirical Bayes approach to estimate the variational parameters $\xi$ by maximizing ${p}_{\l}^{\alpha}(y \mid \bX, \xi) = \int {p}_{\l}^{\alpha}(y,\beta \mid \bX, \xi) d\beta$ with respect to $\xi$.
The true posterior distribution of $\beta$  in \eqref{sec_TTA_eqn1} is not available in closed form. However, assuming \eqref{sec_TTA_eqn2} to be a working (pseudo)-likelihood of $y,\beta$ given $\bX, \xi$, it is straightforward to see that the corresponding conditional posterior distribution of $\beta$  is  $\mbox{N}(\mu_{\alpha}(\xi), \Sigma_{\alpha}(\xi)/\alpha)$ where
\begin{eqnarray}\label{eq:muxisigxi}
\Sigma_{\alpha}^{-1}(\xi) = \Sigma^{-1}_{\beta}/\alpha - 2 \bX^{\T}\mbox{diag}\{A(\xi)\}\bX,\,\,\,\,  \mu_{\alpha}^{\T}(\xi)\Sigma_{\alpha}^{-1}(\xi) = \Big(y-  \frac{1}{2}\ind_n \Big)^{\T}\bX + \mu^{\T}_{\beta}\Sigma^{-1}_{\beta}/\alpha.
\end{eqnarray}
Treating $\beta$ as latent variables and augmenting with $y$ to get the complete data, one obtains the E-step,
\begin{align}\label{eq:EMQ}
    Q_{\alpha}(\xi^{t+1} \mid \xi^{t}) &= \mb{E}_{\beta \mid y,\xi^{t},\bX} \left[\log {p}_{\l}^{\alpha}(y, \beta \mid \xi^{t+1}, \bX) \right]\\
    &= \mbox{tr}\left[ \alpha \bX^{\T} \mbox{diag}\{A(\xi^{t+1}\}\bX\{ \Sigma_{\alpha}(\xi^{t})/\alpha + \mu_{\alpha}(\xi^{t})\mu_{\alpha}^{\T}(\xi^{t})\}\right] + \alpha \ind_n^{\T}C(\xi^{t+1})  + \mbox{Const.}, \nonumber 
\end{align}
where $\mbox{tr}(A)$ denotes the trace of a matrix $A$. 
Upon differentiating the above expression with respect to $\xi^{t+1}$ and using the fact that $C^{\prime}(x) = -x^2 A^{\prime}(x)$, we get the M-step,
 \begin{align}\label{sec_TTA_eqn3}
      (\xi^{t+1})^2 = \mbox{diag}[\bX\{ \Sigma_{\alpha}(\xi^{t})/\alpha + \mu_{\alpha}(\xi^{t})\mu_{\alpha}^{\T}(\xi^{t})\}\bX^{\T}]. 
 \end{align}
 The square operation in the above display is to be interpreted elementwise. We assume convergence when the increment in ${p}_{\l}^{\alpha}(y \mid \bX, \xi^t) $ is negligible which implies convergence of $\xi^t$ by virtue of EM algorithm. The EM sequence in \eqref{sec_TTA_eqn3} is recognized to be a fixed point iteration corresponding to the fixed point equation given by,
\begin{align}\label{eq:fp}
(\xi^{*})^2 = \mbox{diag}[\bX\{ \Sigma_{\alpha}(\xi^*)/{\alpha} + \mu_{\alpha}(\xi^*)\,\mu^{\T}_{\alpha}(\xi^*)\}\bX^{\T}].
\end{align}
Assuming \eqref{sec_TTA_eqn3} converges to a fixed point $\xi^*$, $\mu_{\alpha}(\xi^*)$ gives the variational estimate of $\beta$.
 
 \section{Statistical optimality of the variational estimate} \label{sec:statopt}
In this section we develop a rigorous framework to obtain frequentist risk bounds of the variational approximation obtained in \eqref{eq:muxisigxi} at any fixed point $\xi^*$ of \eqref{sec_TTA_eqn3}. 
 Throughout the section, 
we assume that the data is generated from a logistic regression model 
\begin{eqnarray}\label{eq:true}
p(y \mid \beta^*, {\bf X}) =  \exp{\Big[ y^{\T}\bX\beta^* - \sum^n_{i=1}\log\big(1+ e^{\bx_i^{\T}\beta^*})\Big]}.
\end{eqnarray}
It is not immediately clear whether the empirical likelihood based inference 
of $\xi$ as discussed in Section \ref{sec:tt} falls into the framework of variational inference in the sense of \eqref{eq:KL}.  In the following, we propose an objective function whose minimizer satisfies the fixed point iteration \eqref{sec_TTA_eqn3}.  Let our {\em working model} be 
\begin{eqnarray}\label{eq:wm}
  {p}_{\l}^{\alpha}(y \mid  \beta, \bX, \xi) &= 
   \exp{\Big\{\alpha \Big(y^{\T}\bX\beta  
       + \beta^{\T}\left[\bX^{\T}\text{diag}\{A(\xi)\}\bX\right]\beta - 0.5\ind_n^{\T}\bX + \ind^{\T}_n C(\xi)\Big)\Big\} }. 
       \end{eqnarray}
It is important to note here that ${p}_{\l}^{\alpha}(y \mid  \beta, \bX, \xi)$ is not a probability density, even when $\alpha=1$. Let $\mathcal{F}$ be the set of densities on $\mathbb{R}^p$.  Define a mapping from $\mathcal{L}: \mathcal{F} \times \mathbb{R}^n$ to $\mathbb{R}$  as
 \begin{eqnarray}\label{eq:ELBO}
\mathcal{L}(q, \xi) =  - \int \log \frac{{p}_{\l}^{\alpha}(y, \beta \mid \bX, \xi)}{q(\beta)} q(\beta) d\beta,
 \end{eqnarray}
where ${p}_{\l}^{\alpha}(y, \beta \mid \bX, \xi)$ is defined in \eqref{sec_TTA_eqn2}.  Observe that $\mathcal{L}(q, \xi)$ is the negative of the evidence lower bound obtained in a variational inference with \eqref{eq:wm} as the working likelihood, $\mbox{N}_p(\mu_{\beta}, \Sigma_{\beta})$ the prior on $\beta$, and variational family $\mathcal{F}  \times \{\delta_\xi: \xi \in \mathbb{R}^n\}$ where $\delta_\xi$ is the Dirac delta measure on $\xi \in \mathbb{R}^n$. In Lemma \ref{eq:varsoln}, we show that the tangent transform algorithm  maximizes $-\mathcal{L}(q, \xi)$. 
\begin{lemma}\label{eq:varsoln}
Any minimizer $(q^*, \xi^*)$ of \eqref{eq:ELBO} over  $\mathcal{F} \times \mathbb{R}^n$ satisfies 
\begin{eqnarray}\label{eq:fp1}
q^* =  \mbox{N}_p\{ \mu_{\alpha}(\xi^*), \Sigma_{\alpha}({\xi^*})/{\alpha}\}, \quad (\xi^{*})^2 = \mbox{diag}[\bX\{ \Sigma_{\alpha}(\xi^*)/{\alpha} + \mu_{\alpha}(\xi^*)\,\mu^{\T}_{\alpha}(\xi^*)\}\bX^{\T}], 
\end{eqnarray}
where $\mu_{\alpha}(\xi), \Sigma_{\alpha}({\xi})$ are defined in \eqref{eq:muxisigxi}. 
\end{lemma}

\begin{proof}
 We start the proof by re-writing  \eqref{eq:ELBO} as
 \begin{eqnarray}\label{eq:ELBO3}
  \mathcal{L}(q, \xi) = - \int q(\beta)  \log {p}_{\l}^{\alpha}(y, \beta \mid \bX, \xi) d\beta + \int q(\beta) \log \{q(\beta) \} d\beta. 
   \end{eqnarray}
 To minimize  \eqref{eq:ELBO3} jointly with respect to $(q, \xi)$, we set up the first order stationarity conditions. We first set the gradient of $\mathcal{L}(q, \xi)$ with respect to $\xi$ to zero holding $q$ fixed. As the second term in \eqref{eq:ELBO3} is independent of $ \xi$, this is equivalent to setting the gradient of $\mb{E}_{q} \big[\log {p}_{\l}^{\alpha}(y, \beta \mid \xi, \bX) \big]$ with respect to $\xi$ to be zero,
 \begin{eqnarray}\label{eq:stationary}
\frac{\partial}{\partial \xi} \mb{E}_{q} \left[\log {p}_{\l}^{\alpha}(y, \beta \mid \xi, \bX) \right] = 0. 
\end{eqnarray}
By an application of Fubini, \eqref{eq:stationary} is equivalent to
\begin{eqnarray}\label{eq:stationary2}
\mb{E}_{q} \left[\frac{\partial}{\partial \xi}  \log {p}_{\l}^{\alpha}(y, \beta \mid \xi, \bX) \right] = 0.
\end{eqnarray}
For fixed $\xi$, to maximize \eqref{eq:ELBO}, we simply apply Lemma \ref{thm:cond} in the Appendix. This leads to the optimal choice of $q(\beta)$ being the conditional distribution ${p}_{\l}^{\alpha}(\beta \mid y, \xi, \bX)$ which is $\mbox{N}_p(\mu_{\alpha}(\xi), \Sigma_{\alpha}({\xi})/{\alpha})$. 
This when combined with \eqref{eq:stationary2} yields
\begin{eqnarray}\label{eq:fo}
  \mb{E}_{\mbox{N}_p(\mu_{\alpha}(\xi), \Sigma_{\alpha}({\xi})/{\alpha})} \left[ \frac{\partial}{\partial \xi} \log{p}_{\l}^{\alpha}(y, \beta \mid \xi, \bX) \right] = 0.
\end{eqnarray}
To show that the solution of \eqref{eq:fo} satisfies \eqref{eq:fp}, recall that the first-order stationarity condition for maximizing  $Q_{\alpha}(\xi^{t+1} \mid \xi^{t})$ in \eqref{eq:EMQ} with respect to $\xi^{t+1}$ is given by 
\begin{eqnarray*}
\frac{\partial}{\partial \xi^{t+1}} Q_{\alpha}(\xi^{t+1} \mid \xi^{t}) = \mb{E}_{\beta \mid y,\xi^{t},\bX} \left[\frac{\partial}{\partial \xi^{t+1}}\log {p}_{\l}^{\alpha}(y, \beta \mid \xi^{t+1}, \bX) \right] = 0,
\end{eqnarray*}
which in turn is equivalent to solving the fixed point iteration $(\xi^{t+1})^2 = \mbox{diag}[\bX\{ \Sigma_{\alpha}(\xi^{t})/\alpha\, +\, \mu_{\alpha}(\xi^{t})\,\mu^{\T}_{\alpha}(\xi^{t})\}\bX^{\T}]$.  Thus the solution to \eqref{eq:fo} satisfies $(\xi^{*})^2 = \mbox{diag}[\bX\{ \Sigma_{\alpha}(\xi^*)/{\alpha} \,+\, \mu_{\alpha}(\xi^*)\,\mu^{\T}_{\alpha}(\xi^*)\}\bX^{\T}]$.
\end{proof}

Although \eqref{eq:ELBO} is reminiscent of the $\alpha$-variational objective function of \cite{yang2017alpha}, we note a couple of key differences : (a) ${p}_{\l}^{\alpha}(y \mid \beta, \xi, \bX)$ is not a valid probability density, but it is a lower bound to  $p^{\alpha}(y \mid \beta, \bX)$, (b) The latent variables $\xi$ lack a probabilistic interpretation as in \cite{yang2017alpha}, where one recovers the original likelihood after marginalization over the latent variables. Here, the latent variables instead correspond to tuning parameters appearing from convex duality. 

The usage of fractional likelihood for $\alpha \in (0, 1)$ results in only minor changes from a methodological and implementation perspective. However, from a theoretical perspective, like \cite{yang2017alpha}, $\alpha \in (0, 1)$ requires fewer assumptions to deliver optimal risk bounds.
 
\subsection{Variational Risk Bounds}
In the following, we develop risk bounds for the variational estimator separately for the case $\alpha \in (0, 1)$ and $\alpha =1$.  In the former case, to quantify the discrepancy between the variational estimate and the true parameter, we use an $\alpha$-R\'{e}nyi divergence
\begin{eqnarray}\label{eq:Renyi}
\mbox{D}_\alpha(\beta, \beta^*) = \frac{1}{n(\alpha -1)} \log \int \Bigg\{ \frac{p(y \mid \beta, \bX)}{p(y \mid \beta^*, \bX)}\Bigg\}^\alpha p(y \mid \beta^*, \bX) dy.
\end{eqnarray}
Refer to \cite{bhattacharya2019bayesian} for more on posterior risk bounds under the $\alpha$-R\'{e}nyi divergence. The factor $(1/n)$ is used to measure {\em average discrepancy} per observation. We can further simplify \eqref{eq:Renyi} to 
\begin{eqnarray*}
\mbox{D}_\alpha(\beta, \beta^*) = \frac{1}{n(\alpha -1)} \sum_{i=1}^n \log \Big[ p_{i,\beta}^\alpha p_{i,\beta^*}^{1- \alpha} + (1-p_{i,\beta})^\alpha (1- p_{i,\beta^*})^{1- \alpha}\Big],
\end{eqnarray*}
where $p_{i,\beta} = 1/\{1+ \exp(-\bx_i^{\T}\beta)\}$. 
The next theorem derives an upper bound  to the risk obtained by integrating the $\alpha$-R\'{e}nyi divergence with respect to the optimal variational solution. Denote by $\phi_p(x; \mu, \Sigma)$ the $p$-dimensional multivariate Gaussian density evaluated at $x \in \mathbb{R}^p$, with mean $\mu$ and variance covariance matrix $\Sigma$.  Let $\|\bX\|_{2, \infty} = \max\{\|\bx_i\|, i=1, \ldots, n\}$ and $ \| \bX\|_{\infty} := \max \{|x_{ij}|, i=1, \ldots, n, j=1, \ldots, p\}$.  Let  $L(\beta^*, \bX) = \max\{4\| \bX\|_{2, \infty}, 8 \| \bX\|_{2, \infty}^2 \|\beta^*\|_2 \}$.  

\begin{theorem}\label{thm:statopt}
For any $\varepsilon \in (0, 1)$, with probability $(1 - \varepsilon) -  1/ \{(D-1)^2\, n \, \varepsilon^2\}$ under \eqref{eq:true}
\begin{eqnarray*}
(1- \alpha) \int \mbox{D}_\alpha(\beta, \beta^*) \phi_p \big\{ \beta; \mu_{\alpha}(\xi^*), \Sigma_{\alpha}(\xi^*) \big\} d\beta &\leq&  D\alpha  \, \varepsilon^2 
+ \frac{p}{n} \log \Big\{\frac{L(\beta^*, \bX) }{\varepsilon^2} \Big\} +  \\
&{}& C_n(\beta^*, \mu_\beta, \Sigma_\beta) + \frac{1}{n} \log \Big( \frac{1}{\varepsilon} \Big) 
\end{eqnarray*}
 for some constant $D > 0$, where
\begin{eqnarray*}
C_n(\beta^*, \mu_\beta, \Sigma_\beta)= \frac{1}{2n}(\beta^*-\mu_\beta)^{\T}\Sigma_\beta^{-1} (\beta^*-\mu_\beta). 
\end{eqnarray*}

\end{theorem}
The proof of Theorem \ref{thm:statopt} can be found in \S \ref{ssec:statopt1} in the Appendix. 
\begin{remark}
Setting $\varepsilon = (p \log n /n)^{1/2}$, the risk bound for discrepancy $\mbox{D}_\alpha$ is $p/n$ upto logarithmic terms which is minimax optimal. The explicit bound is non-asymptotic and depends on prior parameters, the covariate matrix $\bX$ and the true data generating density.  
\end{remark}

Next, we separately deal with the case $\alpha =1$.  In doing so, we work with a limiting metric of  $\alpha$-R\'{e}nyi divergence as $\alpha$ tends to $1$.   Let  
$a(t) = \log (1+ e^t)$ and $a^{(1)}$ and $a^{(2)}$ denote the first and second derivatives. 
$a$ satisfies $a(t + h) \ge a(t) + h \, a^{(1)}(t) + \r(|h|) \, a^{(2)}(t)/2$ for all $t, h$,  where $r(h) = h^2/(\r_1 h + 1)$ for $\r_1 > 0$.  Define 
\begin{eqnarray*}
\mbox{D}(\beta^\ast, \beta)&:=& \frac{1}{n}\mb E_{\beta^*} \bigg\{ \log \frac{p(y \mid \beta^*, \bX)}{p(y \mid \beta, \bX)}\bigg\}  = \frac{1}{n}\sum_{i=1}^n \big\{ a(\bx_i^{\T}\beta) - 
a(\bx_i^{\T}\beta^*) - a^{(1)}(\bx_i^{\T}\beta^*) \bx_i^{\T}(\beta - \beta^*)\big\}. 
\end{eqnarray*}
$\mbox{D}(\beta^\ast, \beta)$ is the KL divergence between $p(\cdot \mid \beta^*, \bX)$ and $p(\cdot \mid \beta, \bX)$. Let $W = \mbox{diag}\{a^{(2)}(\bx_1^{\T}\beta^*), \ldots, a^{(2)}(\bx_n^{\T}\beta^*)\}$ and let $\kappa_1 = \lambda_p(\bX^{\T}\bX/n)$, $\kappa_2 = \lambda_1(\bX^{\T}W\bX/n)$, where $\lambda_j(A)$ denotes the $j$th largest  eigen value of a positive definite matrix $A$. 

\begin{theorem}\label{thm:statopt2}
Fix $\gamma, \varepsilon \in (0, 1)$ and $\epsilon = \kappa_1 \varepsilon/(2\r_1 \|\bX\|_\infty p^{1/2})$. If  ${n}^{1/2} \geq p (\log p)^{1/2}  \|\bX\|_\infty^2/\{\kappa_1\surd{2}\}$ and both $\kappa_1, \kappa_2 > 0$,  then with probability  $1- 2e^{- n \gamma \epsilon/8} - e^{-n \kappa_2 \varepsilon^2/2} - 1/ \{(D-1)^2\, n \, \varepsilon^2\}$ 
 under \eqref{eq:true}, 
\begin{eqnarray*}
\gamma \int \mbox{D}(\beta^\ast, \beta) \phi_p \big\{ \beta; \mu(\xi^*), \Sigma(\xi^*) \big\} d\beta \leq  (D + \kappa_2/2) \, \varepsilon^2 
+  \frac{p}{n} \log \Big\{\frac{L(\beta^*, \bX) }{\varepsilon^2} \Big\}  + C_n(\beta^*, \mu_\beta, \Sigma_\beta).  
\end{eqnarray*}
\end{theorem}
The proof of Theorem \ref{thm:statopt2} can be found in \S \ref{ssec:statopt2} in the Appendix. 

\begin{remark}
To obtain a risk bound, we keep $\gamma$ to be a fixed number in $(0, 1)$ and 
set $\varepsilon = (p \log n /n)^{1/2}$. Then KL divergence risk is $p/n$ upto logarithmic terms which is again minimax optimal. As opposed to Theorem \ref{thm:statopt},    Theorem \ref{thm:statopt2} requires the eigen values of $\bX^{\T}W \bX/n$ to be bounded from below and the eigen values of $\bX^{\T} \bX/n$  to be bounded from above. 
\end{remark}

Now we conduct a numerical study to empirically support the conclusions of the theorems above. For fixed $(n,p)$
 we construct a $n\times p$ design matrix $\bX$ where $\bx^\T_i\ (i=1,\ldots, n)$ are independently drawn from  $\mbox{N}_p\big( 0,(0.5\,\mb{I}_p + 0.5\,\ind_p\ind_p^{\T}) \big)$. We then normalize each row of $\bX$ by $\surd{p}$. We fix $\beta^*$ to be $\set{-4,4,4,-4}$
and generate $y_i \sim \mbox{Bernoulli}(p_i)$ with $p_i = 1/ \{1 + \exp(-\bx^\T_i\beta^*)\}$, independently for $i \in \set{1,2,\ldots,n}$. We place a zero-mean Gaussian prior $\beta\sim \mbox{N}_p(0_p,\Sigma_{\beta})$ and set $\Sigma_{\beta} = 5^2\mb{I}_p$. 
Given a dataset $(y,\bX)$ and fixed $\alpha \in \{0.50, 0.65, 0.80, 0.95, 1.00\}$ we calculate the fixed point solution $\xi^*$ using \eqref{sec_TTA_eqn3}. We use $p(y\mid \mu(\xi^*),\bX)$ as the final estimated density and calculate the discrepency $\mbox{D}_{\alpha}(\mu(\xi^*),\beta^*)$ with $p(y\mid \beta^*,\bX)$. In panel (a) we plot $\mbox{D}_{\alpha}(\mu(\xi^*),\beta^*)$ for $\alpha \in \{0.50, 0.65, 0.80, 0.95\}$ along with  $\mbox{D}(\beta^*,\mu(\xi^*))$ that corresponds to $\alpha=1$. In panel (b) we plot the $\ell_2$ norm between $\mu(\xi^*)$ and $\beta^*$. We repeat this process for $500$ independent samples with $(n=100, p=4)$ and $(n=200, p=4)$. 
Clearly, increasing the sample size leads to improved estimation as seen from either panel of Figure \ref{opt_fig1}. 
Also, $\mbox{D}_{\alpha}(\mu(\xi^*),\beta^*)$ slightly increases as $\alpha$ increases to $1$. Since $\alpha$-R\'{e}nyi divergence counterbalances the effect of the misspecified likelihood and reinforces concentration around the truth, this behavior is expected.
\begin{figure}[htbp!]
    \centering
\begin{subfigure}{0.50\textwidth}
  \centering
  \includegraphics[scale= 0.225]{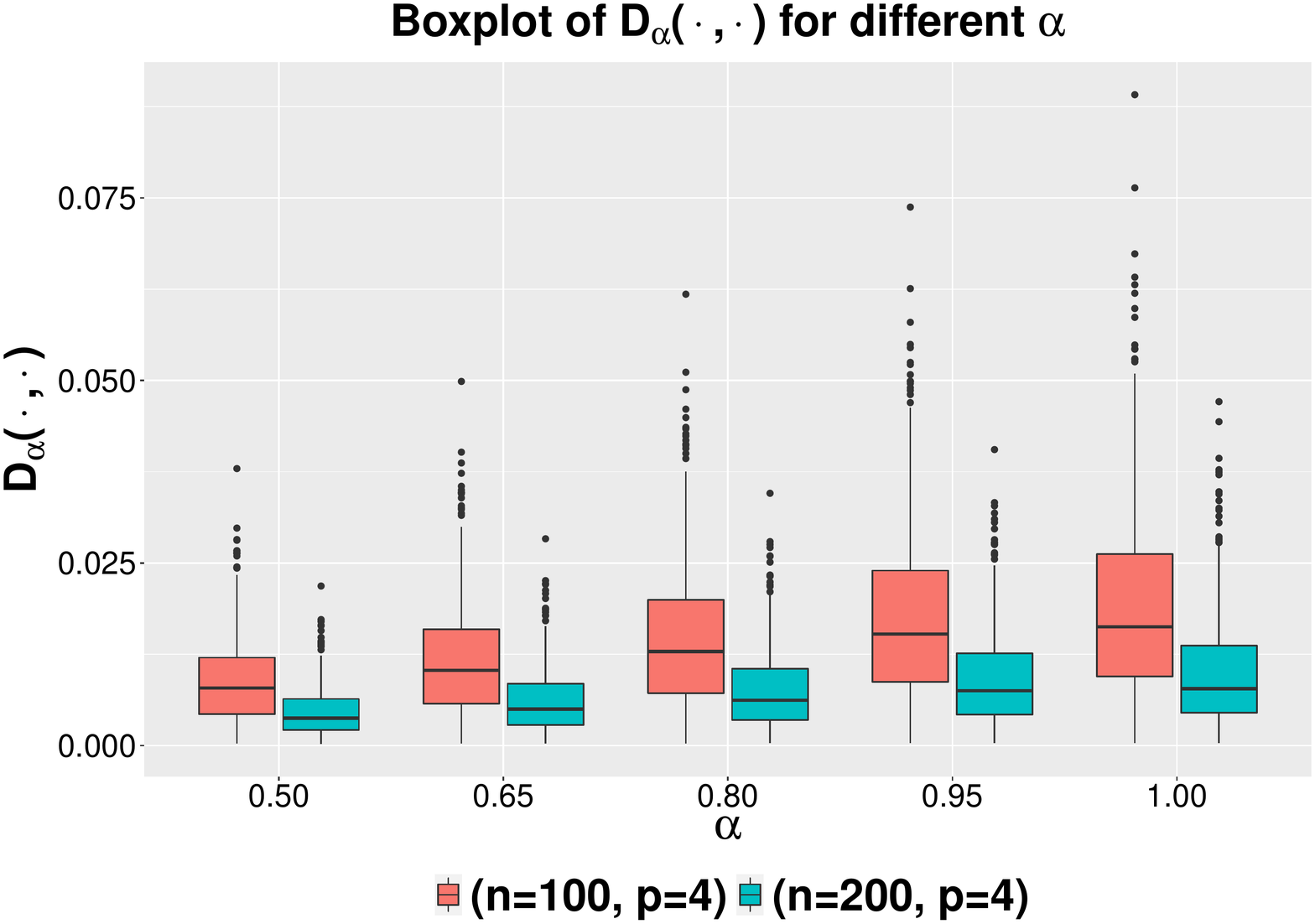}
  \caption{}
  \label{opt_fig1_a}
\end{subfigure}%
\begin{subfigure}{0.50\textwidth}
  \centering
  \includegraphics[scale= 0.225]{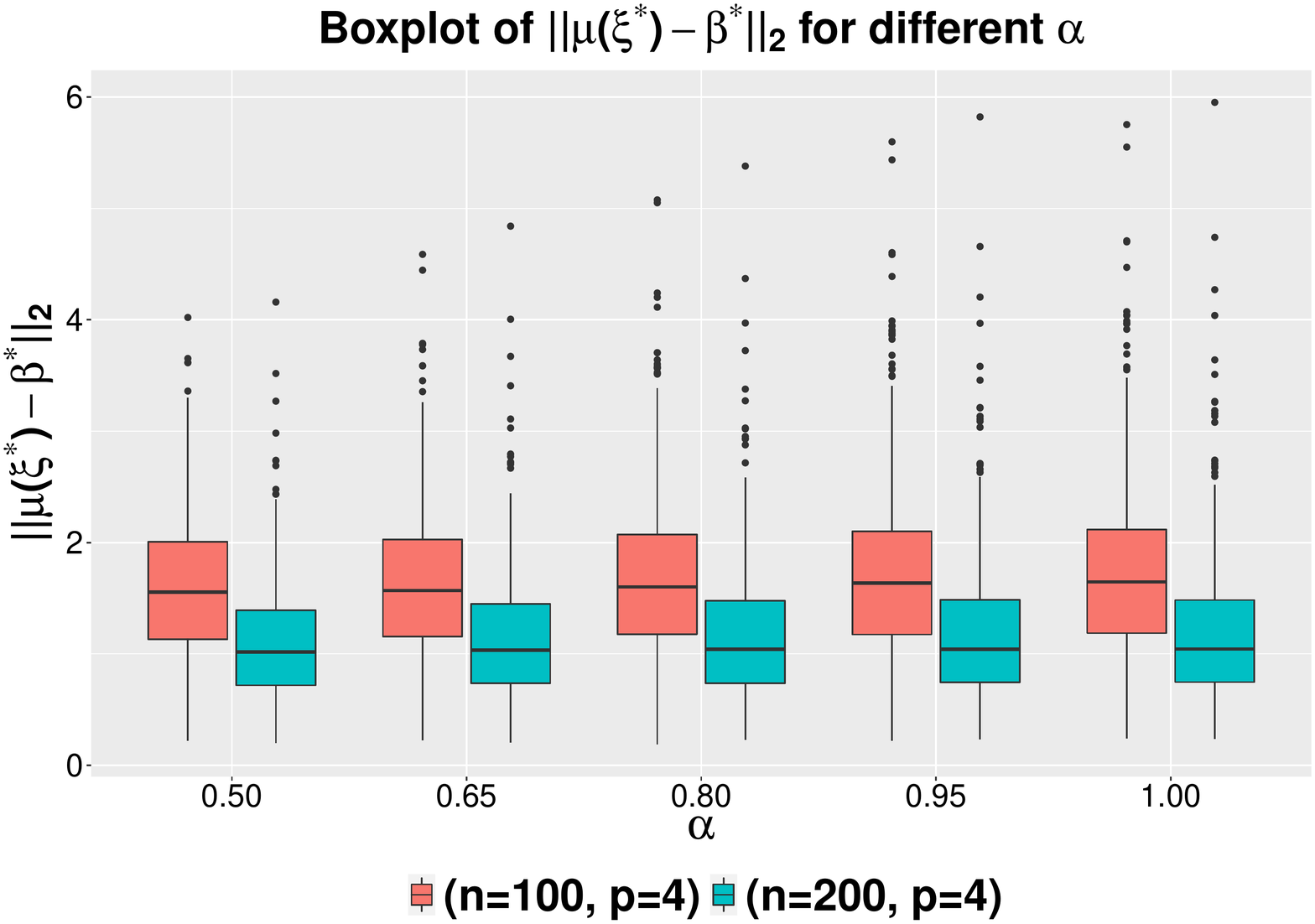}
  \caption{}
  \label{opt_fig1_b}
\end{subfigure}%
\caption{(a) Boxplot of $\mbox{D}_{\alpha}(\beta,\beta^*)$ for $\alpha\in (0,1)$ and $\mbox{D}(\beta^*,\beta)$ for $\alpha=1$ (b) Boxplot of $\|\mu(\xi^*)- \beta^*\|_2$ for different values of $\alpha \in (0,1]$}
\label{opt_fig1}
\end{figure}

In Figure \ref{opt_fig2} we show the contour plots of the marginal of $(\beta_2, \beta_4)$ obtained from the variational approximation 
$q^* =  \mbox{N}_p\{\mu(\xi^*), \Sigma({\xi^*})\}$   for a given dataset $(y,\bX)$. The upper and lower panels correspond to $(n=100 ,p=4)$ and $(n=200, p=4)$ respectively with different $\alpha\in\{0.80, 0.95, 1.00\}$. Clearly the concentration of the approximate posterior increases as $\alpha$ tends to 1. Also concentration increases with increase in the sample size. Further, the variational approximations appear to be almost similar for $\alpha = 0.95$ and $\alpha = 1$. 

\begin{figure}[htbp!]
    \centering
\begin{subfigure}{0.33\textwidth}
  \centering
  \includegraphics[scale= 0.155]{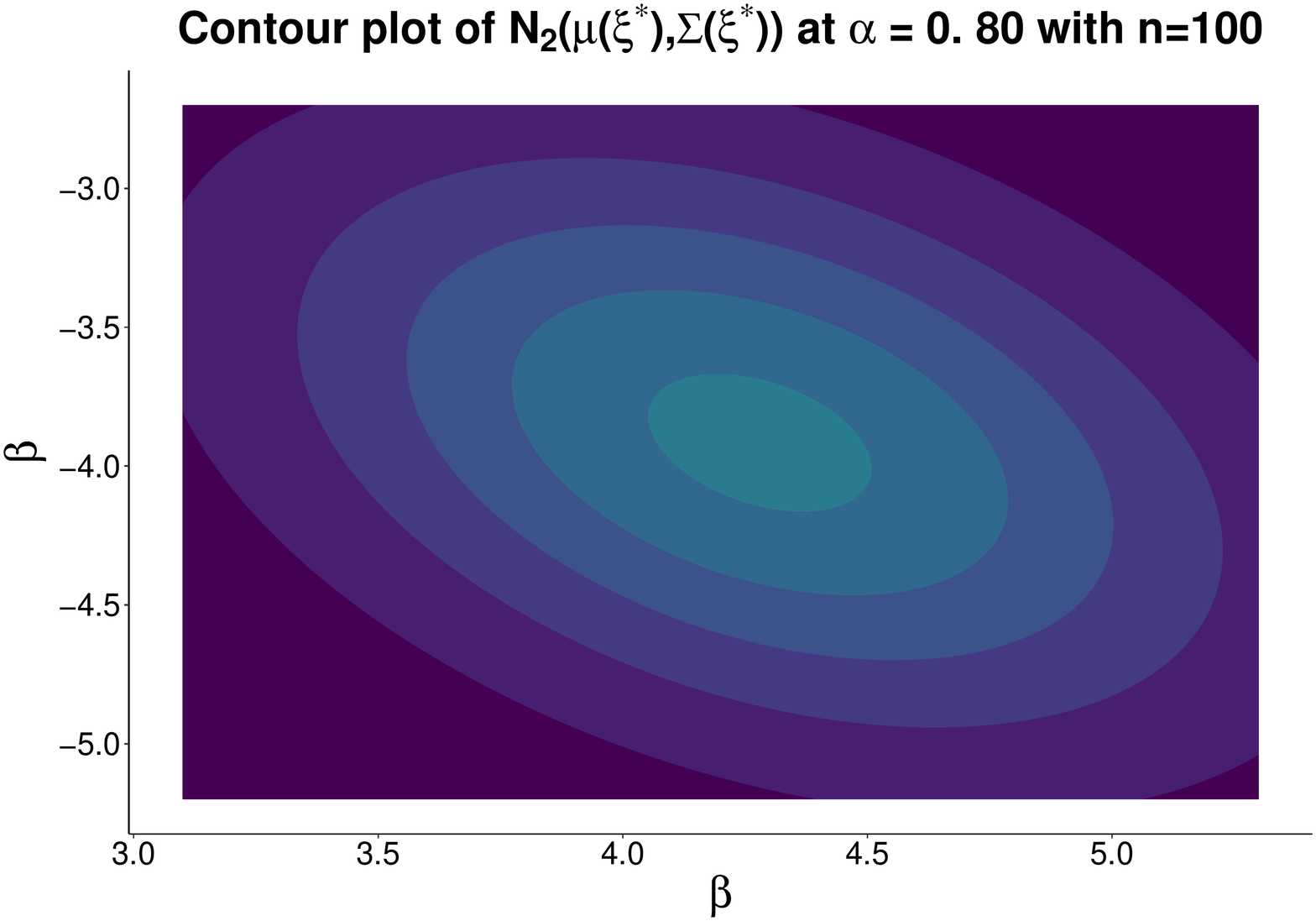}
  \caption{}
  \label{opt_fig2_a}
\end{subfigure}%
\begin{subfigure}{0.33\textwidth}
  \centering
  \includegraphics[scale= 0.155]{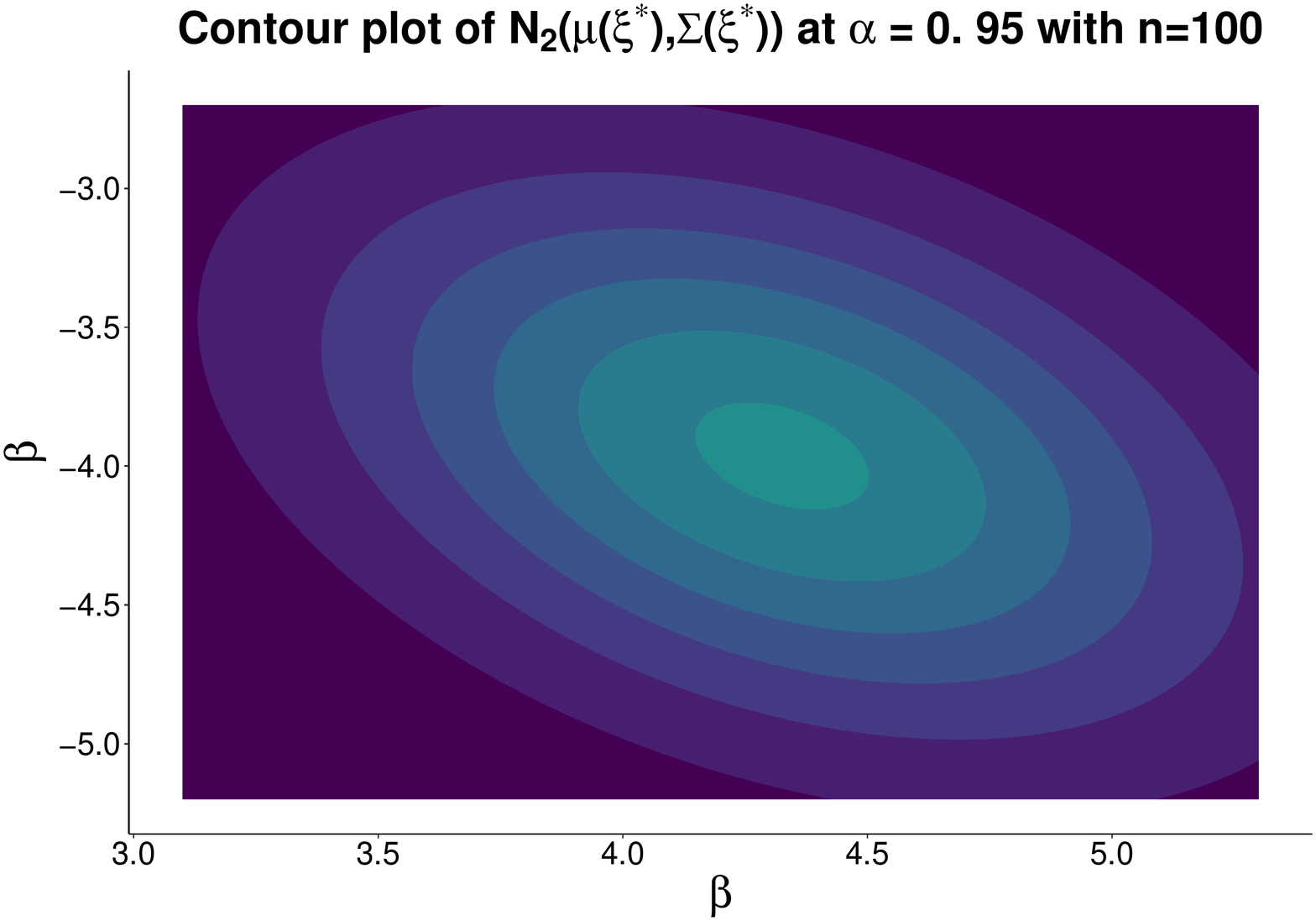}
  \caption{}
  \label{opt_fig2_b}
\end{subfigure}%
\begin{subfigure}{0.33\textwidth}
  \centering
  \includegraphics[scale= 0.155]{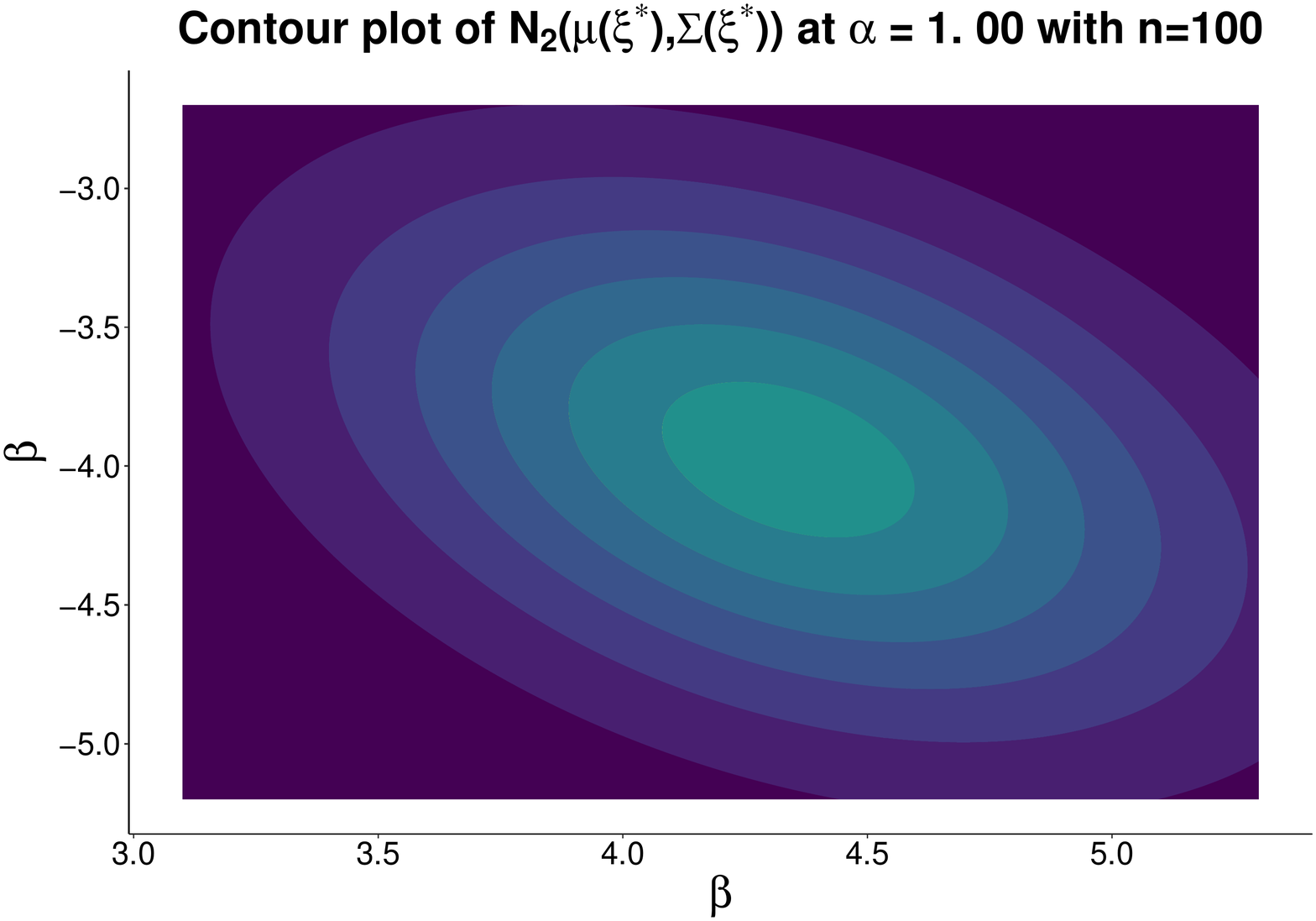}
  \caption{}
  \label{opt_fig2_c}
\end{subfigure}%
\\
\begin{subfigure}{0.33\textwidth}
  \centering
  \includegraphics[scale= 0.155]{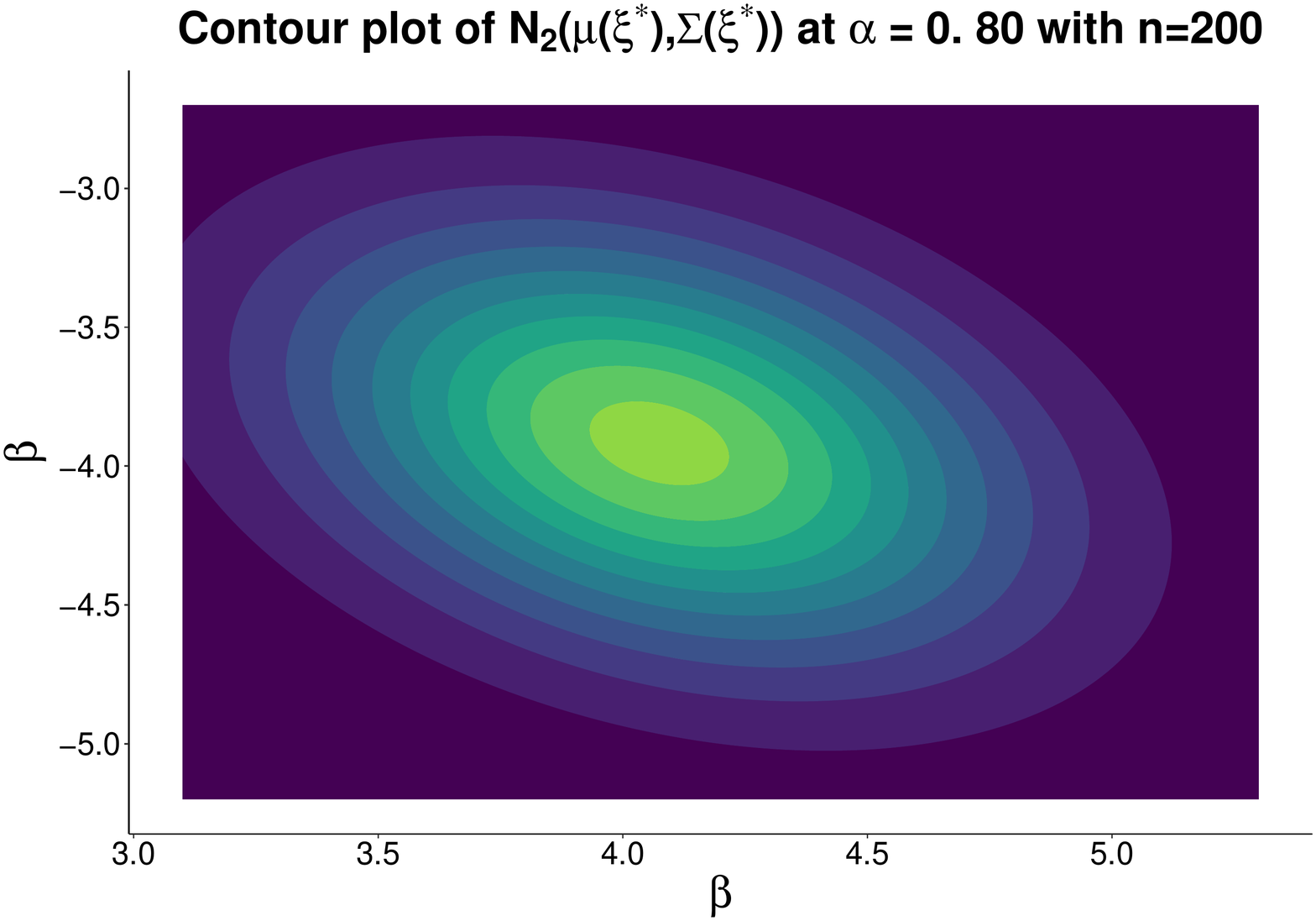}
  \caption{}
  \label{opt_fig2_d}
\end{subfigure}%
\begin{subfigure}{0.33\textwidth}
  \centering
  \includegraphics[scale= 0.155]{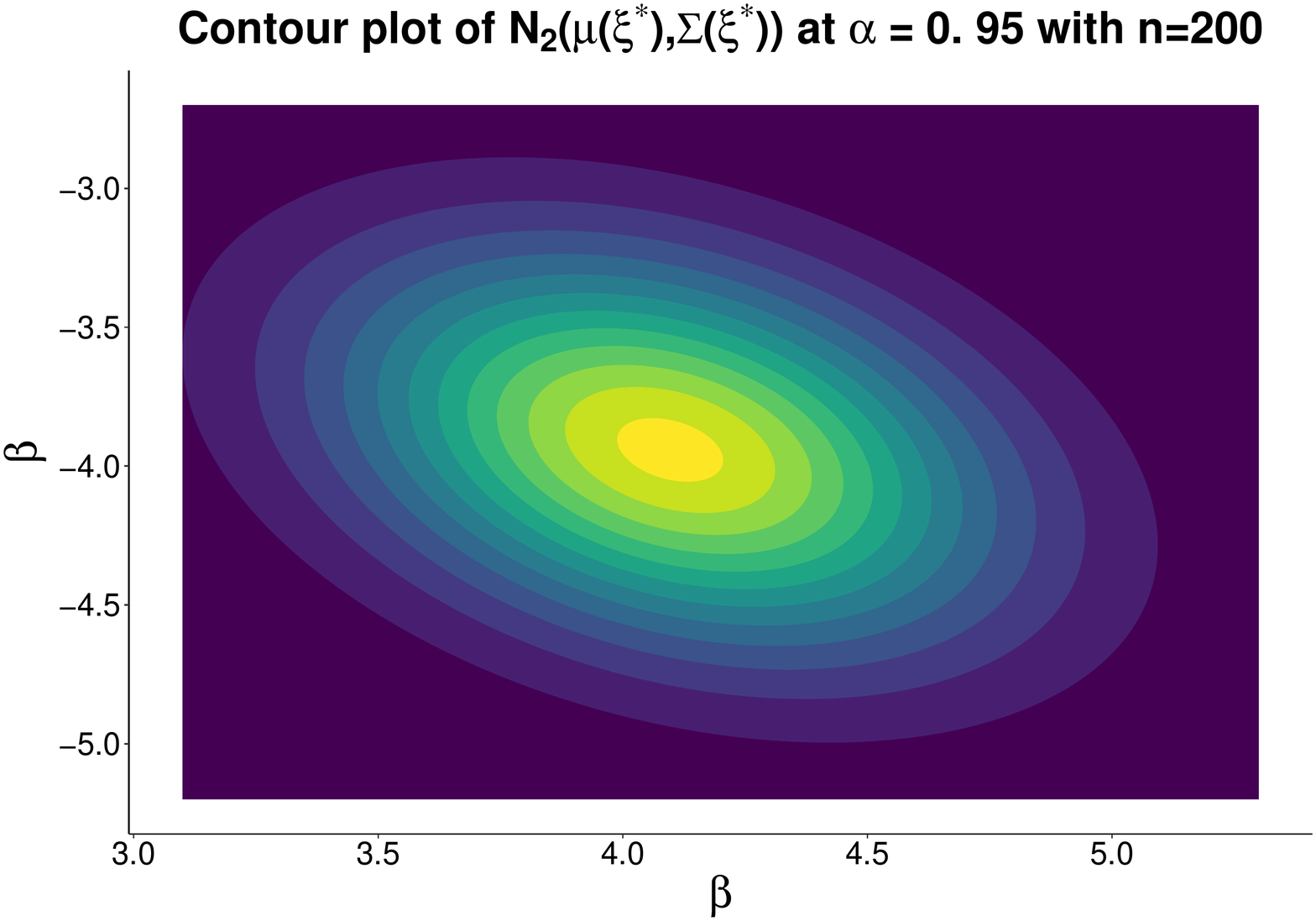}
  \caption{}
  \label{opt_fig2_e}
\end{subfigure}%
\begin{subfigure}{0.33\textwidth}
  \centering
  \includegraphics[scale= 0.155]{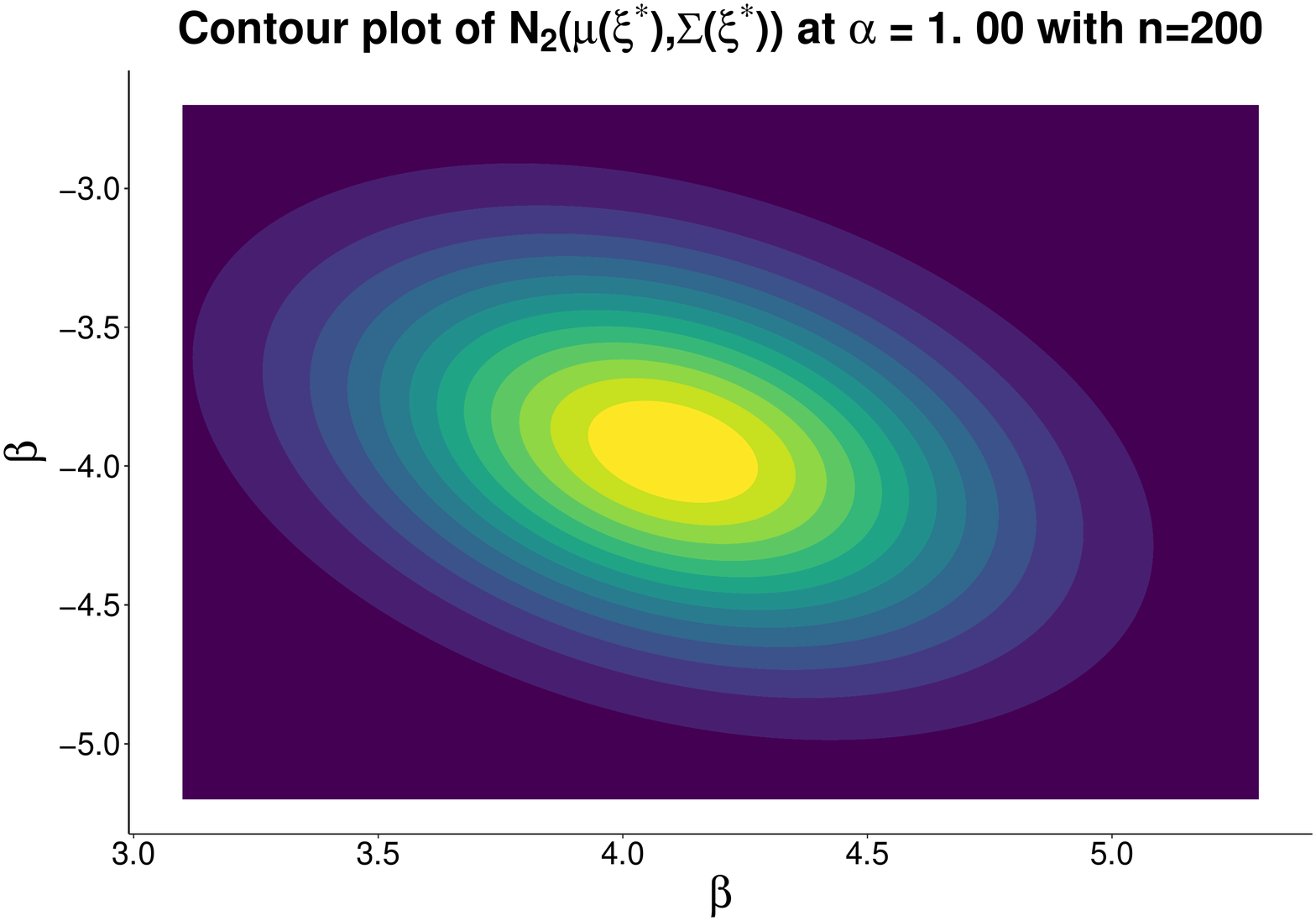}
  \caption{}
  \label{opt_fig2_f}
\end{subfigure}%
\caption{ Contour plot of posterior distribution of $(\beta_2,\beta_4)$. The upper row ((a)--(c)) corresponds to $(n=100, p=4)$ and the lower row ((d)--(f)) corresponds to $(n=200, p=4)$.}
\label{opt_fig2} 
\end{figure}

\section{Stability  and convergence of tangent transform algorithm} \label{sec:algoconv}
\subsection{Preliminaries}
We provide a brief review of stability of dynamical systems here; a more detailed review and relevant references can be found in \S \ref{sec:dyn} of the Appendix. Consider the following discrete-time autonomous system,
\begin{align}\label{sec_NOS_eq11}
    \psi^{t+1} = f(\psi^t), \quad t\in \mb{N}, 
\end{align}
where $f : \mb{R}^n \to \mb{R}^n$ (or, $f : \mb{D} \to \mb{R}^n, \mb{D} \subseteq \mb{R}^n$) is a twice continuously differentiable function. Any $\psi^* \in \mb{R}^n$ satisfying $\psi^* = f(\psi^*)$ is called a fixed point for this system.  A fixed point $\psi^*$ of \eqref{sec_NOS_eq11} is called {\em locally asymptotically stable} if given any $\epsilon > 0$,   there exists $\delta = \delta(\epsilon)$ such that whenever $\|\psi^0 - \psi^*\| < \delta$, we have $\|f(\psi^t) - \psi^*\| < \epsilon$ for all $t$ and $\lim_{t \to \infty} \|\psi^t - \psi^*\| = 0$. 

The following well-known result is instrumental to show that a fixed point is locally asymptotically stable. Denote by $\rho(\mathbf{J})$ the spectral radius of a square matrix $\mathbf{J}$, the largest eigenvalue of $\mathbf{J}$ in absolute value.
\begin{lemma}\label{sec_NOS_lemma21}
Let $\psi^*$ be a fixed point solution to the discrete-time autonomous system given by $\psi_{t+1} = f(\psi_t)$. Suppose, $f:\mb{D} \to \mb{R}^n (\mb{D} \subseteq \mb{R}^n)$ is a twice continuously differentiable function around a neighborhood $\mb{D}$ of $\psi^*$. Let $\mathbf{J} = [\partial_i f(\psi)/\partial \psi_j]_{\psi=\psi^*}$ be the Jacobian matrix of $f$ evaluated at $\psi^*$. Then, $\psi^*$ is locally asymptotically stable if $\rho(\mathbf{J})$ is less than $1$.  
\end{lemma}

\subsection{Asymptotic stability of tangent transform EM}\label{sec:stabilityTT}
In this subsection, we study the EM sequence of $\xi$ from equation \eqref{sec_TTA_eqn3} viewed as a discrete time dynamical system in $\xi^2$. As noted above, the convergence and stability aspects of the system depends crucially on the properties of the Jacobian of the map. Since the function $A(\xi)=  -{\tanh(\xi/2)}/{4\xi}$ is symmetric around $0$, and $\Sigma_{\alpha}(\xi)$ and $\mu_{\alpha}(\xi)$ are dependent on $\xi$ through $A(\cdot)$, only the magnitude of $\xi$ is relevant and hence we will discuss the nature of EM iterates on $\mb{R}^+$. The properties of the function $A(\cdot)$ play a crucial role in such an analysis. In Proposition \ref{sec_TTA_prop1} in the Appendix, we in particular show that the function $A: \mb{R}^+ \to \mb{R}^-$ is monotonically increasing and twice continuously differentiable with $A(0)=-1/8$ and $A(\xi) + \xi A^{\prime}(\xi) < 0$ for all $\xi\in \mb{R}^+$. 

In Theorem \ref{main_VB_1} below, we show that the EM sequence in equation \eqref{sec_TTA_eqn3} is locally asymptotically stable. 
\begin{theorem}\label{main_VB_1}
Suppose the design matrix $\bX$ does not have any row equal to the zero vector. For any $\alpha \in (0, 1]$ and positive definite $\Sigma_\beta$, any fixed point solution 
$\xi^*$ of the EM sequence in \eqref{sec_TTA_eqn3}
is locally asymptotically stable. 
\end{theorem}
In light of  Lemma \ref{sec_NOS_lemma21}, one needs to check the spectral radius of the Jacobian of the system at the fixed point to prove Theorem \ref{main_VB_1}. We present an outline of the proof here; refer to \S \ref{appendix_thm2_proof} in the Appendix for a complete proof. Given positive semi-definite matrices $A, B$ of the same dimension, we follow the usual convention to denote $B \prec A$ (resp. $B \precsim A$) to mean $(A-B)$ is positive definite (resp. positive semi-definite). 

Since $\Sigma_{\alpha}(\xi^*) = [\Sigma^{-1}_{\beta}/{\alpha} - 2 \bX^\T\mbox{diag}\{A(\xi^*)\}\bX]^{-1}$ is positive definite and 
$\bx^\T_i \neq 0\ (i = 1,2,\ldots,n)$, one can conclude $\xi_i^*>0$ from \eqref{eq:fp}.
Next, we show that the Jacobian matrix evaluated at the fixed point $\xi^*$ can be analytically expressed as 
\begin{eqnarray}\label{eq:Jacobian}
\mathbf{J}_{\alpha} =\left[\bX\Sigma_{\alpha}\left({\xi^*}\right)\bX^{\T} \circ \bX\left\{ \Sigma_{\alpha}\left(\xi^*\right)/{\alpha} + 2\mu_{\alpha}\left(\xi^*\right)\mu^{\T}_{\alpha}\left(\xi^*\right)\right\}\bX^{\T}\right] \mathrm{D},
\end{eqnarray}
where $\circ$ denotes the Hadamard (or, elementwise) product, $\mu_{\alpha}(\xi), \Sigma_{\alpha}({\xi})$ are defined in \eqref{eq:muxisigxi}, and $\mathrm{D} = \text{diag}\left\{{A^{\prime}\left(\xi^*\right)}/{\xi^*}\right\}$, with the $\slash$ operation  interpreted elementwise. By similarity, $\mathbf{J}_{\alpha}$ and
\be 
\tilde{\mathbf{J}}_{\alpha} = \D^{1/2}\left[\bX \Sigma_{\alpha}(\xi^*)\bX^{\T} \circ \bX\left\{ \Sigma_{\alpha}(\xi^*)/{\alpha} + 2\mu_{\alpha}(\xi^*)\,\mu^{\T}_{\alpha}(\xi^*)\right\}\bX^{\T}\right]\D^{1/2},
\ee 
have the same set of eigenvalues. Clearly, $\tilde{\mathbf{J}}_{\alpha}$ is real symmetric and positive semi-definite by the Schur product theorem. Therefore, $\tilde{\mathbf{J}}_{\alpha}$, and hence $\mathbf{J}_{\alpha}$, have non-negative eigenvalues. Hence, the spectral radius $\rho(\mathbf{J}_{\alpha})$ is simply the largest eigenvalue of $\mathbf{J}_{\alpha}$, which we proceed to bound next. 

Using the fact that, $\D^{1/2}[\bX \Sigma_{\alpha}(\xi^*)\bX^{\T}\circ \bX \Sigma_{\alpha}(\xi^*)\bX^{\T}]\D^{1/2}/{\alpha}$ is a positive semi-definite matrix, we have,
 \begin{align}\label{sec_MT_proof_1}
     \tilde{\mathbf{J}}_{\alpha} \precsim 2\,\D^{1/2}\left[\bX \Sigma_{\alpha}(\xi^*)\bX^{\T} \circ \bX\left\{ \Sigma_{\alpha}(\xi^*)/{\alpha} + \mu_{\alpha}(\xi^*)\,\mu_{\alpha}^{\T}(\xi^*)\right\}\bX^{\T}\right]\D^{1/2},
 \end{align}
Denote  $\Lambda_{\alpha}  =  \Sigma_{\alpha}(\xi^*)/{\alpha} + \mu_{\alpha}(\xi^*)\,\mu^\T_{\alpha}(\xi^*)$ and $\bX\Lambda_{\alpha} \bX^{\T} = \Delta(\xi^*)\circ \Gamma_{\alpha}$ where $[\Delta(\xi^*)]_{ij} = \xi^*_i\xi^*_j$, $[\Gamma_{\alpha}]_{ij} = \bx^\T_i \Lambda_{\alpha}  \bx_j/(\xi^*_i\xi^*_j)$. Then the matrix on the right hand side of the \eqref{sec_MT_proof_1} can be written as, 
$$2\, \D^{1/2}\set{\bX \Sigma_{\alpha}(\xi^*)\bX^{\T} \circ \Delta(\xi^*)} \D^{1/2}\circ \Gamma_{\alpha}.$$
A result from \cite{horn1994topics}, stated in Lemma \ref{sec_ev_3} in the Appendix, provides bounds on the largest eigenvalues of $\mbox{M} \circ \mbox{N}$ as a product of the largest eigenvalue of $\mbox{M}$ and largest diagonal of the $\mbox{N}$. The diagonals of $\Gamma_{\alpha}$ are $1$ and the largest eigenvalue of $2\, \D^{1/2}\set{\bX \Sigma_{\alpha}(\xi^*)\bX^{\T} \circ \Delta(\xi^*)} \D^{1/2} = 2\, \mbox{diag}[\{\xi^*\, A^{\prime}(\xi^*)\}^{1/2}] \bX \Sigma_{\alpha}(\xi^*)\bX^{\T} \mbox{diag}[\{\xi^*\, A^{\prime}(\xi^*)\}^{1/2}]$ is the same as that of $2\,\Sigma_{\alpha}(\xi^*)\bX^{\T} \mbox{diag}\{\xi^*\, A^{\prime}(\xi^*)\}\bX$. Since $A(x) + x A^{\prime}(x) < 0$ for all $x\in\mb{R}$, $ 2\,\bX^{\T} \mbox{diag}\{\xi^*\, A^{\prime}(\xi^*)\}\bX \prec  \Sigma^{-1}_{\alpha}(\xi^*)$. Lemma \ref{sec_ev_4} shows that the largest eigenvalue of $\mbox{M}^{-1/2}\,\mbox{N}\,\mbox{M}^{-1/2}$ is strictly less than 1 where $\mbox{N}\prec \mbox{M}$ and $\mbox{M}, \mbox{N}$ are positive definite and positive semi-definite matrices respectively. This delivers the proof that $\rho(\mathbf{J}_{\alpha})< 1$.

In the special case when $p = 1$, we can make substantial simplifications and show that (see \S \ref{appendix_specialcase} in the Appendix for details),
\begin{align}\label{sec_mainthm_1}
    \rho(\mathbf{J}_{\alpha}) &= \frac{ 2 \sum^n_i x^2_i A^{\prime}(\xi^*_i) \xi^*_i}{\{\sigma_{\beta}^{-2}/\alpha- \sum^n_{i=1} 2x^2_i A(\xi^*_i)\}} 
 - \frac{ \sum^n_i x^4_i A^{\prime}(\xi^*_i)/ \xi^*_i}{\alpha\,\{\sigma_{\beta}^{-2}/\alpha - \sum^n_{i=1} 2x^2_i A(\xi^*_i)\}^2},\\
 &<\frac{ 2 \sum^n_{i=1} x^2_i A^{\prime}(\xi^*_i) \xi^*_i}{\sigma_{\beta}^{-2}/\alpha - \sum^n_{i=1} 2x^2_i A(\xi^*_i)}< 1, \nonumber
\end{align}
where the first inequality follows from the fact that the second term in \eqref{sec_mainthm_1} is positive as $A^{\prime}(x)/x >0$ for all $x \in \mb{R}$. The second inequality follows from the fact that $A(x) + xA^{\prime}(x)<0$ for all  $x\in \mb{R}$.  

It is important to note that Theorem \ref{main_VB_1} places minimal restriction on the design matrix $\bX$ and its conclusion remains true for any $p$ and $n$. We conduct a replicated numerical study to empirically demonstrate some of these features. 
We use the same simulation design corresponding to Figure \ref{opt_fig1} except now for a fixed $(n, p)$, we provide a sufficiently flat prior $\Sigma_{\beta} = 10^2 \mb{I}_p$ while fixing the first $p/2$ (resp.  $[p/2]+1$) entries of $\beta^*$ to be $-4$, and the remaining $p/2$ (resp. $[p/2]$) to be $4$ when $p$ is even (resp. odd). To remain faithful to the assumptions of Theorem \ref{main_VB_1}, we do not normalize $\bX$ with $\surd{p}$. 
We compute the spectral radius $\rho \equiv \rho(\mathbf{J}_{\alpha})$ of the Jacobian matrix $\mathbf{J}_{\alpha}$ for $\alpha\in\{0.50,1.00\}$ at the fixed point $\xi^*$ for different values of $(n,p)$ over 500 independent replicates, with summary boxplots shown in Figure \ref{numeric_fig1}. In panel (a), we fix $n = 150$ and vary $p \in \{2, 5, 10, 20\}$.
 In panel (b), we fix $p$ at $15$ and vary $n \in \set{5, 10, 50, 100}$. It is evident that $\rho$ remains less than 1 for all combinations of $(n,p)$. 
Observe also that the first two cases in panel (b) correspond to $p > n$, and as predicted by the theory, the spectral radius continues to be smaller than $1$. It can be seen from either panel that on an average $\rho$ at $\alpha = 0.5$ is higher than the corresponding value at $\alpha = 1$.

\begin{figure}[htbp!]
    \centering
\begin{subfigure}{0.50\textwidth}
  \centering
  \includegraphics[scale= 0.225]{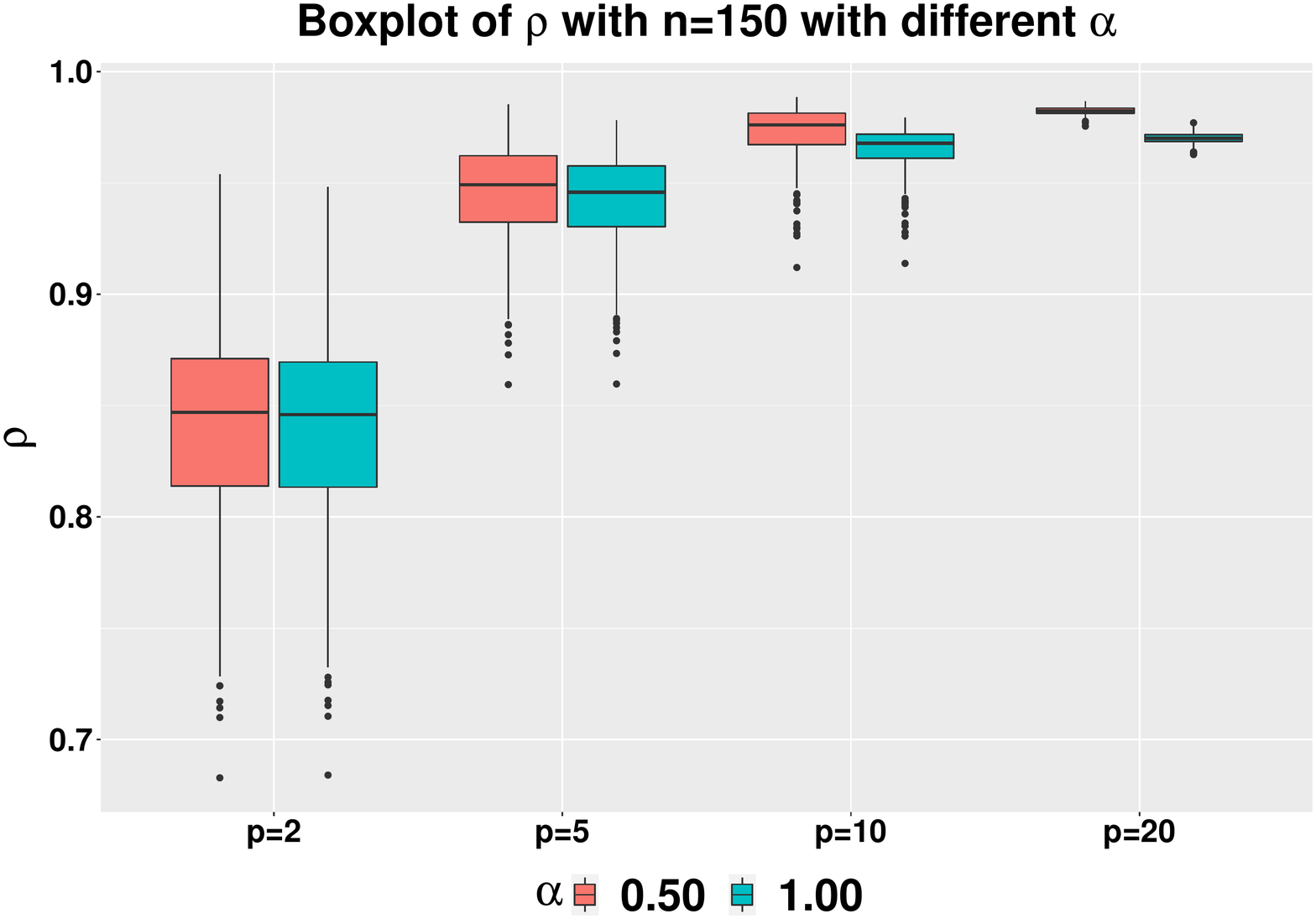}
  \caption{}
  \label{numerical_fig1_a}
\end{subfigure}%
\begin{subfigure}{0.50\textwidth}
  \centering
  \includegraphics[scale= 0.225]{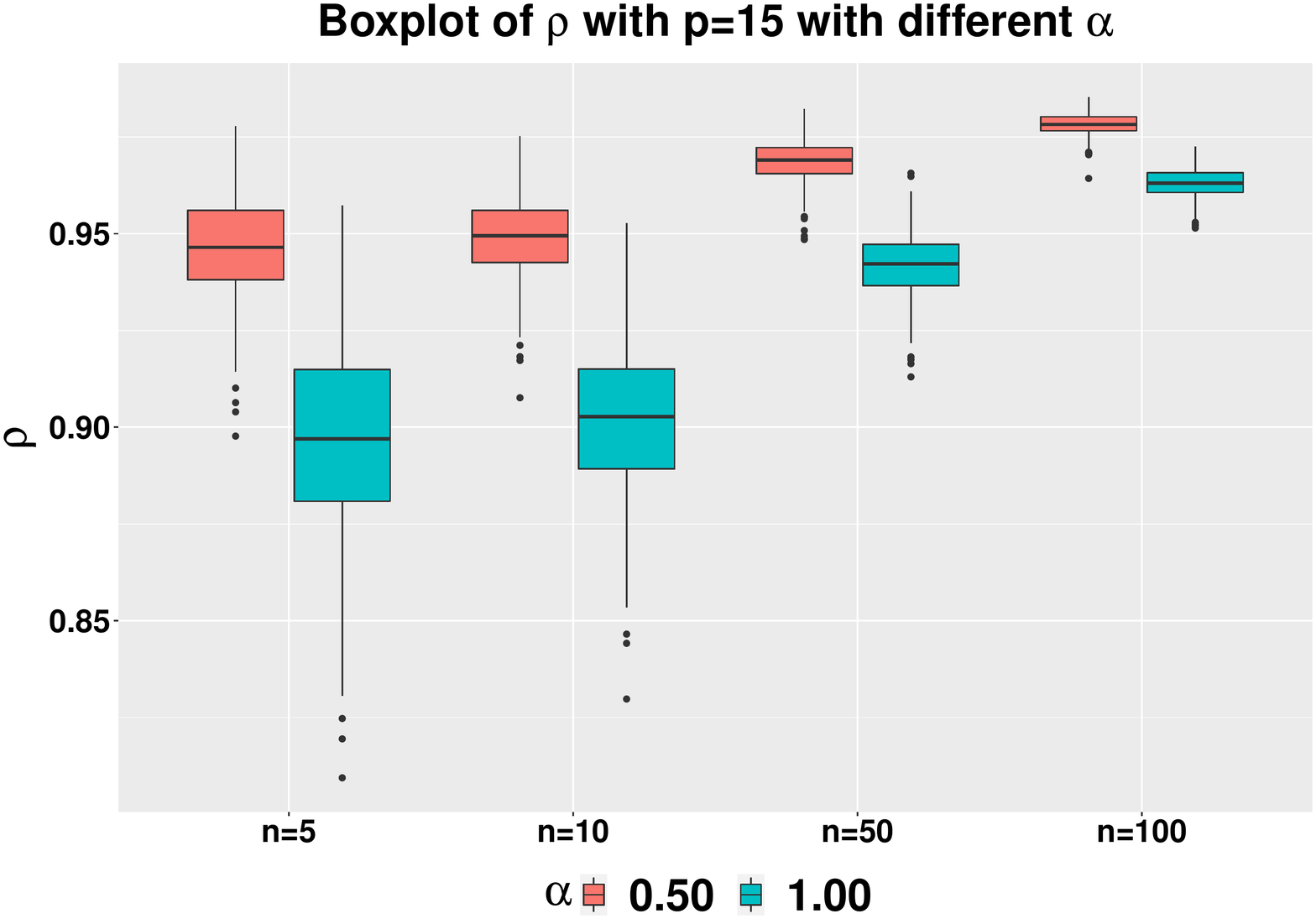}
  \caption{}
  \label{numerical_fig1_b}
\end{subfigure}%
\caption{(a) Boxplot of $\rho$ with $n$ fixed and varying $p$ (b) Boxplot of $\rho$ with $p$ fixed and different values of $n$. Both the plots are produced with 500 replications using the same data $(y,\bX)$. It can be clearly seen that, for both the $\alpha\in\{0.50,1.00\}$ the spectral radius is strictly less than 1 irrespective of $n$ and $p$.}
\label{numeric_fig1}
\end{figure}

\begin{figure}[htbp!]
    \centering
\begin{subfigure}{0.50\textwidth}
  \centering
  \includegraphics[scale= 0.225]{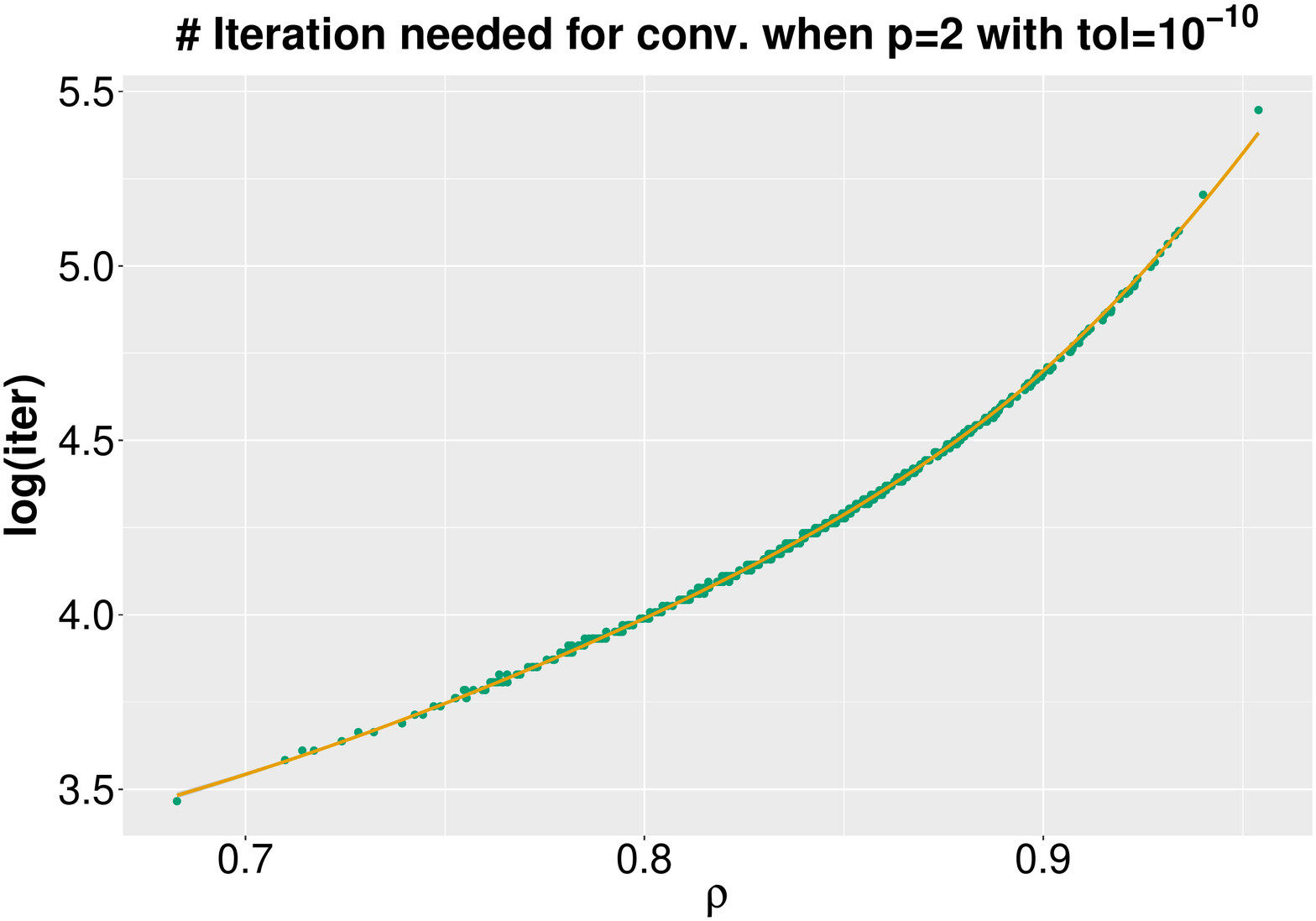}
  \caption{}
  \label{sec_NOS_fig1_a}
\end{subfigure}%
\begin{subfigure}{0.50\textwidth}
  \centering
  \includegraphics[scale= 0.225]{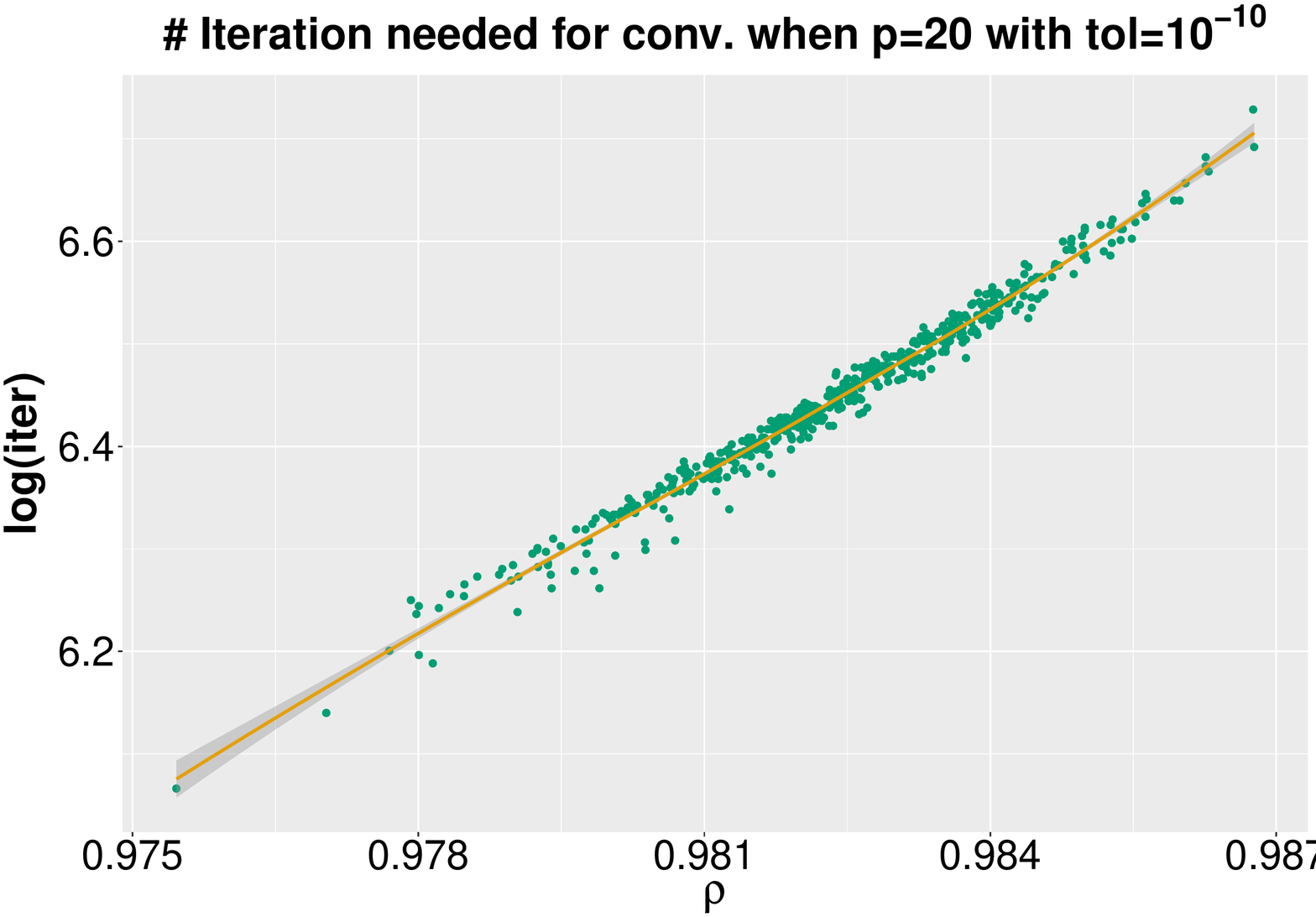}
  \caption{}
  \label{sec_NOS_fig1_b}
\end{subfigure}%
\caption{For each replicated dataset $i \in \set{1,2,\ldots,500}$, we observe the number of iterations ($\log$-scale) required by the algorithm for convergence and calculate the $\rho$ at the fixed point solution for $\alpha=1$. We plot ($\rho_i$,$\log(\text{iter}_i)$) for all $i \in \set{1,2,\ldots,500}$ and fitted with LOWESS line to explore the relationship between these two dependent variables. In (a), we generated the data $(y,\bX)$ with $(n=150, p=2)$. In (b), data $(y,\bX)$ is generated with $(n=150, p=20)$. It can be seen that number of iteration grows almost exponentially with the increasing $\rho$. Also, for fixed $n$, bigger $p$ leads to higher $\rho$ and as a consequence more iteration are required for convergence.}
\label{sec_NOS_fig1}
\end{figure}

It is worth noting that local asymptotic stability does not provide any information other than existence of a $\delta$ - neighborhood around $\xi^*$ such that, if the system is initialized in that region the iterates converge to $\xi^*$ as $t \to \infty $. Also, the definition does not say anything about the rate of convergence.  In the following, we provide a heuristic argument to connect the notion of rate of convergence with the spectral radius.

For simplicity, consider the one-dimensional system $x^{t+1} = g(x^t)$ for some function $g:\mb{R} \to \mb{R}$ which is twice continuously differentiable. If $x^*$ is a fixed point of this system, using Taylor's theorem we have for some $T_0>0$, $(x^{t+1} - x^*)\approx g(x^*)(x^t - x^*)$ for all $t \geq T_0$. Recall that the linear rate of convergence \citep{romero2019convergence} is given by,
$
\gamma = \lim_{t \to \infty} {\|x^{t+1} - x^*\|}\slash\|x^{t} - x^*\|,
$
provided the limit exists. In the above scenario, the iterates converge when $g(x^*) < 1$ and the rate of convergence is $g(x^*)$. For a general $d$-dimensional linear system $\alpha^{(t+1)} = A \alpha^{(t)}$ with fixed point $\alpha^* = 0$, it can be shown that, $\|\alpha^{(t)}\|_2 = \|A^t \alpha^{(0)}\|_2 \leq \{\rho(A)\}^t \|\alpha^{(0)}\|_2$ where $A$ is a square matrix and $\|\cdot\|_2$ is the Euclidean norm. Hence $\rho(A)$ acts as a rate of convergence for this case.  Figure  \ref{sec_NOS_fig1} is an illustration of the number of iterations needed for the system given by \eqref{sec_TTA_eqn3} to converge to the fixed point as a function of $\rho(\mathbf{J}_{\alpha=1})$. It is evident that the  number of  iterations increases exponentially as $\rho(\mathbf{J}_{\alpha=1})$ tends to $1$.    

\subsection{A special case of semi-orthogonal design}\label{ssec:semi}
In this section, we shall consider a simple hierarchical logistic regression model given by,
\begin{align}\label{Semi_ortho_eqn1}
p(y_{ij}=1 \mid \beta) = 1/\{1+ \exp{(-  \beta_j)}\} \quad ( i = 1,2,\ldots,n; \ j = 1,2,\ldots,p),
\end{align}  
We assume a prior $\beta \sim \mbox{N}_p(0,\sigma^2_{\beta}\mathbbm{I}_p)$. 
In this case, the results of Section \ref{sec:stabilityTT} can be strengthened to obtain a global convergence rate of the EM sequence \eqref{sec_TTA_eqn3}.  
One key advantage here is the ability to decouple the EM sequence into independent coordinate-wise updates. This is illustrated in 
Lemma \ref{Semi_ortho_lemma1}.  
\begin{lemma}\label{Semi_ortho_lemma1}
 The EM updates for the model \eqref{Semi_ortho_eqn1} can be simplified to,
\begin{align}\label{semi_ortho_eqn2}
\Big(\zeta^{t+1}_j\Big)^2 = \frac{1}{\{\sigma_{\beta}^{-2} - 2n\,A(\zeta^{t}_j)\}} + \frac{n^2\,(\bar{y}_j - 1/2)^2}{\{\sigma_{\beta}^{-2} - 2n\,A(\zeta^{t}_j)\}^{2}} \quad (j = 1,2,\ldots,p),
\end{align}
where, $\bar{y}_j = \sum^n_{i=1} y_{ij}/n$, for all $j = 1,2,\ldots,p$ and $\zeta^{t}$ is the update at the $t^{th}$ iteration.
\end{lemma}
\begin{proof} 
The log-likelihood from \eqref{Semi_ortho_eqn1} is given by,
$$
\log p(y\mid \beta) \propto \sum^p_{j=1} n\,\bar{y}_j \, \beta_j - \sum^p_{j=1}n\, \log\{1 + \exp(\beta_j)\}.
$$
Following the calculations of \eqref{sec_TTA_eqn1} and \eqref{sec_TTA_eqn2} corresponding to $\alpha=1$,
$$
\log{p}_{\l}(y,\beta \mid \zeta) = -\frac{1}{2}\beta^{\T}\left[\sigma^{-2}_{\beta}\mathbbm{I}_p - 2n\, \text{diag}\{A(\zeta)\}\right]\beta +  n\,{(\bar{Y}-1/2\ind_p )}^{\T}\beta + n\,\ind^{\T}_p C(\zeta)  + \mbox{Const.},
$$
where, $\bar{Y}^\T = [ \bar{y}_1,\bar{y}_2,\ldots, \bar{y}_p ]$. We claim from the above equation that, $ \beta_j \mid Y,\zeta_j \sim \mbox{N}\big(\mu(\zeta_j),\Sigma(\zeta_j)\big)$, independently for all $j =1,2,\ldots,p$.  Here $\Sigma^{-1}(\zeta_j) =\{\sigma_{\beta}^{-2} - 2n\,A(\zeta_j)\}$ and \\
$\mu(\zeta_j) = {n\,\big(\bar{y}_j- {1}/{2}\big)}/ \{\sigma_{\beta}^{-2} - 2n \,A(\zeta_j)\}$. Following the calculation similar to \eqref{sec_TTA_eqn3} we obtain \eqref{semi_ortho_eqn2}.
\end{proof}
It is important to distinguish between the EM update $\xi$ in \eqref{sec_TTA_eqn3} and $\zeta$ in \eqref{semi_ortho_eqn2}. In the general setting \eqref{sec_TTA_eqn0}, the variational parameter $\xi$ are introduced for each individual $i \in \set{1,2,\ldots,n}$, whereas $\zeta$ is introduced here for different groups $j\in \set{1,2,\ldots,p}$. Though we used similar techniques to get the updates, they have different interpretation. Figure \ref{numeric_fig1} and Figure \ref{ODM_fig1} are not comparable in that sense.

 The parallelization of the updates of $\zeta$ makes the posterior of $\beta$ independent.  Also, since the updates are independent and identical for all  $j = 1,2,\ldots,p$ given the initial point, it suffices to study the stability of a single coordinate. The following theorem assures the global asymptotic stability of the EM sequence in \eqref{semi_ortho_eqn2}.
\begin{theorem}\label{thm_semi_ortho}
The EM updates in \eqref{semi_ortho_eqn2} are globally asymptotically stable assuming $\beta_j \sim \mbox{N}(0,\sigma_\beta^2)$ with $\sigma_\beta =1$ for all $j = 1,2, \ldots, p$ and $n\geq 2$.  Moreover, with $\lambda_j^{t} := (\zeta_j^{t})^2\ (j=1, \ldots, p)$, there exists a global constant $\rho \in (0, 1)$ such that 
\begin{eqnarray*}
|\lambda_j^{t}  - \lambda_j^*| \leq \rho^t |\lambda_j^{0}  - \lambda_j^*|. 
\end{eqnarray*}
\end{theorem}
\begin{proof}
The proof of Theorem \ref{thm_semi_ortho} is provided for $\sigma_\beta = 1$ for technical convenience.
Letting $z=\zeta_j^2$ and $u= (\bar{y}_j - 0.5)$, consider $h_{u,\sigma_{\beta},n}(z) = 1/\{\sigma^{-2}_{\beta} - 2n\,A(\sqrt{z})\} + n^2\,u^2/ \{\sigma^{-2}_{\beta} - 2n\,A(\sqrt{z})\}^2$  for fixed $j = 1,2,\ldots,p$. Then one can write \eqref{semi_ortho_eqn2} by, $z^{t+1}= h(z^t)$. It is easy to see that,
\begin{align}\label{semi_ortho_eqn3}
h_{u,\sigma_{\beta},n}^{\prime}(z) &= \frac{n\,A^{\prime}(\sqrt{z})}{\sqrt{z}} \{\sigma^{-2}_{\beta} - 2n\, A(\sqrt{z})\}^{-2}\trd{1 +  \frac{2n^2\, u^2}{\{\sigma^{-2}_{\beta} - 2n\,A(\sqrt{z})\}}},
\end{align}
Let us call $\sigma_n = \{\sigma^{-2}_{\beta}/n - 2\,A(\sqrt{z})\}$. Since $u^2 \leq 1/4$ as $\bar{y}_j \in [0,1]$, we have the following inequality,
\begin{align}
h_{u,\sigma_{\beta},n}^{\prime}(z) &\leq \frac{A^{\prime}(\sqrt{z})}{\sqrt{z}} \sigma_n^{-2}\trd{\frac{1}{n} +  \frac{1}{2\,\sigma_n}} := h_{\sigma_{\beta},n}^{\prime}(z).\nonumber
\end{align}
In Appendix \S \ref{ortho_proof_convergence}, we show that $\sup_{n \geq 1} \| h_{\sigma_{\beta},n}^{\prime}\|_\infty <1$ when $\sigma_{\beta} = 1$, where 
$\| h_{\sigma_{\beta},n}^{\prime}\|_\infty := \sup_{z \in \mathbb{R}^+} h_{\sigma_{\beta},n}^{\prime}(z)$. The proof is then concluded by appealing to Lemma \ref{sec_NOS_lemma3} in the Appendix with $\rho = \sup_{n \geq 1} \| h_{1,n}^{\prime}\|_\infty$.
\end{proof} 

\begin{figure}[h]
    \centering
\begin{subfigure}{0.48\textwidth}
  \centering
  \includegraphics[scale= 0.220]{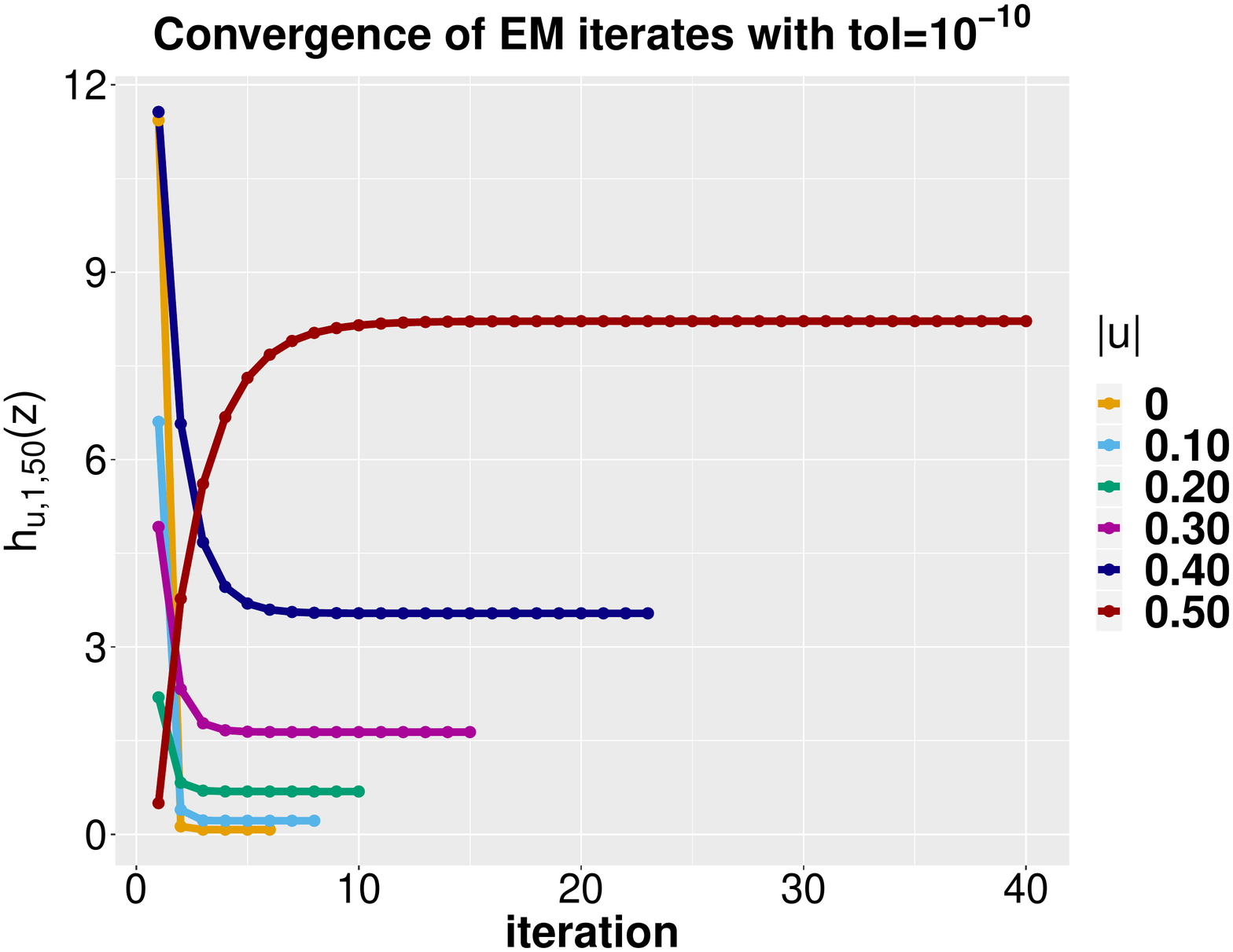}
  \caption{}
  \label{ODM_fig1_a}
\end{subfigure}%
\begin{subfigure}{0.48\textwidth}
  \centering
  \includegraphics[scale= 0.225]{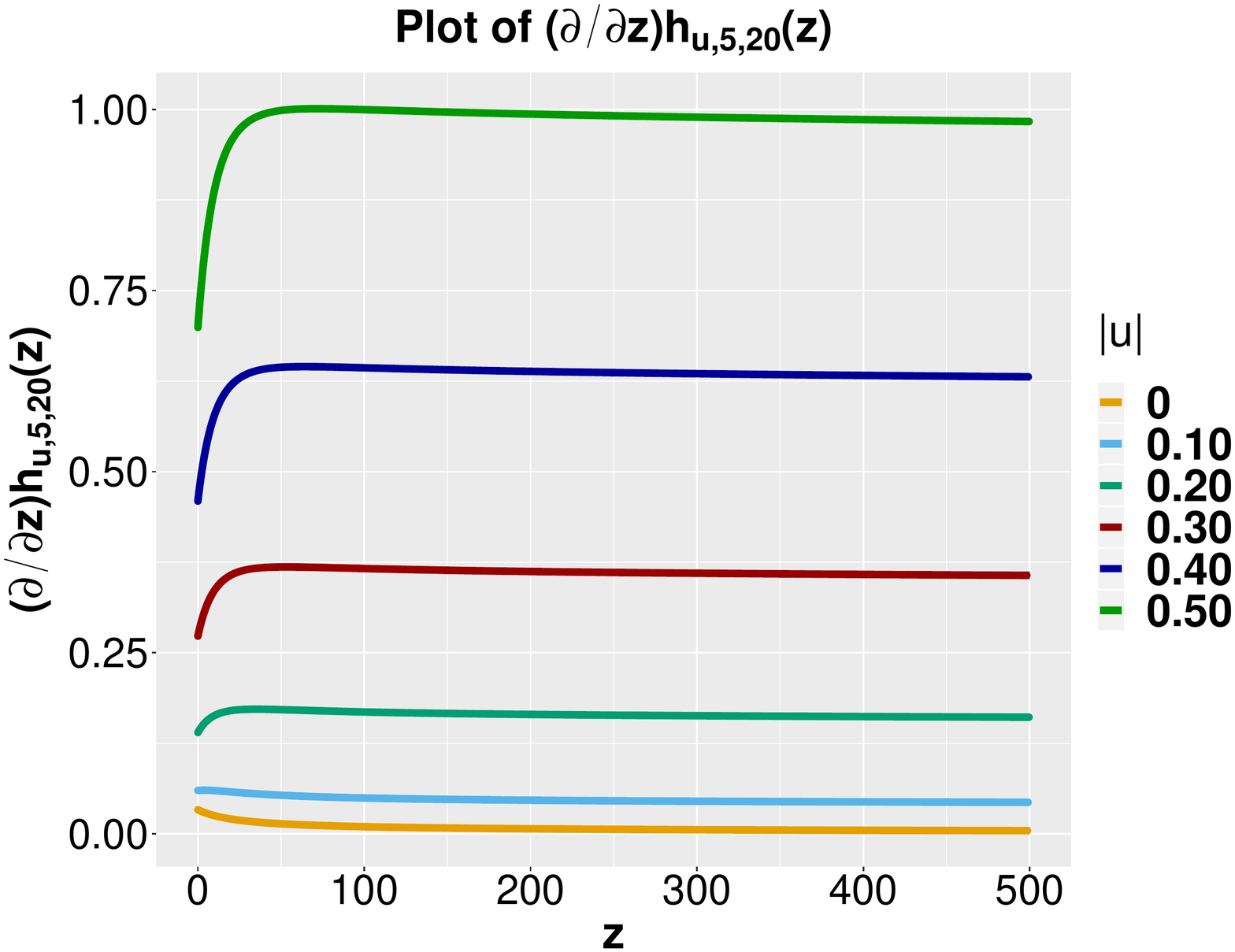}
  \caption{}
  \label{ODM_fig1_b}
\end{subfigure}%
\\
\begin{subfigure}{0.48\textwidth}
  \centering
  \includegraphics[scale= 0.225]{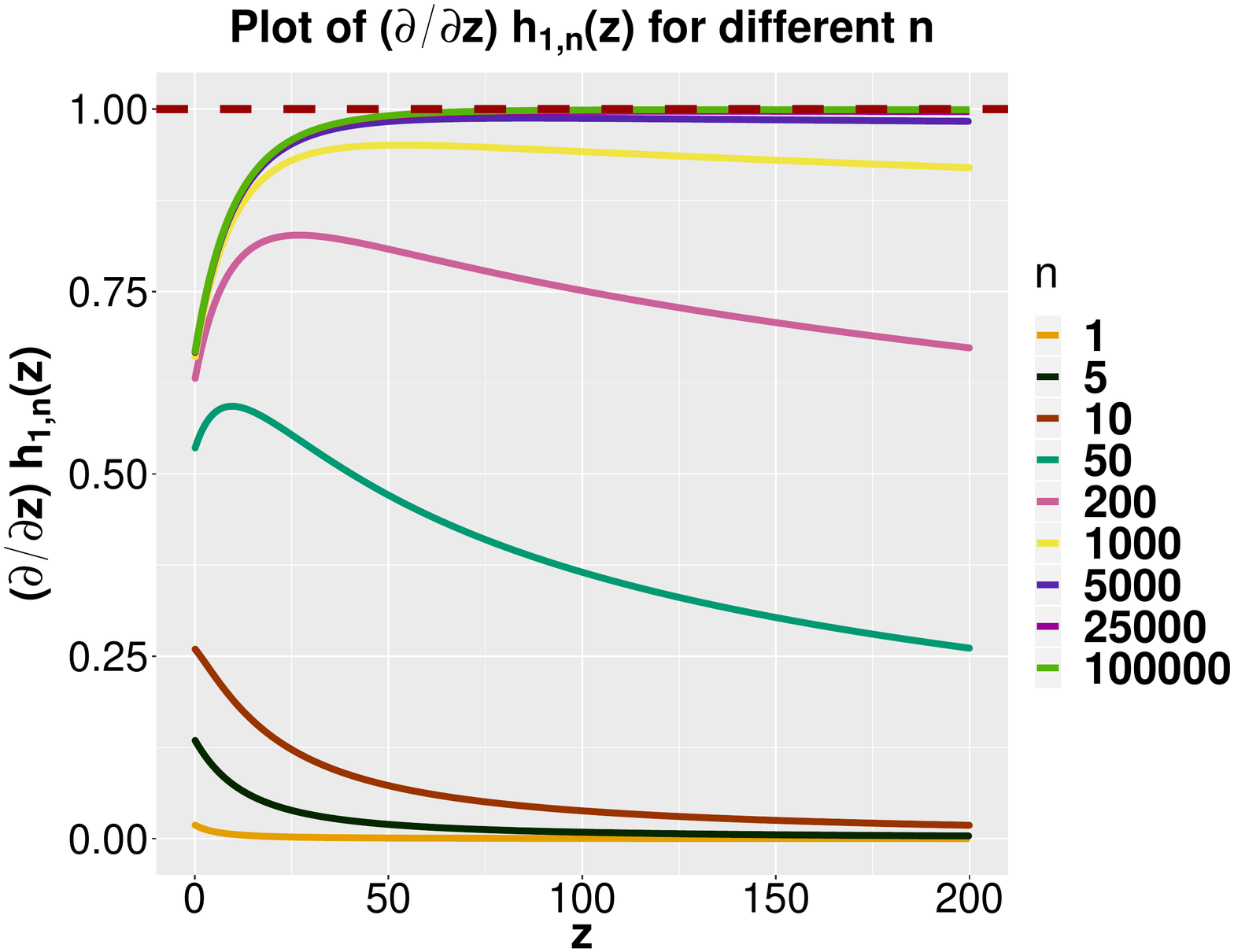}
  \caption{}
  \label{ODM_fig1_c}
\end{subfigure}%
\begin{subfigure}{0.48\textwidth}
  \centering
  \includegraphics[scale= 0.225]{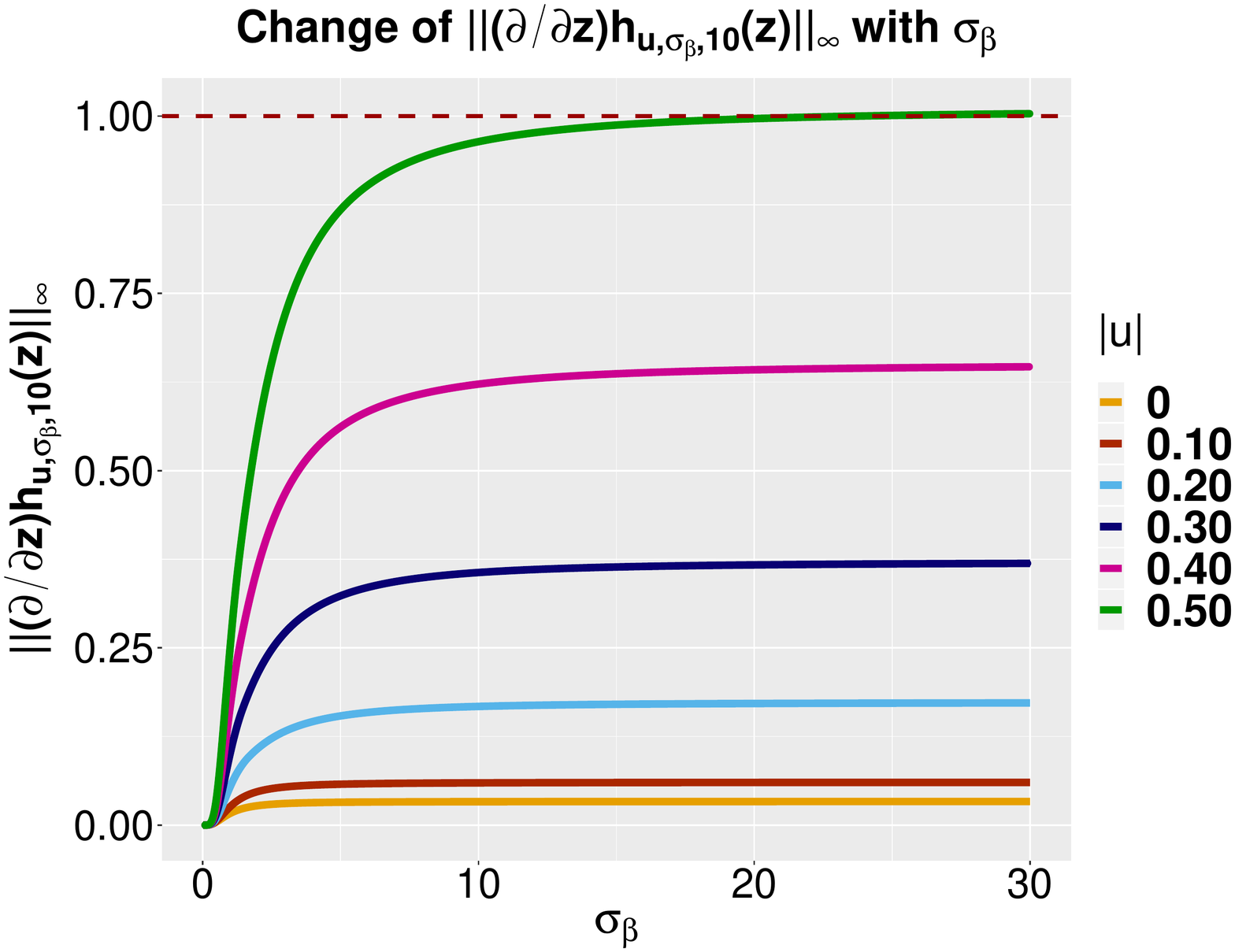}
  \caption{}
  \label{ODM_fig1_d} 
\end{subfigure}%
    \caption{(a) Evolution of $\lambda_1^t$ with arbitrary initialization when $\sigma_{\beta}=1$ and $n=50$ for different $|u|$ (b) Plot of $h_{u,\sigma_{\beta},n}^{\prime}(z)$ for varying $|u|$ when $\sigma_{\beta}=5$ and $n=20$.
(c) Plot of $h_{1,n}^{\prime}(z)$ for different values of $n$, (d) We plot $\|h_{u,\sigma_{\beta},n}^{\prime}\|_{\infty} = \max_{z\in \mb{R}^+}h_{u,\sigma_{\beta},n}^{\prime}(z)$ as a function of $\sigma_{\beta}$, for different $|u|$ and a fixed $n=10$. Numerically it is seen that  $\|h_{0.5,\sigma_{\beta},10}^{\prime}(z)\|_{\infty}\geq1$ when $\sigma_{\beta}\geq 12.894$.}
    \label{ODM_fig1} 
\end{figure}
It can be seen from \eqref{semi_ortho_eqn2} that the updates of $\zeta_j$ depend on the $(y,\bX)$ through $\bar{y}_j$, which is a sufficient statistic for $\beta_j$. Therefore if for some $j\neq j^{\prime}$ we have $\bar{y}_j = \bar{y}_{j^{\prime}}$,  the sequences $\{\zeta^t_j\}$ and $\{\zeta^t_{j^{\prime}} \}$ converge to the same limit.  Figure \ref{ODM_fig1_a} shows the global convergence of the EM sequence for different $\bar{y}_j$ when $\sigma_{\beta} = 1$ and $n=50$ with arbitrary initializations. Numerically we assumed convergence when $|\zeta^{t+1} - \zeta^t|<10^{-10}$. Interestingly, it is observed that convergence is slower when the data becomes more imbalanced, i.e. $|\bar{y}_j-0.5| \to \pm 0.5$.  A similar behavior for the mixing time of the P\'{o}lya-Gamma data augmentation Gibbs sampling in Bayesian logistic regression is observed in \cite{johndrow2016inefficiency}, which is all the more interesting given the connection between P\'{o}lya-Gamma augmentation and tangent transforms established by \cite{durante2019conditionally}.


Figure \ref{ODM_fig1_b} shows the behavior of $h_{u,5,20}^{\prime}(z)$  for different values of $u$. Barring $u=0$, in all other cases $h_{u,5,20}^{\prime}(z)$ increases first before dropping off. Figure \ref{ODM_fig1_c} shows that for fixed $z$, $h_{1,n}^{\prime}(z)$ is an increasing function of $n$ and less than $1$. Lemma \ref{semi_ortho_lemma1} proves this fact and in addition shows that for fixed $z$, $\lim_{n\to \infty} h_{1,n}^{\prime}(z) < 1$.  It is important to note is that $h_{\sigma_{\beta},n}^{\prime}(z)$ is dependent on $\sigma_{\beta}$ and for large $\sigma_{\beta}$ and fixed $z$, $h_{\sigma_{\beta},n}^{\prime}(z)$ may not be an increasing function of $n$.  Finally, Figure \ref{ODM_fig1_d} shows $\|h_{u,\sigma_{\beta},10}^{\prime}\|_\infty$ increases as $\sigma_{\beta}$ increases. It can be easily verified that for fixed $u$ and $n$, $h_{u,\sigma_{\beta},n}^{\prime}(z)$ is an increasing function of $\sigma_{\beta}$ and also for fixed $n$ and $\sigma_{\beta}$, an increasing function of $|u|$. Numerically it can be seen that $\|h^{\prime}_{0.5, \sigma_{\beta},10}\|_{\infty}\geq1$ when $\sigma_{\beta}\geq 12.894$. Overall, as the data gets more imbalanced, a flatter prior $\beta$ increasingly hurts the convergence. 


\section{Extension to Multinomial Logit}
In this section we provide an extension of the results in Section \ref{sec:stabilityTT}  to the case of multinomial logit regression where the response is an unordered categorical random variables with $K$ levels. Assume $y_i \,(i =1,2,\ldots,n)$, takes values in $\set{1, 2, \ldots, K}$ with following probabilities:
\[
p[y_i = j \mid \beta_1,\beta_2,\ldots,\beta_{K-1}]=
\begin{cases}
\frac{\exp{(\bx^T_i\beta_j)}}{1 + \sum^{K-1}_{j=1}\exp{(\bx^T_i\beta_j)}}&\text{for $ j = 1,2,\ldots,K-1$}\\
\frac{1}{1 + \sum^{K-1}_{j=1}\exp{(\bx^T_i\beta_j)}} &\text{for $j = K$}.\\
\end{cases}
\]
Assume $\beta_j \sim \mbox{N}_p(\mu_j,\Sigma_j)\ (j = 1,2,\ldots,K-1)$.  Let us define, $Y_{n \times K-1} = [Y_1, Y_2,\ldots, Y_{K-1}]$ with $Y^{\T}_j = \trd{\ind{\p{y_1=j}}, \ind{\p{y_2=j}}, \ldots, \ind{\p{y_n=j}}}\ (j = 1,2,\ldots,K-1)$, $\bX$ is the design matrix. Specific to each individual $i$ and class $j$, we introduce a variational parameter denoted by $\chi_{ij}\ ( i = 1,2,\ldots,n;\ j = 1, 2, \ldots, K-1)$. Let us call $\chi^{\T}_{j} = (\chi_{1j},\chi_{2j}, \ldots, \chi_{nj})$.
The following Lemma provides the update equation for the EM sequence.
\begin{lemma}\label{sec_EML_lemma1}
The EM updates to the above multinomial logit regression are given by,
 \begin{align}\label{sec_EML_eqn4}
      (\chi_{j}^{t+1})^2 = \mbox{diag}[\bX\{ \Sigma_{\alpha}(\chi^t_{j})/{\alpha} + \mu_{\alpha}(\chi^t_{j})\mu_{\alpha}^{\T}(\chi^t_{j})\}\bX^{\T}],\quad j\in\set{1, 2, \ldots, K-1} 
 \end{align}
 where, $\Sigma_{\alpha}^{-1}(\chi_{j}) = \Sigma^{-1}_{j}/{\alpha} - 2\bX^\T\mbox{diag}\{A(\chi_{j})\}\bX$ and $\mu_{\alpha}^\T(\chi_{j})\Sigma_{\alpha}^{-1}(\chi_{j}) = {\left(y_j- (1/2)\ind_n \right)}^{\T}\bX + \mu^{\T}_{j}\Sigma^{-1}_{j}/{\alpha}$.
\end{lemma}
\begin{proof}
We begin with the log-fractional likelihood,
\begin{eqnarray}\label{sec_EML_eqn1}
        \log{p^{\alpha}(y \mid \bX,\beta)} &=&   {\alpha}\,\sum^n_{i=1} \sum^{K-1}_{j=1} x^\T_i\beta_j \ind{[y_i=j]} - {\alpha}\,\sum^n_{i=1} \log \big(1+ \sum^{K-1}_{j=1} e^{x^\T_i\beta_j}\big),
\end{eqnarray}
\eqref{sec_EML_eqn1} poses the same difficulty of intractability as \eqref{sec_TTA_eqn1}. Moreover, the logistic term can not be optimized straightaway due to sum of exponents inside the logistic function.  Various methods have been propose to circumvent this issue; Taylor approximation to the $\log$-sum-$\exp$ term \citep{braun2010variational}, Quasi-Monte-Carlo \citep{lawrence2004reducing}, Jensen's inequality \citep{blei2007correlated}, quadratic approximation \citep{bouchard2008efficient,jebara2012majorization}.  
In the following, we use a very recent method by \cite{aueb2016one} which has been numerically shown to outperform the others which is based on the following inequality, 
\begin{eqnarray}\label{sec_EML_eqn2}
\sum^{K-1}_{j=1}\log\p{1+ e^{\bx^\T_i\beta_j}} \geq  \log\bigg(1+ \sum^{K-1}_{j=1} e^{\bx^\T_i\beta_j}\bigg).
\end{eqnarray}
In an ideal scenario, if $y_i = l$ both the terms in the inequality above should be $\approx \log(1 + \exp(\bx^\T_i \beta_l))$. This inequality is not too loose in that sense.  Using \eqref{sec_EML_eqn2} in \eqref{sec_EML_eqn1}, we get a lower bound to $p^{\alpha}(y \mid \bX,\beta )$ given by,
\begin{eqnarray}\label{sec_EML_eqn3}
    \log{p^{\alpha}(y \mid \bX,\beta)} &\geq&  {\alpha}\, \sum^{K-1}_{j=1} \bigg\{\sum^n_{i=1}\bx^\T_i\beta_j\ind{[y_i=j]} - {\alpha}\, \sum^n_{i=1}\log\p{1+ e^{\bx^\T_i\beta_j}}\bigg\}.
\end{eqnarray}
Now we use the quadratic bound proposed by \cite{jaakkola2000bayesian} on the right hand side of the above inequality. 
This leads to a lower bound to $\log{p^{\alpha}(y \mid \bX,\beta)}$ similar to \eqref{sec_TTA_eqn2} where
\begin{align*}
\log{{p}_{\l}^\alpha(y,\beta \mid  \bX,\chi)} = &{}\sum^{K-1}_{j=1}\trd{\alpha\,\bigg\{Y^\T_j\bX\beta_j + {(\bX\beta_j)}^{\T} \text{diag}\{A(\chi_{ij})\} (X\beta_j) - \frac{1}{2}\ind^{\T}_n X \beta_j + \ind^{\T}_n C(\chi_{ij})\bigg\}}\\ 
&-   \sum^{K-1}_{j=1}\trd{\frac{1}{2}\p{\beta_j - \mu_j}^\T\Sigma^{-1}_j\p{\beta_j - \mu_j}} + \mbox{Const.},
\end{align*}
for fixed $j \in \set{1,2,\ldots ,K-1}$ the updates are exactly similar to the updates in logistic version. Moreover, updates to $\chi_{j}^{\T} = \trd{\chi_{1j},\chi_{2},\ldots,\chi_{nj}}^{\T}$ are independent over $j \in \set{1,2,\ldots ,K-1}$. Following the similar E-step and M-step for the logistic version as in \eqref{eq:EMQ}-\eqref{sec_TTA_eqn3}, it can be easily seen that for fixed $j$ the update equation is given by \eqref{sec_EML_eqn4}.     
\end{proof}

As the updates across each level $j\in \set{1,2,\ldots,K-1}$ are independent and the behavior of the updates is exactly similar to the binary setup in \eqref{sec_TTA_eqn3}, this leads us to the following theorem that guarantees the local asymptotic stability of EM updates in Lemma \ref{sec_EML_lemma1}.
\begin{theorem}
Suppose the design matrix $\bX$ does not have any row equal to the zero vector. For any $\alpha \in (0, 1]$ and positive definite $\Sigma_\beta$, any fixed point solution  $\chi^*_{j}$ of the EM sequence in \eqref{sec_EML_eqn4} 
is locally asymptotically stable. 
\end{theorem}
\begin{proof}
For each fixed $j \in \set{1,2, \ldots,K-1}$, the fixed point equation is,
\begin{align}\label{sec_EML_fp}
      (\chi^{*}_{j})^2 = \mbox{diag}[\bX\{ \Sigma_{\alpha}(\chi^*_{j}) + \mu_{\alpha}(\chi^*_{j})\mu_{\alpha}^{\T}(\chi^*_{j})\}\bX^{\T}].
\end{align}
 Call $\xi^* = \chi^*_{j}$ and $\xi^t = \chi^t_{j}$. Then,  \eqref{sec_EML_eqn4} and \eqref{sec_EML_fp} reduces to \eqref{sec_TTA_eqn3} and \eqref{eq:fp} respectively. Now, we directly apply Theorem \ref{main_VB_1} to conclude that the updates in \eqref{sec_EML_eqn4} are locally asymptotically stable. 
\end{proof}

\section*{Acknowledgements}
Dr. Pati and Dr. Bhattacharya acknowledge support from NSF DMS (1613156, 1854731, 1916371) and NSF CCF 1934904 (HDR-TRIPODS). In addition, Dr. Bhattacharya acknowledges NSF CAREER 1653404 award for supporting this project.

\appendix

\section{Proof of Statistical optimality results in Section \ref{sec:statopt}}
In the following, we first provide the proofs of Theorems \ref{thm:statopt} and \ref{thm:statopt2} in \S \ref{ssec:statopt1} and \ref{ssec:statopt2} respectively and then provide the proofs of some of the auxiliary results used in subsequent \S \ref{aux:1}. 
\subsection{Proof of Theorem \ref{thm:statopt}}\label{ssec:statopt1}
The proof consists of two major steps.  \\ \\
{\em Risk majorization.}  In this first step, we obtain an upper bound to the integrated risk in terms of easily controllable quantities. 
We denote $\bbE_{\beta^*}$ as taking expectation under \eqref{eq:true}.  From  the definition of the $\alpha$-Renyi divergence  and the fact that ${p}_{\l}$ lower bounds $p(y \mid \beta, \bX)$ 
\begin{align*}
\bbE_{\beta^*} \exp\Big\{\alpha \,  \log\frac{{p}_{\l}(y\,\mid\,\beta, \xi, \bX)}{p(y\,\mid\,\beta^\ast, \bX)}\Big\} \leq \bbE_{\beta^*} \exp\Big\{\alpha \,  \log\frac{p(y \,\mid\, \beta, \bX)}{p(y \,\mid\, \beta^\ast, \bX)}\Big\}
 = e^{- n(1 - \alpha) \mbox{D}_{\alpha}(\beta, \beta^\ast) }.
\end{align*}
Thus, for any $\varepsilon \in (0, 1)$, we have
\begin{align*}
\bbE_{\beta^\ast} \exp \bigg[\alpha \,  \log\frac{{p}_{\l}(y\,\mid\,\beta, \xi, \bX)}{p(y\,|\,\beta^\ast, \bX)} + n (1 - \alpha) \mbox{D}_{\alpha}(\beta, \beta^\ast) - \log(1/\varepsilon) \bigg] \le \varepsilon. 
\end{align*}
Integrating both side of this inequality with respect to the prior $\pi_\beta$ and interchanging the integrals using Fubini's theorem, we obtain
\begin{align*}
\bbE_{\beta^\ast} \int \exp\bigg[\alpha \,  \log\frac{{p}_{\l}(y\,|\,\beta, \xi, \bX)}{p(y\,|\,\beta^\ast, \bX)} + n(1 - \alpha) \mbox{D}_{\alpha}(\beta, \beta^\ast) - \log(1/\varepsilon)\bigg] \pi_\beta(\beta)\,d\beta \le \varepsilon. 
\end{align*}
Now, recall the {\em variational inequality} for a probability measure $\mu$ and for $h$ such that  $e^h$ is integrable, 
\begin{align} \label{eq:vi}
\log \int e^h d\mu = \sup_{\rho \ll \mu} \bigg[ \int h d\rho - D(\rho \vert \vert \mu) \bigg].
\end{align}
Using \eqref{eq:vi}, 
\begin{align*}
\bbE_{\beta^\ast} \exp \sup_{q \ll \pi_\beta} \bigg[\int \bigg\{ \alpha \,  \log\frac{{p}_{\l}(y\,\mid\,\beta, \xi, \bX)}{p(y\,\mid\,\beta^\ast, \bX)} + n(1 - \alpha) \mbox{D}_{\alpha}(\beta, \beta^\ast)  - \log(1/\varepsilon)\bigg\}\,q(\beta)\,d\beta - \mbox{D}(q \,\vert \vert \,\pi_\beta)  \bigg] \le \varepsilon.
\end{align*}
If we choose $\rho = q^*_\beta \equiv \phi_p \big\{ \beta; \mu_{\alpha}(\xi^*), \Sigma_{\alpha}(\xi^*) \big\}$ as the variational approximation  and set $\xi = \xi^*$
\begin{align*}
\bbE_{\beta^\ast} \exp \bigg[\int \bigg\{ \alpha \,  \log\frac{{p}_{\l}(y\,\mid\,\beta, \xi^*, \bX)}{p(y\,\mid\,\beta^\ast, \bX)} + (1 - \alpha) \mbox{D}_{\alpha}(\beta, \beta^\ast)  - \log(1/\varepsilon)\bigg\} \,q^*_\beta(\beta)\,d\beta &- \mbox{D}(q^*_\beta \,\vert \vert\, \pi_\beta)  \bigg]\\ &\le \varepsilon. 
\end{align*}
By applying Markov's inequality, we further obtain that with $\bbP_{\beta^\ast}$ probability at least $(1 - \varepsilon)$, 
\begin{align*}
n(1 - \alpha) \int \mbox{D}_{\alpha}(\beta, \beta^\ast) \, q^*_\beta(\beta)\,d\beta
& \le - \alpha  \int_\beta  \log\frac{{p}_{\l}(y\,\mid\,\beta, \xi^*, \bX)}{p(y\,|\,\beta^\ast, \bX)} \,q^*_\beta(\beta)\,d\beta+ \mbox{D}(q^*_\beta \, \vert \vert\, \pi_\beta) + \log(1/\varepsilon).
\end{align*}
Now using the Lemma \ref{eq:varsoln}, 
\begin{align}\label{eq:riskmajor}
  &-\alpha\, \int_\beta  \log\frac{{p}_{\l}(y\,\mid\,\beta, \xi^*, \bX)}{p(y\,\mid\,\beta^\ast,\bX)} \,q^*_\beta(\beta)\,d\beta+ \mbox{D}(q^*_\beta\, \vert \vert\, \pi_\beta) \nonumber\\
 &= \inf_{q, \xi}  \bigg\{-\alpha\, \int_\beta  \log\frac{{p}_{\l}(y\,\mid\,\beta, \xi, \bX)}{p(y\,|\,\beta^\ast, \bX)} \,q(\beta)\,d\beta+ \mbox{D}(q\, \vert \vert\, \pi_\beta)\bigg\}. 
\end{align}

\noindent {\em Optimizing the majorized risk.}
Our second step consists of optimizing the term obtained in \eqref{eq:riskmajor} by choosing  suitable candidates for $q$ and $\xi$. We refer to them as $\tilde{q}$ and $\tilde{\xi}$.  The idea is to choose $\tilde{q}$ and $\tilde{\xi}$  so that $\tilde{q}$ places almost all its mass into a small neighborhood around truth $\bX \beta^*$, so that the first term in the right hand side of \eqref{eq:riskmajor} becomes small; on the other hand, the neighborhood is large enough so that the second regularization term $n^{-1}\,D(q\, ||\, \pi_\beta)$ is not too large. 
We choose $\tilde{q}$ first and $\tilde{\xi}$ later.  Let $\tilde{q}$
\begin{align}
\tilde{q}(\beta) = \frac{\pi_\beta(\beta)}{\pi_\beta\big[\mathcal{B}_n(\beta^\ast,\,\varepsilon)\big]}\, I_{\mathcal{B}_n(\beta^\ast,\,\varepsilon)}(\beta),\quad\forall \beta\in\beta, \label{Eqn:KL_res}
\end{align}
be the restriction of the prior density $\pi_\beta$ into the KL neighborhood $\mathcal{B}_n(\beta^\ast,\,\varepsilon)$ around $\beta^\ast$ with radius $\varepsilon$ defined as
\begin{align*}
\mathcal{B}_n(\beta^\ast,\varepsilon) &=  \Big\{n^{-1}\,\widetilde{\mbox{D}}\big[ p(\cdot\mid\beta^\ast, \bX)\, \big|\big| \, {p}_{\l}(\cdot\mid\beta, \xi, \bX)\big]\leq \varepsilon^2,\,\, n^{-1}\,\mbox{V}\big[ p(\cdot\mid\beta^\ast, \bX)\, \big|\big| \,  {p}_{\l}(\cdot\mid\beta, \xi, \bX)\big] \leq \varepsilon^2\Big\},
\end{align*}
where for two non-negative functions $f, g$, 
$\widetilde{\mbox{D}}(f \,\| \,g) = \int f | \log (f/g)|$ and  $\mbox{V}(f, g) := \int f (\log f/g)^2 - \widetilde{\mbox{D}}(f \,\| \,g)^2$. Note that $\tilde{\mbox{D}}(f \,|\,g)$ is an extension of the usual KL distance for probability measures to positive functions which may not integrate to one.   With this substitution, the second term in \eqref{eq:riskmajor} becomes the negative log prior mass $[n\,(1-\alpha)]^{-1}\, \log \big\{\pi_\beta[ \mathcal{B}_n(\beta^\ast,\,\varepsilon)]\big\}^{-1}$ and it remains to provide a high-probability bound for the first term and an upper bound for the log-prior concentration term $\log \big\{\pi_\beta[ \mathcal{B}_n(\beta^\ast,\,\varepsilon)]\big\}$.  

\noindent {\em i) High probability upper bound for the first term in \eqref{eq:riskmajor}}. 
By applying Fubini's theorem and invoking the definition of $\mathcal{B}_n(\beta^\ast,\,\varepsilon)$, we have
\begin{align*}
&\bbE_{\beta^\ast}\bigg[\int_\beta \tilde{q}(\beta)\, \log\frac{{p}_{l}(y\,\mid\,\beta, \xi, \bX)}{p(y\,\mid\,\beta^\ast, \bX)}\, d\beta\bigg] 
= \int_\beta  \bbE_{\beta^\ast}\bigg[ \log\frac{{p}_{\l}(y\,\mid\,\xi, \beta, \bX)}{p(y\,\mid\,\beta^\ast, \bX)} \bigg]\, \tilde{q}(\beta)\,d\beta\\
\leq &\, \int_{ \mathcal{B}_n(\beta^\ast,\, \varepsilon)} \widetilde{\mbox{D}}\big[ p(\cdot\,\mid\,\beta^\ast, \bX)\, \big|\big| \, {p}_{\l}(\cdot\,\mid\,\beta, \xi, \bX)\big]\, \tilde{q}(\beta)\,d\beta \leq n\,\varepsilon^2.
\end{align*} 
Similarly, we have the following bound for the second moment by applying the Cauchy-Schwarz inequality,
\begin{align*}
&\mbox{Var}_{\beta^\ast}\bigg[\int_\beta \tilde{q}(\beta)\, \log\frac{{p}_{\l}(y\,\mid\,\beta, \xi, \bX)}{p(y\,\mid\,\beta^\ast, \bX)}\, d\beta\bigg]\leq \int_{\mathcal{B}_n(\beta^\ast,\, \varepsilon)} V\big[ p(\cdot\,\mid\,\beta^\ast, \bX)\, \big|\big| \, p(\cdot\,\mid\,\beta, \xi, \bX)\big]\, \tilde{q}(\beta)\,d\beta \leq n\,\varepsilon^2.
\end{align*} 

Putting pieces together, applying Chebyshev's inequality, we obtain
\begin{align*}
& \bbP_{\beta^\ast} \bigg\{\int_\beta \tilde{q}(\beta)\, \log\frac{{p}_{\l}(y\,\mid\,\beta, \xi, \bX)}{p(y\,\mid\,\beta^\ast,\bX)}\, d\beta \leq - D\,n \,\varepsilon^2\bigg\}  \\
& \le \bbP_{\beta^\ast} \bigg\{\int_\beta \tilde{q}(\beta)\, \log\frac{{p}_{\l}(y\,\mid\,\beta, \xi, \bX)}{p(y\,\mid\,\beta^\ast, \bX)}\, d\beta - \bbE_{\beta^\ast}\Big[\int_\beta \tilde{q}(\beta)\, \log\frac{{p}_{\l}(y\,\mid\,\beta, \xi, \bX)}{p(y\,\mid\,\beta^\ast, \bX)}\, d\beta\Big] \leq -(D-1) \,n \,\varepsilon^2\bigg\}  \\
& \le \frac{\mbox{Var}_{\beta^\ast}\big[\int_\beta \tilde{q}(\beta)\, \log\frac{{p}_{\l}(y\,\mid\,\beta,\xi, \bX)}{p(y\,\mid\,\beta^\ast,\bX)}\, d\beta\big] }{(D-1)^2\, n^2\, \varepsilon^4} 
\le \frac{1}{(D-1)^2\, n \, \varepsilon^2}. 
\end{align*}
It follows with probability $1- 1/ \{(D-1)^2\, n \, \varepsilon^2)\}$, 
the first term of \eqref{eq:riskmajor} evaluated at $q = \tilde{q}$ satisfies
\begin{align*}
 - \alpha \int_\beta \tilde{q}(\beta)\, \log\frac{{p}_{\l}(y\,\mid\,\beta, \xi, \bX)}{p(y\,\mid\,\beta^\ast, \bX)}\, d\beta  \leq  Dn \alpha \, \varepsilon^2. 
\end{align*}

\noindent {\em ii) Upper bound for the negative $\log$-prior concentration term $-\log \big\{\pi_\beta[ \mathcal{B}_n(\beta^\ast,\,\varepsilon)]\big\}$}.  
We first obtain an upper bound for the $\log$-pseudo-likelihood ratio
\begin{align*}
 \log\frac{{p}_{\l}(y\,\mid\,\beta, \xi, \bX)}{p(y\,\mid\,\beta^\ast, \bX)} &=  
   y^{\T}\bX(\beta - \beta^*)  
       + \beta^{\T}\left\{\bX^{\T}\text{diag}\{A(\xi)\}\bX\right\}\beta + 0.5\ind_n^{\T}\bX \\
&+ \ind^{\T}_n C(\xi) +  \ind_n^{\T} \log \{ 1+ \exp(\bX \beta^*)\}. 
\end{align*}

To obtain the lower bound ${p}_{\l}(y \mid \xi, \beta, \bX)$ of 
$p(y \mid  \beta, \bX)$, 
\cite{jaakkola2000bayesian} used $- \log \{1+\exp(-x)\} = x/2 - \log(e^{x/2} + e^{-x/2})$ and noted 
 $f(x) = - \log(e^{x/2} + e^{-x/2})$ is a convex function in the variable $x^2$. 
 Since a tangent surface to a
convex function is a global lower bound for the function, we can bound $f (x)$ globally
with a first order Taylor expansion in the variable of $x^2$ around $\xi^2$ as
\begin{align}\label{eq:convex_gap}
f(x) &\geq f(\xi)  + \frac{d f(\xi)}{d\xi^2} (x^2 - \xi^2)\nonumber \\
 &=   - \log \{1+\exp(-\xi)\} - \xi/2 - \frac{1}{4\xi} \tanh(\xi/2)(x^2 - \xi^2). 
\end{align}
To quantify the gap  $\Delta(\beta, \beta^*):= \log{p}_{\l}(y\,\mid\,\beta, \xi, \bX)  - \log p(y\,\mid\,\beta^\ast, \bX)$, observe that
\begin{align*}
\Delta(\beta, \beta^*) &=  \log p(y\,\mid\,\beta, \bX)  - \log p(y\,\mid\,\beta^\ast, \bX) +
  \log{p}_{\l}(y\,\mid\,\beta, \xi, \bX) - \log p(y\,\mid\,\beta, \bX)  \\
  &:=  y^\T \bX(\beta - \beta^*) + \ind^{\T}_n [\log (1+ \exp(\bX\beta^*)  - \log (1+ \exp(\bX\beta)] + \Delta 
\end{align*}
where $\Delta$ is the {\em Jensen-Gap} in \eqref{eq:convex_gap}.  To estimate $\Delta$, we perform a second order Taylor-expansion around $\xi^2$
\begin{eqnarray*}
f(x) = f(\xi)  + \frac{d f(\xi)}{d\xi^2} (x^2 - \xi^2) + \frac{1}{2} \frac{d^2 f(\xi)}{d\xi^4} |_{\xi = \hat{\xi}} (x^2 - \xi^2)^2. 
\end{eqnarray*}
Observe further, 
\begin{eqnarray*}
\frac{d f(\xi)}{d\xi^2} = - \frac{1}{4\surd{\xi^2}} \tanh \frac{\surd{\xi^2}}{2}, \quad \frac{d^2 f(\xi)}{d\xi^4} = - \Big[\frac{0.0625\, \mbox{sech}^2(\surd{\xi^2}/2)}{\xi^2} - \frac{0.125 \tanh(\surd{\xi^2}/2)}{(\xi^2)^{1.5}} \Big].  
\end{eqnarray*}
Moreover, $d^2 f(\xi)/ d\xi^4$ is a decreasing function of $\xi^2$ and $0 < d^2 f(\xi)/d\xi^4 < 1$. Hence $
 \Delta  \leq  \sum_{i=1}^n \{ (\bx_i^{\T}\beta)^2 - \xi_i^2\}^2.
$
Setting $\xi_i = \bx_i^{\T} \beta^*$ for all $i=1, \ldots, n$, we have 
\begin{eqnarray*}
\Delta &\leq& \sum_{i=1}^n \{\bx_i^{\T}(\beta - \beta^*)\}^2 \{\bx_i^{\T}(\beta - \beta^*) + 2 \bx_i^{\T} \beta^* \}^2  \\
&\leq& 2 \sum_{i=1}^n \{\bx_i^{\T}(\beta - \beta^*)\}^4 + 8\sum_{i=1}^n \{\bx_i^{\T} \beta^* \}^2\{\bx_i^{\T}(\beta - \beta^*)\}^2.   \\
&\leq& 2n \| \bX\|_{2, \infty}^4 \|\beta - \beta^*\|^4 + 8n  \| \bX\|_{2, \infty}^4 \|\beta^*\|_2^2 \|\beta - \beta^*\|^2,
\end{eqnarray*}
where the final inequality follows from $\bx_i^{\T}(\beta - \beta^*) \leq \| \bX\|_{2, \infty} \|\beta - \beta^*\|$.  Plugging in the bound obtained above in $\Delta(\beta, \beta^*)$, we get 
\begin{eqnarray*}
\Delta(\beta, \beta^*) &\leq& 
   y' \bX(\beta - \beta^*) + \ind^{\T}_n [\log (1+ \exp(\bX\beta^*)  - \log (1+ \exp(\bX\beta)] + \sum_{i=1}^n \{ (\bx_i^{\T}\beta)^2 - (\bx_i^{\T}\beta^*)^2\}^2 \\
 & \leq&  \sum_{i=1}^n \{ y_i + 1 \} \bx_i^{\T} (\beta - \beta^*) + 2n \| \bX\|_{2, \infty}^4 \|\beta - \beta^*\|^4 + 8n  \| \bX\|_{2, \infty}^4 \|\beta^*\|_2^2 \|\beta - \beta^*\|^2 \\
   & \leq&  2n  \| \bX\|_{2, \infty} \|\beta - \beta^*\| + 2n \| \bX\|_{2, \infty}^4 \|\beta - \beta^*\|^4 + 8n  \| \bX\|_{2, \infty}^4 \|\beta^*\|_2^2 \|\beta - \beta^*\|^2. 
\end{eqnarray*}
where the last inequality is obtained by noting that $\log(1+ e^x)$ is a $1$-Lipschitz function. Recall that $L(\beta^*, \bX) = \max\{4\| \bX\|_{2, \infty}, 8 \| \bX\|_{2, \infty}^2 \|\beta^*\|_2 \}$. If $\|\beta  - \beta^*\| < \varepsilon^2/L(\beta^*, \bX)$, then
$\Delta(\beta, \beta^*) \leq n \varepsilon^2$ which implies 
\begin{align*}
n^{-1}\,\widetilde{\mbox{D}}\big[ p(\cdot\,\mid\,\beta^\ast, \bX)\, \big|\big| \, {p}_{\l}(\cdot\,\mid\,\beta, \xi, \bX)\big]\leq \varepsilon^2.
\end{align*}  
Also, since 
\begin{align*}
V \big[ p(y\,\mid\,\beta^\ast, \bX)\, \big|\big| \, p(y\,\mid\,\beta, \xi, \bX)\big]= n 
V \big[ p(y_1\,\mid\,\beta^\ast, \bx_1)\, \big|\big| \, p(y_1\,\mid\,\beta, \xi, \bx_1)\big],
\end{align*}
following the same argument as before but with one observation, we have
\begin{align*}
\Delta_1(\beta, \beta^*)&:= \log{p}_{\l}(y_1\,|\,\beta, \xi, \bx_1)  - \log p(y_1\,|\,\beta^\ast, \bx_1) \\
& \leq 2 \| \bX\|_{2, \infty} \|\beta - \beta^*\| + 2 \| \bX\|_{2, \infty}^4 \|\beta - \beta^*\|^4 + 8  \| \bX\|_{2, \infty}^4 \|\beta^*\|_2^2 \|\beta - \beta^*\|^2, 
\end{align*} 
which implies if $\|\beta  - \beta^*\| < \varepsilon^2/L(\beta^*, \bX)$, then $
n^{-1}\,\mbox{V}\big[ p(\cdot\,\mid\,\beta^\ast, \bX)\, \big|\big| \, {p}_{\l}(\cdot\,\mid\,\beta, \xi, \bX)\big] \leq \varepsilon^2$. Hence 
\begin{align*}
-\log \big\{\pi_\beta[ \mathcal{B}_n(\beta^\ast,\,\varepsilon)]\big\} \leq 
-\log\pi_\beta \big\{\|\beta  - \beta^*\| < \varepsilon^2/L(\beta^*, \bX)\big\} \leq p \log 
\Big\{\frac{L(\beta^*, \bX)}{\varepsilon^2} \Big\} + \frac{1}{2}\beta^*\Sigma_\beta^{-1} \beta^*, 
\end{align*}
where the final inequality holds using using multivariate Gaussian concentration through Anderson's inequality.   

\subsection{Proof of Theorem \ref{thm:statopt2} }\label{ssec:statopt2}
We start by rewriting the $\log$-likelihood ratio as
\begin{align*}
\log \frac{p(y \mid \beta, \bX)}{p(y \mid \beta^*, \bX)} &= (\beta - \beta^\ast)^{\T} \bX^\T y - \sum_{i=1}^n [a(\bx_i^{\T}\beta) - a(\bx_i^{\T}\beta^\ast)] \\
&= \langle y - \mb E y, \bX(\beta - \beta^\ast)\rangle - n\,D(\beta^\ast, \beta). 
\end{align*}
 Since $a(t) = \log(1 + e^t)$ satisfies $a(t + h) \ge a(t) + h \, a^{(1)}(t) + \r(|h|) \, a^{(2)}(t)/2$ for all $t, h$,  where $r(h) = h^2/(\r_1 h + 1)$ for $\r_1 > 0$, we have 
 \begin{eqnarray*}
n\ \mbox{D}(\beta^\ast, \beta)=   \sum_{i=1}^n \big\{ a(\bx_i^{\T}\beta) - 
a(\bx_i^{\T}\beta^*) - a^{(1)}(\bx_i^{\T}\beta^*) \bx_i^{\T}(\beta - \beta^*)\big\}  \geq \sum_{i=1}^n \r (|\bx_i^{\T}(\beta - \beta^*)|) a^{(2)}(\bx_i^{\T}\beta^*). 
\end{eqnarray*}
Defining $\k(h) = h^2/\r(h)$ and $W = \mbox{diag}[a^{(2)}(\bx_1^{\T}\beta^*), \ldots, a^{(2)}(\bx_n^{\T}\beta^*)]$, 
 \begin{eqnarray*}
n\ \mbox{D}(\beta^*, \beta) \geq (\beta - \beta^*)^{\T}  \bigg[\sum_{i=1}^n \frac{a^{(2)}(\bx_i^{\T}\beta^*)}{\k(|\bx_i^{\T}(\beta - \beta^*)|) } \bx_i \bx_i^{\T} \bigg] (\beta - \beta^*) \geq  \frac{(\beta - \beta^*)^{\T}\bX^{\T} W\bX  (\beta - \beta^*)}{1+ \r_1\|\bX\|_\infty \surd{p} \|\beta - \beta^*\|},
\end{eqnarray*}	
where the last inequality follows $|\bx_i^{\T}(\beta - \beta^*)| \leq \|\bX\|_\infty \surd{p}\|\beta - \beta^*\|$.  
Define $\Omega_n$ be the set 
\begin{eqnarray*}
\max_{1\leq j \leq p } \Big|\sum_{i=1}^n (y_i - \mb E y_i) \, x_{ij}\Big| \leq \|\bX\|_\infty ({n \log p})^{1/2}/2.
\end{eqnarray*}
Setting  $\lambda_p :=  \|\bX\|_\infty (2np \log p)^{1/2}/4$, it follows  that in $\Omega_n$,  $\langle y - \mb E y, \bX(\beta - \beta^\ast) \rangle \leq \lambda_p \| \beta - \beta^*\|$. From the assumption in Theorem \ref{thm:statopt2}, we have 
$n\ \mbox{D}(\beta^\ast, \beta) > 2\langle y - \mb E y, \bX(\beta - \beta^\ast) \rangle $ inside $\Omega_n$ which further leads to  
\begin{eqnarray*}
\log \frac{{p}_{\l}(y \mid \beta, \xi, \bX)}{p(y \mid \beta^*, \bX)} \ind_{\Omega_n}   \leq \log \frac{p(y \mid \beta, \bX)}{p(y \mid \beta^*, \bX)} \ind_{\Omega_n}  \leq -n\ \mbox{D}(\beta^\ast, \beta). 
\end{eqnarray*}
For any $\epsilon > 0$, define $B(\beta_1; \epsilon) = \{{p}_{\l}(\cdot \mid \beta, \xi, \bX):  \| \beta - \beta_1 \| < \epsilon\}$. Denote by $\mathrm{conv} \{B(\beta_1; \epsilon)\}$ the convex hull of 
$B(\beta_1; \epsilon)$. Pick any $\beta_1$ such 
that $\|\beta_1 - \beta^*\| = r$.    Then, we have from Lemma 2 of \cite{bhattacharya2020nonasymptotic} and because of the assumption ${n}^{1/2} \geq p ({\log p})^{1/2}  \|X\|_\infty^2/\{\kappa_1\surd{2}\}$,
there exists measurable functions $0 \leq \Phi_n \leq 1$ such that for every $n \geq 1$ and $\gamma \in (0, 1)$ 
\begin{eqnarray}\label{eq:tests}
  \sup_{{p}_{\l}(\cdot \mid \beta, \xi, \bX) \in \mathrm{conv}\{B(\beta_1; r/2 )\}} \mb E_{\beta^\ast} \Phi_n + \mb E_{{p}_{\l}(\cdot \mid \beta, \xi, \bX)} (1- \Phi_n) &\leq&  \exp \big\{ -  (\gamma/2) n\ \mbox{D}( \beta^*, \beta_1) \big\}. 
 \end{eqnarray} 
We show in Lemma \ref{lem:test} that with high probability (w.r.t. $\mb P_{\beta^\ast}$), 
\begin{align}\label{pf_sketch_1}
\int \bigg[ \exp \big\{ \ell_n(\beta, \beta^\ast) + (\gamma /4) n\ \mbox{D}(\beta^\ast, \beta) \big\} \bigg] \,\pi_{\beta}(\beta) d\beta \le 2e^{n \kappa_2 \epsilon_n^2/2}, 
\end{align}
where  $\ell_n(\beta, \beta^\ast) = \log \{{p}_{\l}(y \mid \beta, \xi, \bX) / p(y \mid \beta^*, \bX)\}$.  
Next, we use the variational inequality  \eqref{eq:vi}
with $\mu = \pi_\beta$,  $\rho = \hat{q}_{\beta}$ to show with high probability
\begin{align*}
 \frac{\gamma}{4}\int n\ \mbox{D}(\beta^{\ast}, \beta) \, q^*_\beta(\beta)\,d\beta
& \le -   \int_\beta  \log\frac{{p}_{\l}(y\mid\beta, \xi, \bX)}{p(y\mid\beta^\ast, \bX)} \,q^*_\beta(\beta)\,d\beta+ \mbox{D}(q^*_\beta \, \vert \vert\, \pi_\beta) + n \kappa_2 \epsilon_n^2/2. 
\end{align*}
This brings us back to the proof of Theorem \ref{thm:statopt} and the remaining part of the proof of Theorem \ref{thm:statopt2} follows verbatim from the proof of Theorem \ref{thm:statopt}.

\subsection{Auxiliary results for proofs in Section \ref{sec:statopt}}\label{aux:1}
 \begin{lemma}\label{thm:cond}
 Let $u$ and $v$ denote two continuous random vectors with joint density function $p(u, v)$. The maximum value of 
 \begin{eqnarray*}
 \int q(u) \log \bigg\{\frac{p(u, v)}{q(u)} \bigg\} du
 \end{eqnarray*}
 over all density functions $q$ is attained by $q^*(u) = p(u \mid v)$.  
 \end{lemma}

\begin{lemma}\label{lem:test}
Fix any $\gamma \in (0, 1)$ and $\epsilon = \kappa_1 \varepsilon/(2\r_1 \|X\|_\infty \surd{p})$.   If  ${n}^{1/2} \geq p ({\log p})^{1/2}  \|X\|_\infty^2/\{\kappa_1\surd{2}\}$, then with probability $1- 2e^{- n \gamma \epsilon/8} - e^{-n \kappa_2 \varepsilon^2/2}$,  
\begin{align}\label{pf_sketch_1}
\int \exp \big\{ \ell_n(\beta, \beta^\ast) + (\gamma /4) n\ \mbox{D}(\beta^*, \beta) \big\} \,\pi_{\beta}(\beta) d\beta \le 2e^{n \kappa_2 \epsilon_n^2/2}. 
\end{align}
\end{lemma}
\begin{proof}
Writing  $\eta(\beta, \beta^*) =	\exp \big\{ \ell_n(\beta, \beta^\ast) + (\gamma /4)n\ \mbox{D}(\beta^\ast, \beta)\big\}$ and  
\begin{eqnarray}
U := \{ \beta: \|\beta - \beta^*\| > \varepsilon\} = \bigcup_{j=1}^\infty U_{j,n}	
\end{eqnarray}	
where $U_{j,n} = \{\beta: j \varepsilon < \| \beta - \beta^*\| < (j+1) \varepsilon\}$,  we express
\begin{eqnarray*}
\int \eta(\beta, \beta^*) \pi_\beta(\beta) d\beta &=& \int_{U^c}  \eta(\beta, \beta^*) \pi_\beta(\beta) d\beta + \int_{U} \tilde{\Phi}_n  \eta(\beta, \beta^*) \pi_\beta(\beta) d\beta\\ 
&{}&+ \sum_{j=1}^\infty \int_{U_{j,n}} ( 1- \tilde{\Phi}_n) \eta(\beta, \beta^*) \pi_\beta(\beta) d\beta \\
T &=& T_1 + S_1, \quad S_1 = T_2 + \sum_{j=1}^\infty T_{2j}
\end{eqnarray*}
for any sequence of test functions  $\{\tilde{\Phi}_n: n \geq 1\}$. 
Then 
\begin{eqnarray*}
\mb E_{\beta^\ast}  T_1 =  \int_{U^c} e^{(\gamma /4) n\ \mbox{D}(\beta^\ast, \beta)} \pi_\beta(\beta) d\beta \leq 
\int_{U^c} e^{n \gamma \kappa_1 \| \beta - \beta^\ast\|^2/4}\pi_\beta(\beta) d\beta \leq 
e^{n \gamma \kappa_2 \varepsilon^2/4}. 
\end{eqnarray*}
By Markov's inequality, $T_1 \le e^{n \kappa_2 \varepsilon^2/2}$ with probability $1- e^{-n \kappa_2 \varepsilon^2/2}$.  
To bound $T_2$ and $T_{2j}, j=1, \ldots, \infty$, we detail the construction of $\tilde{\Phi}_n$.  Let $N_{j,n} := N(j \varepsilon/2, U_{j,n}, \| \cdot \|)$ denote the $j \varepsilon/2$-covering number of $U_{j,n}$ with respect to  
$\| \cdot \|$.   For each $j \geq 1$, let $S_j$ be a maximal $j \epsilon_n/2$-separated points in $U_{j,n}$ and for each point $\tilde{\beta}_k \in S_j$ we can construct a test function  $\Phi_{n,\tilde{\beta}_k}$ as in \eqref{eq:tests}, with $r =  j \epsilon_n$.  Then we set  $\tilde{\Phi}_n$ to $
 \tilde{\Phi}_n = \sup_{j\geq 1} \max_{\tilde{\beta}_k \in S_j} \Phi_{n,\tilde{\beta}_k}$. 
Note also that 
\begin{eqnarray*}
 \sup_{{p}_{\l}(\cdot \mid \beta, \xi, \bX) \in \mathrm{conv}\{B(\beta_1; r/2 )\}}\mb E_{\beta^\ast} \Phi_n &+& \mb E_{{p}_{\l}(\cdot \mid \beta, \xi, \bX)} (1- \Phi_n) \\
 &\leq& \exp \bigg\{\gamma \lambda_p\|\beta_1 - \beta^\ast\|  -  \frac{n\kappa_1\gamma \|\beta_1 - \beta^*\|^2}{1 + \r_1 \|X\|_\infty \surd{p} \|\beta - \beta^*\|} \bigg\}\\
 &\leq & \exp \bigg\{\gamma\lambda_p\|\beta_1 - \beta^\ast\|  -  \frac{n\kappa_1\gamma\|\beta_1 - \beta^*\|}{ \r_1 \|X\|_\infty \surd{p}} \bigg\},  \\
 &\leq & \exp \bigg\{-\gamma\frac{n\kappa_1\|\beta_1 - \beta^*\|}{2 \r_1 \|X\|_\infty \surd{p}} \bigg\}, 
\end{eqnarray*}
where the second last inequality follows since ${n}^{1/2} \geq p (\log p)^{1/2}  \|X\|_\infty^2/\{\kappa_1 \surd{2}\}$. 
Then 
\begin{eqnarray*} 
\mb E_{\beta^\ast} \{ \tilde{\Phi}_n\} &\leq& \sum_{j=1}N_{j,n} \exp \bigg\{ - \gamma\frac{n\kappa_1 j \varepsilon}{2\r_1 \|X\|_\infty \surd{p}}\bigg\} \leq  \exp \bigg\{ - C \frac{n\kappa_1 \varepsilon}{2\r_1 \|X\|_\infty \surd{p}}\bigg\}  := e^{-Cn \epsilon}, 
\end{eqnarray*}	
for some constant $C > 0$.  
From Markov's inequality we obtain $ \tilde{\Phi}_n \leq e^{-Cn \epsilon}$ with probability at least $1 - e^{-Cn \epsilon}$. Hence $T_2 \leq S_1 e^{-Cn \epsilon}$ with probability at least $1 - e^{-Cn \epsilon}$.  Finally note that 
\begin{eqnarray*}
\mb E_{\beta^\ast} T_{2j} &=&  \int_{U_{j,n}} \mb E_{{p}_{\l}(\cdot \mid \beta, \xi, \bX)} ( 1- \tilde{\Phi}_n) e^{(\gamma /4)n \ \mbox{D}(\beta^\ast, \beta)} \pi_\beta(\beta) d\beta \\
&\leq& \int_{U_{j,n}} e^{-(\gamma/4)n\ \mbox{D}(\beta^\ast, \beta)} \pi_\beta(\beta) d\beta \leq 
\exp \bigg\{ - \gamma\frac{n\kappa_1j \varepsilon}{4\r_1 \|X\|_\infty \surd{p}}\bigg\}. 
\end{eqnarray*}
Hence $\sum_{j=1}^\infty E_{\beta^\ast}  T_{2j} \leq e^{- n \gamma \epsilon/4}$ and 
$E_{\beta^\ast} S_1 \leq e^{n \gamma \kappa_2 \varepsilon^2/4}E_{\beta^\ast} S_1+ e^{- n \gamma \epsilon/4}$, whence $E_{\beta^\ast} S_1 \leq 2e^{- n \gamma \epsilon/4}$.  Hence $S_1 \leq 2e^{- n \gamma \epsilon/8}$ with probability at least $1- 2e^{- n \gamma \epsilon/8}$. 
\end{proof}

\section{Review of Dynamical Systems \& Notion of Stability}\label{sec:dyn}
Dynamical systems theory is a classical technique that deals with stability and convergence of complex iterative methods. We call a dynamical system to be discrete-time if the system is observed on discrete time points $\set{t_0,t_1,t_2,\ldots}$. Usually, we consider the time-points to be evenly placed, i.e. $t_{j+1} = t_j + h$ for some $h>0$.
 Moreover, a system is considered autonomous if the function is independent of time and non-autonomous otherwise. 
 In this section, we will discuss the notion of stability for discrete time autonomous systems. Let us consider the following discrete-time autonomous system given by,
\begin{align}\label{sec_NOS_eq1}
    \psi^{t+1} = f(\psi^t), \quad t\in \mb{N}
\end{align}
where $f : \mb{R}^n \to \mb{R}^n$(or, $f : \mb{D} \to \mb{R}^n, \mb{D} \subseteq \mb{R}^n$) is a diffeomorphism, i.e. a smooth function with smooth inverse and $\psi^t \in \mb{R}^n$. $\psi^* \in \mb{R}^n$ is called a fixed point to this system if $\psi^* = f(\psi^*)$. We recall the following definition from \cite{bof2018lyapunov}.
\begin{definition}
A fixed point $\psi^*$ of a system given by \eqref{sec_NOS_eq1} is called

    (a) locally stable if given any $\epsilon > 0$, there exists $\delta = \delta(\epsilon)$ such that, whenever $\|\psi^0 - \psi^*\| < \delta$, we have $\|f(\psi^t) - \psi^*\| < \epsilon$ for all $t$.  

    (b) locally asymptotically stable if it is stable and $\delta$ can be chosen such that, whenever $\|\psi^0 - \psi^*\| < \delta$, we have $\psi^t \to \psi^*$ as $t \to \infty$.

    (c) locally unstable if it is not locally stable.

\end{definition}
The {\em locality}  in the definition is used to denote the fact that we are initializing the system in a $\delta$-ball around the fixed point. We say the stability is {\em global} if the system converges to the fixed point independent of the initialization, i.e. we can initialize at any point in the function domain.
\begin{lemma}\label{sec_NOS_lemma3}
Consider a system $x^{t+1} = g(x^{t})$ where $g : \mb{D} \to \mb{R}\,\ (\mb{D}\subseteq \mb{R})$ with a fixed point $x^*$ such that, $|g^{\prime}(x)|\leq \delta$ for all $x\in \mb{D}-\{x^*\}$, for some $\delta < 1$. Then, $x^*$ is globally asymptotically stable. 
\end{lemma}
\begin{proof}
Given a fixed point $x^* = g(x^*)$, use the Mean Value Theorem to get, $|x^{t+1}-x^*| = |g^{\prime}(x)||x^{t} - x^*|,\,\ \text{for some $x\in(x^t,x^*)$}$. Since, $|g^{\prime}(x)| \leq \delta $, we have, $|x^{t+1}-x^*| \leq \delta^{t+1} |x^{0}-x^*|$. This implies $|x^{t+1}-x^*| \to 0$ as $t \to \infty$.
\end{proof}
Let $\beta_k$ be the $k$-th coordinate of a vector $\beta \in \mb{R}^n$. Consider the system in \eqref{sec_NOS_eq1} with a fixed point $\psi^*$. Using generalized Taylor's theorem we get,
\begin{align*}
    \psi^{t+1}_k - \psi^*_k &= f_k(\psi^t) - f_k(\psi^*)
     = \nabla f_k (\psi^*) ( \psi^{t} - \psi^*) + h(\psi^t) |\psi^{t} - \psi^*|,
\end{align*}
where $\nabla f_k (\psi)$ is the gradient vector with $i^{th}$ entry given by 
$\partial f_k(\psi)/ \partial \psi_i$ and $h: \mb{R}^n \to \mb{R}$ such that, $\lim_{\psi \to \psi^*} h(\psi) = 0$. If $\psi_t$ is close to $\psi^*$, the convergence of the system depends on $\nabla f_k (\psi^*)$ by the following approximation, 
\begin{align}\label{sec_NOS_eqn2}
     (\psi^{t+1} - \psi^*) \approx \mathbf{J}( \psi^{t} - \psi^* ),
\end{align}
where  $\mathbf{J}$ is the $n \times n$ Jacobian matrix evaluated at $\psi^*$ with $i^{th}$ row given by $\nabla f^{\T}_i(\psi^*)$. Thus the behavior of the dynamical system \eqref{sec_NOS_eq1} around a small neighbourhood of $\psi^*$ is exactly same as that of the linearization in \eqref{sec_NOS_eqn2}. This is formalized in the  Hartman-Grobman theorem. 
\begin{definition}
A fixed point $\psi^*$, for a map $\psi \to f(\psi),\, \psi \in \mb{R}^n$ is called \emph{hyperbolic}  if none of the eigenvalues of $\mathbf{J}$
 has magnitude 1.
\end{definition}
\begin{theorem}[Hartman \& Grobman]\label{sec_NOS_thm1}
In a neighborhood of a hyperbolic fixed point, a diffeomorphism is topologically conjugate to the derivative at that fixed point. 
\end{theorem}
The theorem above asserts that the behavior of a system around a hyperbolic fixed point is essentially same as the linearization near this point. Refer to \cite{quandt1986hartman} for a complete review. This motivates us to check stability of a fixed point using Lemma \ref{sec_NOS_lemma2}. Refer to  \cite{wiggins2003introduction,barbarossa2011stability} for a proof and for further reading on this topic. 
\begin{definition}
For a square matrix A, the spectral radius $\rho(A)$ is defined by
$$\rho(A) := \max\set{|\lambda| : \lambda \text{ is eigenvalue of } A}$$
\end{definition}
\begin{lemma}\label{sec_NOS_lemma2}
Let $\psi^*$ be a fixed point solution to the discrete-time autonomous system given by $\psi_{t+1} = f(\psi_t)$. Suppose, $f:\mb{D} \to \mb{R}^n (\mb{D} \subseteq \mb{R}^n)$ is a twice continuously differentiable function around a neighbourhood $\mb{D}$ of $\psi^*$. Let $\mathbf{J} = [\partial_i f(\psi)/\partial \psi_j]_{\psi=\psi^*}$ be the Jacobian matrix of $f$ evaluated at $\psi^*$. Then,

    (a) $\psi^*$ is locally asymptotically stable if  $\rho(\mathbf{J}) < 1$.
  
    (b) $\psi^*$ is locally unstable if at least one eigenvalue of $\mathbf{J}$ is greater than one in absolute value. 
\end{lemma}
Lemma \ref{sec_NOS_lemma2} along with Theorem \ref{sec_NOS_thm1} provides sufficient conditions for the local convergence of a system.
Consider the linear system given by $\alpha^{(t+1)} = A \alpha^{(t)}$, with a fixed point $\alpha^* = 0$. Let us consider, $A\nu_i = \lambda_i \nu_i\, (i = 1,2,\ldots,n)$ and $|\lambda_1| \geq |\lambda_2| \geq \ldots \geq |\lambda_n|$. Suppose $A$ has a complete set of eigenvectors, i.e. the set of eigenvectors $\set{\nu_1,\nu_2,\ldots,\nu_n}$ form a basis of $\mb{R}^n$. Then it can be easily seen that a solution to the system is $\alpha^{(t)} = c_1 \lambda^t_1\nu_1 + c_2 \lambda^t_2\nu_2 + \ldots + c_n \lambda^t_n\nu_n$ for some arbitrary constants $c_1,c_2,\ldots,c_n$. Also, $c_1 \lambda^t_1\nu_1 + c_2 \lambda^t_2\nu_2 + \ldots + c_n \lambda^t_n\nu_n \to 0$ iff $|\lambda_1|<1$. This illustrates the Lemma above, in the most simplistic scenario, can be extended to the case where $A$ does not have a complete set of eigenvectors using \emph{Jordan Canonical form} of $A$ (refer to  \cite{wood2003always}). 

\section{Proofs of algorithmic convergence results in Section \ref{sec:algoconv}}
\begin{definition}\label{def_definite}
Consider two $n\times n$ real symmetric matrices $A$ \& $B$. Then we write,

    (a) $B \precsim A$ if for any $a \in \mb{R}^n$ such that, $a\neq 0$; we have, $a^{\T}(A-B)a\geq 0$, i.e. $A-B$ is a positive semi-definite matrix.

    (b) $B \prec A $ if for any $a \in \mb{R}^n$ such that, $a\neq 0$; we have, $a^{\T}(A-B)a > 0$,  i.e. $A-B$ is a positive definite matrix.

\end{definition}
In the following, we first provide the proof of Theorems \ref{main_VB_1} in \S \ref{appendix_thm2_proof}, calculation of spectral radius for $p=1$ in \S \ref{appendix_specialcase}  and then provide the proofs of some of the auxiliary results used in subsequent \S \ref{aux:2}. 

\subsection{Proof of {Theorem \ref{main_VB_1}} }\label{appendix_thm2_proof}
For some fixed $\alpha \in (0,1]$, the update equation in \eqref{sec_TTA_eqn3} can be rewritten as,
$$
      (\xi^{t+1})^2 = \mbox{diag}[\bX\{ \Sigma_{\alpha}(\xi^{t})/\alpha + \Sigma_{\alpha}(\xi^{t}) B_{\alpha}B^\T_{\alpha}\Sigma_{\alpha}(\xi^{t})\}\bX^{\T}].
$$
where $B_{\alpha}= [\bX^\T(Y-{1}/{2}\, \ind_n) + \Sigma^{-1}_{\beta}\mu_{\beta}/\alpha]$ and 
$\Sigma_{\alpha}(\xi)= [\Sigma^{-1}_{\beta}/\alpha - 2\bX^\T \text{diag}\{A(\xi) \}\bX]^{-1}$. We calculate the partial derivatives in order to get the Jacobian matrix,
$$
\frac{\partial \p{\xi^{t+1}_i}^2}{\partial \p{\xi^{t}_j}^2} = \frac{A^{\prime}(\xi^{t}_j)}{\xi^{t}_j}\bx^T_i\left[\Sigma_{\alpha}\left(\xi^t\right)
\bx_j \bx^\T_j\Sigma_{\alpha}\left(\xi^t\right)/\alpha + 2 \Sigma_{\alpha}\left(\xi^t\right)
\bx_j \bx^\T_j\Sigma_{\alpha}\left(\xi^t\right) B_{\alpha}B_{\alpha}^\T\Sigma_{\alpha}\left(\xi^t\right)\right]\bx_i.
$$
Then Jacobian Matrix($\mathbf{J}_{\alpha}$) at $\xi= \xi^t$ is given by,
$$
\mathbf{J}_{\alpha} = \bigg[\bX\Sigma_{\alpha}\left(\xi^t\right)\bX^\T \circ \bX\left\{ \Sigma_{\alpha}\left(\xi^t\right)/{\alpha} + 2\mu_{\alpha}\left(\xi^t\right)\mu_{\alpha}^\T\left(\xi^t\right)\right\}\bX^\T\bigg] \mbox{diag}\bigg(\frac{A^{\prime}\left(\xi^t\right)}{\xi^t}\bigg),
$$
where \textbf{$\circ$} denotes the {Hadamard Product}. Let us denote the maximum eigenvalue of a matrix $A$ by $\lambda_1(A)$. Our objective is to show that $\lambda_1(\mathbf{J}_{\alpha})\vert_{\xi = \xi^*} < 1$. We call $D= \mbox{diag}\left({A^{\prime}(\xi)}/{\xi}\right)$. By Lemma \ref{sec_ev_2} $D^{1/2}\left[\bX\Sigma_{\alpha}(\xi^*)\bX^\T \circ \bX\left\{ \Sigma_{\alpha}(\xi^*)/\alpha + 2\mu_{\alpha}(\xi^*)\mu_{\alpha}^\T(\xi^*)\right\}\bX^\T\right]D^{1/2}$ has the same set of eigenvalues as with  $\mathbf{J}_{\alpha}\vert_{\xi = \xi^*}$. $\bX\Sigma_{\alpha}(\xi^*)\bX^\T$ and $\bX\left\{ \Sigma_{\alpha}(\xi^*)/\alpha + 2\mu_{\alpha}(\xi^*)\mu_{\alpha}^\T(\xi^*)\right\}\bX^\T$ are positive semi-definite matrices which implies $D^{1/2}\left[\bX\Sigma_{\alpha}(\xi^*)\bX^\T \circ \bX\left\{ \Sigma_{\alpha}(\xi^*)/\alpha + 2\mu_{\alpha}(\xi^*)\mu_{\alpha}^\T(\xi^*)\right\}\bX^\T\right]D^{1/2}$ is positive semi-definite as well as symmetric. Since the eigenvalues of a real symmetric positive semi-definite matrix are real and non-negative, the eigenvalues of $\mathbf{J}_{\alpha}\vert_{\xi = \xi^*}$ are real and non-negative. We denote $\xi^*$ by $\xi$ in the following discussion for notational simplicity. 

 From the assumptions of the theorem and fixed point equation, it is clear that $\xi_i>0$ for all $i\in \set{1,2,\ldots,n}$. We begin with the fact that, $A(x)+x A^{\prime}(x) < 0 $ for all $x\in \mb{R}$. Then, recalling the Definition\,\ref{def_definite}, we have the following, 
\begin{align*}
    2[\bX^\T \text{diag}\{A(\xi) + \xi A^{\prime}(\xi)\} \bX] \prec \Sigma^{-1}_{\beta}/{\alpha},
\end{align*}
Since for all non-zero $a\in \mathbb{R}^p$, we have $a^\T(\Sigma^{-1}_{\beta}/{\alpha} - 2[\bX^\T \text{diag}\{A(\xi) + \xi A^{\prime}(\xi)\} \bX])a > 0$, assuming $\Sigma_{\beta}$ to be a positive definite matrix. Then, 
\begin{align}\label{pd_ineq1}
    2\bX^\T \text{diag}\{\xi A^{\prime}(\xi)\} \bX &\prec \Sigma^{-1}_{\beta}/{\alpha} - 2 \bX^\T \text{diag}\{A(\xi)\} \bX.
\end{align}
Now, $\bX \Sigma_{\alpha}(\xi) \bX^\T =  \bX[\Sigma^{-1}_{\beta}/{\alpha} - 2 \bX^\T \text{diag}\{A(\xi)\} \bX]^{-1}\bX^{\T}$ is a positive semi-definite matrix. This implies $\D^{1/2}\left[\bX\Sigma_{\alpha}(\xi)\bX^\T \circ \bX\Sigma_{\alpha}(\xi)\bX^\T\right]\D^{1/2}/{\alpha}$ is positive semi-definite by Schur product theorem. Then we have the following,
\begin{align}\label{sec_ev_eqn1}
    &{} \D^{1/2}\left[\bX\Sigma_{\alpha}(\xi)\bX^\T \circ \bX\left\{ \Sigma_{\alpha}(\xi)/{\alpha} + 2\mu_{\alpha}(\xi)\mu_{\alpha}^\T(\xi)\right\}\bX^\T\right]  \D^{1/2}\nonumber\\
    &\precsim 2\, \D^{1/2}\left[\bX\Sigma_{\alpha}(\xi)\bX^\T \circ \bX\left\{ \Sigma_{\alpha}(\xi)/{\alpha} + \mu_{\alpha}(\xi)\mu_{\alpha}^\T(\xi)\right\}\bX^\T\right]\D^{1/2}.
\end{align}

Recall $(\xi)^\T_{1 \times n} = [\xi_1, \xi_2,\ldots,\xi_n]$ and denote $[\Lambda_{\alpha}(\xi)]_{p \times p}= \Sigma_{\alpha}(\xi)/{\alpha} + \mu_{\alpha}(\xi)\mu_{\alpha}^\T(\xi)$ and $ [Q_{\alpha}]_{n \times n}= \bX\Lambda_{\alpha} \bX^\T$. Then $Q_{\alpha} =  \Delta(\xi) \circ \Gamma_{\alpha}$
where, $\Delta (\xi)= \xi\, \xi^\T$ and $\Gamma_{\alpha} = \text{diag}(1/\xi)\, Q_{\alpha} \, \text{diag}(1/\xi)$. Now, $\Gamma_{\alpha}$ is positive definite because $Q_{\alpha}$ is positive definite and $1/\xi>0$ for all $\xi\in \mb{R}^+$. Note that, $[\Gamma_{\alpha}]_{ii}=1$ for all $i\in \set{1,2,\ldots,n}$ because at the fixed point solution $\xi^2_i = [Q_{\alpha}]_{ii}$ for all $i\in \set{1,2,\ldots,n}$. Using the above expression and properties of {Hadamard product}, we rewrite \eqref{sec_ev_eqn1}  as 
\begin{align}\label{sec_ev_eqn2}
    &2\, \D^{1/2}[\bX \Sigma_{\alpha}(\xi) \bX^\T \circ \bX\left\{ \Sigma_{\alpha}(\xi)/{\alpha} + \mu_{\alpha}(\xi)\mu_{\alpha}^\T(\xi)\right\}\bX^\T]\D^{1/2} \nonumber\\
    &=2\, \D^{1/2}\left\{ \bX \Sigma_{\alpha}(\xi) \bX^\T \circ \Delta(\xi)\right\}\D^{1/2}  \circ \Gamma_{\alpha}.
\end{align}
Next we can write, 
\begin{align*}
    \D^{1/2} \left\{\bX \Sigma_{\alpha}(\xi) \bX^\T \circ \Delta(\xi)\right\} \D^{1/2}    = \mbox{diag}\Big[\{ \xi A^{\prime}(\xi)\}^{1/2}\Big] \bX \Sigma_{\alpha}(\xi) \bX^\T \mbox{diag}\Big[\{ \xi A^{\prime}(\xi)\}^{1/2}\Big]. 
\end{align*} 
The above equality follows from the fact that the $(i,j)^{th}$ entry of the matrices on the both side of the equation is given by, $\{\xi_iA^{\prime}(\xi_i)\}^{1/2} \,[\bX \Sigma_{\alpha}(\xi) \bX^\T]_{ij}\, \{\xi_jA^{\prime}(\xi_j)\}^{1/2}$. Let us call $ \mbox{R}_{\alpha} = 2\, \mbox{diag}[\{ \xi A^{\prime}(\xi)\}^{1/2}] \bX \Sigma_{\alpha}(\xi) \bX^\T \mbox{diag}[\{ \xi A^{\prime}(\xi)\}^{1/2}]$. Then $\mbox{R}_{\alpha}$ has the same set of non-zero eigenvalues with $ 2\, \Sigma_{\alpha}(\xi) \bX^\T \text{diag}\left\{{{\xi}A^{\prime}\left(\xi\right)}\right\}\bX$. Using Lemma \ref{sec_ev_4} with $B=2\, \bX^\T \text{diag}\left\{{{\xi}A^{\prime}\left(\xi\right)}\right\}\bX$ and $A = \Sigma^{-1}_{\beta}/{\alpha} - 2 \bX^\T \text{diag}\{A(\xi)\} \bX = \Sigma_{\alpha}^{-1}(\xi)$, along with \eqref{pd_ineq1} we have, $\lambda_1(2\, \Sigma_{\alpha}^{1/2}(\xi) \,\bX^\T \text{diag}\left\{{{\xi}A^{\prime}\left(\xi\right)}\right\}\bX\, \Sigma_{\alpha}^{1/2}(\xi) )<1$. Hence we can write, $\lambda_{1}( 2\, \Sigma_{\alpha}(\xi) \bX^\T \text{diag}\left\{{{\xi}A^{\prime}\left(\xi\right)}\right\}\bX)< 1$ which implies $\lambda_1(\mbox{R}_{\alpha})<1$. Also, we can rewrite \eqref{sec_ev_eqn2} 
\begin{align}\label{sec_ev_eqn3}
2\, \D^{1/2}\left\{ \bX \Sigma_{\alpha}(\xi) \bX^\T \circ \Delta(\xi)\right\}\D^{1/2}  \circ \Gamma_{\alpha} = \mbox{R}_{\alpha} \circ \Gamma_{\alpha}.
\end{align}
Finally we use Lemma \ref{sec_ev_3} on \eqref{sec_ev_eqn3} with $A=\Gamma_{\alpha}$ and $B=\mbox{R}_{\alpha}$ for $k=1$. This concludes the proof. 

\subsection{Calculation of spectral radius for $p=1$}\label{appendix_specialcase}
For a fixed $\alpha \in (0,1]$, One can write the updtaes from \eqref{sec_TTA_eqn3} for $p=1$ 
\begin{equation}\label{eq:ev_p=1}
\p{\xi^{t+1}_i}^2 = x^2_i[\sigma_{\alpha}(\xi^t) / \alpha + \{c\, \sigma_{\alpha}(\xi^t)\}^2] \quad (i = 1, 2,\ldots, n),
\end{equation}
where $\sigma_{\alpha}(\xi^t) =  \{ \sigma_{\beta}^{-2}/\alpha - 2\sum^n_{i=1} x^2_iA(\xi_i^t)\}^{-1} $ and $c = x^{\prime}(y-(1/2)\ind_n)+ {\mu_{\beta}}/\{\alpha\,\sigma_{\beta}^2\}$. We calculate ${\partial(\xi^{t+1}_i)^2}/{\partial ({\xi^{t}_j})^2}$ to get the Jacobian Matrix.
\begin{align*}
    \frac{\partial\p{\xi^{t+1}_i}^2}{\partial ({\xi^{t}_j})^2} &= {2 x^2_ix^2_j \{A^{\prime}(\xi^{t}_j)/2 \xi^{t}_j}\}\, \{\sigma^2_{\alpha}(\xi^t)/\alpha + {2c^2}{\sigma^3_{\alpha}(\xi^t)}\}.
\end{align*} 
Denote $\eta_{\alpha}(\xi^t)= \sigma^2_{\alpha}(\xi^t)/\alpha + {2c^2}{\sigma^3_{\alpha}(\xi^t)}$ and $a_{ij} = 2 x^2_ix^2_j {A^{\prime}(\xi^t_j)}/{2 \xi^t_j}$. Then the $(i,j)^{th}$ entry of the Jacobian Matrix $\mathbf{J}_{\alpha}^t$ at $\xi=\xi^t$ is given by $[\mathbf{J}_{\alpha}^t]_{ij} = \eta^t  a_{ij}$.
Since $[\mathbf{J}_{\alpha}^t]_{*j}= [\mathbf{J}_{\alpha}^t]_{*1} \times {x^2_j}/{x^2_1} \times {A^{\prime}(\xi^t_j)}/{A^{\prime}(\xi^t_1)} \times {\xi^t_1}/{\xi^t_j}$ it follows Rank($\mathbf{J}_{\alpha}$)=1. Here, $[\mathbf{J}_{\alpha}^t]_{*j}$ is the $j^{th}$ column of $\mathbf{J}_{\alpha}^t$. Order the eigenvalues $\lambda_1 \geq \lambda_2 \geq \ldots \geq \lambda_n$. Then assuming $\mathrm{tr}(\mathbf{J}_{\alpha}^t) \neq 0$ and using Lemma \ref{sec_ev_1} we obtain that the non-zero eigenvalue of $\mathbf{J}_{\alpha}^t$ is given by,
\begin{align}\label{eigen_p=1}
\lambda_{1} = \sum^n_i \eta_{\alpha}(\xi^t) a_{ii} = \sum^n_i x^4_i {A^{\prime}(\xi^t_i)}/{ \xi^t_i}\{\sigma^2_{\alpha}(\xi^t)/\alpha + {2c^2}{\sigma^3_{\alpha}(\xi^t)}\}.
\end{align}
{\em Further Simplification at $\xi^t = \xi^*$.}
In case of $p=1$, the self-consistency or the  fixed point equation \eqref{eq:fp} for \eqref{eq:ev_p=1} is given by
 \begin{equation}\label{eq_self_consis}
\p{\xi^{*}_i}^2 = x^2_i[ \sigma_{\alpha}(\xi^*) / \alpha + \{c\, \sigma_{\alpha}(\xi^*)\}^2].
\end{equation}
From \eqref{eigen_p=1}, we can calculate $\lambda_{1}$ at $\xi = \xi^*$ ,
\begin{align*}\label{upp_bd_1}\noindent
    \lambda_{1} = \sum^n_{i=1} x^4_i \{A^{\prime}(\xi^*_i)/ \xi^*_i\}\{\sigma^2_{\alpha}(\xi^*)/\alpha + {2c^2}{\sigma^3_{\alpha}(\xi^*)}\}.\nonumber
\end{align*}
Substituting \eqref{eq_self_consis} into the \eqref{upp_bd_1} gives us,
\begin{align}
    \lambda_1 &= \sum^n_{i=1} x^4_i \{A^{\prime}(\xi^*_i)/ \xi^*_i\} \sigma_{\alpha}(\xi^*) [2 (\xi^*_i)^2 /x^2_i - \sigma_{\alpha}(\xi^*)/\alpha], \nonumber\\
    &\leq \frac{ \sum^n_{i=1}2 x^2_i A^{\prime}(\xi^*_i) \xi^*_i}{\sigma_{\beta}^{-2}/\alpha- \sum^n_{i=1} 2x^2_i A(\xi^*_i)}.
\end{align}
The above inequality follows from the fact that $\sum^n_{i=1} \{x^4_i A^{\prime}(\xi^*_i)/ \xi^*_i\} \{\sigma^2_{\alpha}(\xi^t)/\alpha\}> 0$ since $A(\xi^*_i)/\xi^*_i >0$ for all $\xi_i^* \in \mb{R}^+$.  From \eqref{sec_TTA_prop1_eqn1} as $\sigma_{\beta} > 0$, we obtain
$$
2 \sum^n_i \left\{A^{\prime}(\xi^*){\xi^*} + A(\xi^*)\right\}x^2_i< 0 <\sigma^{-2}_{\beta}/\alpha.
$$
By rearranging the terms it follows that $\lambda_1< 1$.

\subsection{Auxiliary results for the proofs in Section \ref{sec:algoconv}} \label{aux:2} 

\begin{lemma}\label{sec_ev_2}
For a symmetric matrix $\mbox{M}_{n\times n}$ and a invertible diagonal matrix $\mbox{N}_{n\times n}$, the set of eigenvalues of $\mbox{MN}$ and $\mbox{N}^{1/2} \mbox{MN}^{1/2}$ are the same.  
\end{lemma}
\begin{proof}
The characteristic equation for $\mbox{MN}$ is given by, $\vert\mbox{MN} - \lambda \mb I\vert = 0$. Since, $\mbox{N}$ is invertible we can write, $|\mbox{N}^{1/2}|\vert\mbox{MN} - \lambda \mb I\vert|\mbox{N}^{-1/2}| = 0 $ which implies
$\vert\mbox{N}^{1/2}\mbox{M}\mbox{N}^{1/2} - \lambda \mb I\vert = 0$. Hence the proof. 
\end{proof}
\begin{lemma}\label{sec_ev_3}
Let $A,B$ be $n\times n$ given positive semidefinite Hermitian matrices. Arrange the eigenvalues of $A \circ B$ and $B$ and the main diagonal entries $d_i(A)$ of A in decreasing order $\lambda_1 \geq \lambda_2 \geq \ldots \geq \lambda_n$ and $d_1(A) \geq d_2(A) \geq \ldots \geq d_n(A)$. Then, 
\begin{align*}
 \sum_{i=1}^k \lambda_{i}(A\circ B) \leq  \sum_{i=1}^k d_i(A)\lambda_{i}(B), \quad k= \set{1,2,\ldots n},
\end{align*}
\end{lemma}
\begin{proof}
See {Theorem 5.5.12} in {\cite{horn1994topics}}.
\end{proof}
\begin{lemma}\label{sec_ev_4}
For two $n\times n$ symmetric matrices $\mbox{A}$ and $\mbox{B}$, such that $\mbox{B} \prec \mbox{A}$ where $\mbox{A}$ is positive definite and $\mbox{B}$ is positive semi-definite. Then the largest eigenvalue of $\mbox{A}^{-1/2}\,\mbox{B}\,\mbox{A}^{-1/2}$ given by $\lambda_{1}(\mbox{A}^{-1/2}\,\mbox{B}\,\mbox{A}^{-1/2})$ is less than 1.
\end{lemma}
\begin{proof}
Since $\mbox{A}-\mbox{B}$ is positive definite and $\mbox{A}$ is invertible, it is easy to see that $\mbox{A}^{-1/2}(\mbox{A}-\mbox{B})\mbox{A}^{-1/2}$ is also positive definite. Then the smallest eigen value of $I_n - \mbox{A}^{-1/2}\,\mbox{B}\,\mbox{A}^{-1/2}$ is bigger than 0. This implies $\lambda_{1}(\mbox{A}^{-1/2}\,\mbox{B}\,\mbox{A}^{-1/2}) < 1$.
\end{proof}

\begin{lemma}\label{sec_ev_1}
For an $n\times n$ matrix with rank 1, the number of non-zero eigenvalues is at most 1. If trace of the matrix (denoted $\mbox{tr}(A)$) is non-zero then a non-zero eigenvalue exists and equal to trace of the matrix.
\end{lemma}
\begin{proof}
Suppose an $n \times n$ matrix A has two non-zero eigenvalue $\lambda_1,\lambda_2$ with non-zero linearly independent eigenvectors $v_1,v_2$. Then $Av_1 = \lambda_1 v_1$ and $Av_2 = \lambda_2 v_2$. This contradicts the fact  $\mbox{rank}(A)=1$. Now assuming that $\mathrm{tr}(A) \neq 0$, and using the fact $\mathrm{tr}(A) = \sum^n_{i=1} \lambda_i$, we claim that a non-zero eigenvalue exists and $\lambda_1 = \mathrm{tr}(A)$.
\end{proof}

\subsection{Global convergence rate in a semi-orthogonal case: Proof of Theorem \ref{thm_semi_ortho}}\label{ortho_proof_convergence}
Recall that throughout this proof we are going to assume $\sigma_{\beta}=1$. We consider two separate cases, given by $n=1$ and $n\geq2$.\\
\noindent {\em Case $n=1$:}
 For notational convenience, let us call $h_{1 ,n}^{\prime}(z) = h_{n}^{\prime}(z)$. For $n=1$ we have,
\begin{align*}
    h_1^{\prime}(z) = \frac{A^{\prime}(\sqrt{z})}{\sqrt{z}} (1 - 2 A(\sqrt{z}))^{-2}\trd{1 + \frac{1}{2}(1 - 2A(\sqrt{z}))^{-1}}.
\end{align*}
From Proposition \ref{sec_TTA_prop1} we have $1 \leq \trd{1-2A(\sqrt{z})} \leq 5/4$ which  $h_1^{\prime}(z) \leq 3{A^{\prime}(\sqrt{z})}/{2\sqrt{z}}$. In the following, we derive  an upper bound for $A^{\prime}(x) /x$ for $x\in \mb{R}^+$. 
\begin{align}\label{ortho_eq1}
    \frac{A^{\prime}(x)}{x} &= \frac{(e^x - x)^2 - (1 + x^2)}{4x^3(1+e^x)^2},\nonumber \\
    &= \frac{2\sum^{\infty}_{n=3} x^{n-3}/n! + \sum^{\infty}_{n=4} x^{n-3}[1/\{2!(n-2)!\} + 1/\{3!(n-3)!\} + \cdots + 1/\{(n-2)!2!\}]}{4(1+e^x)^2}, \nonumber\\
    &\leq \frac{2e^x}{4(1+e^x)^2} + \frac{8 e^{2x}}{4(1+e^x)^2}.
\end{align}
The inequality in \eqref{ortho_eq1} is due to $\sum^{\infty}_{n=3} x^{n-3}/n! < e^x$ and $x^{n-3}[1/\{2!(n-2)!\} + 1/\{3!(n-3)!\} + \cdots + 1/\{(n-2)!2!\}]$ and $\sum^{\infty}_{n=1} x^{n-3}\, 2^{n}/n! \leq 2^3 \exp(2x)$. Using  $\exp(x) + \exp(-x) \geq 2$ for all $x \in \mb{R}$, we further obtain, 
\begin{align*}
     \frac{2e^x}{4(1+e^x)^2} + \frac{8 e^{2x}}{4(1+e^x)^2} &\leq \frac{1}{2(e^{-x} + e^x + 2)} + \frac{8}{4(e^{-x} + e^x)^2}
    \leq \frac{1}{8} + \frac{8}{16} = 5/8.
\end{align*}
Hence $\|h_1^{\prime}\|_\infty \leq 15/16$. \\
\noindent {\em Case $n\geq2$:}
We begin with the function $h^{\prime}_n(z)$ which is given by,
\begin{align}
h_{n}^{\prime}(z) =  \frac{A^{\prime}(\sqrt{z})}{\sqrt{z}} \sigma_n^{-2}\trd{\frac{1}{n} +  \frac{1}{2\,\sigma_n}},
\end{align}
where, $\sigma_n = \{1/n- 2\,A(\sqrt{z})\}$. In Lemma \ref{semi_ortho_lemma1} we show that for any $z\in \mb{R}^+$, $h_n^{\prime}(z)$ is a monotonically increasing function of $n$, provided $n\geq2$. And also $h_n^{\prime}$ converges pointwise to $h^{\prime}(z):= - A^{\prime}(\sqrt{z})/ \{16 \,\sqrt{z} A^3(\sqrt{z})\}$ and $h'(z) < 1$ for $z \in \mathbb{R}^+$. So for any fixed $z\in \mb{R}^+$ and $n \in \set{2,3,4,\ldots}$ we have $h_n^{\prime}(z) \leq h^{\prime}(z)<1$.  Hence $\| h_n^{\prime}\|_\infty < 1 $ for any fixed $n$. 

\subsection{Auxiliary results for the global convergence rate result}
The function $A(\cdot)$ plays a crucial role in studying the convergence of the EM. The following proposition provides some properties of $A(\xi)$. 
\begin{proposition}\label{sec_TTA_prop1}
The following are true for the function $A(\xi):=-\tanh{(\xi/2)}/4\xi$, defined on $\mb{R}^+$.
 $A: \mb{R}^+ \to \mb{R}^-$ is monotonically increasing and twice continuously differentiable with $A(0)=-1/8$.
 \end{proposition}
 \begin{proof}
It is easy to see that the {range} of $A(\cdot) \subseteq \mb{R}^-$ since $\tanh(\xi/2) > 0$ for all $\xi\in \mb{R}^+$. $A(0)=-1/8$ follows from the fact that, $\lim_{\xi \to 0} \{\exp(\xi) - 1\}/\xi = 1$ and  
 $A(\xi) = -{\{\exp(\xi)-1\}}/{4\xi\{\exp(\xi)+1\}}$. Differentiating $\xi\,A(\xi)$ gives the following for all $\xi\in \mb{R}^+$,
 \begin{align}\label{sec_TTA_prop1_eqn1}
     A(\xi)+\xi A^{\prime}(\xi) = -\frac{1}{2}\frac{e^{\xi}}{(e^{\xi}+1)^2} <0.
 \end{align}
 It follows immediately that, 
 \begin{align}\label{sec_TTA_prop1_eqn2}
        {A^{\prime}(\xi)} = \frac{(e^\xi - \xi)^2 - (1 + \xi^2)}{4\xi^2(1+e^\xi)^2}.
 \end{align}
Since $ (e^\xi - \xi)^2 - (1+ \xi^2) > 0 $ for all $ \xi\in \mb{R}^+ $, $A'(\xi) >0$ for all $ \xi\in \mb{R}^+ $. 
Also, $A^{\prime\prime}(\xi)=-{2A^{\prime}(\xi)}/{\xi}+4A(\xi)\big[A(\xi)+\xi A^{\prime}(\xi)\big]$ is a continuous function, thus completing the claim. 
 \end{proof}

\begin{lemma}\label{semi_ortho_lemma1}
For any $n \ge 2$, the following claims are true for the function $h^{\prime}_n : \mb{R}^+\to \mb{R}^+$, \\
(a) For any fixed $z\in \mb{R}^+$, $h_{n}^{\prime}(z)$ is an increasing function of $n$.
\\
(b) For any fixed $z\in \mb{R}^+$, define $h^{\prime}(z) = - A^{\prime}(\sqrt{z})/ \{16 \,\sqrt{z} A^3(\sqrt{z})\}$. Then, $h_{n}^{\prime}(z)$ converges pointwise to $h^{\prime}(z)$. Also, $h^{\prime}(z) < 1$ for all $z\in \mb{R}^+$.
\end{lemma}
\begin{proof} 
\noindent {\em Part (a):} From Proposition \ref{sec_TTA_prop1} it is clear that  $h_{n}^{\prime}(z) > 0$  for all $z\in \mb{R}^+$. 
For any fixed $z\in \mb{R}^+$, it is easy to see $\sigma^{-2}_n$ increases with $n$. Next we show that for $n\geq2$, 
\begin{align}\label{semi_ortho_proof_1}
\frac{1}{n} +  \frac{1}{2\,\sigma_n} < \frac{1}{n+1} +  \frac{1}{2\,\sigma_{n+1}}.
\end{align}
We begin with the fact that for $n\geq 2$, $\{(n+1)^{-1} +  1/4\}(1/n + 1/4)  < 1/2$. From Proposition \ref{sec_TTA_prop1} we know that $-1/8\leq A(\sqrt{z}) < 0$ for all $z\in \mb{R}^+$. Then for any fixed $z$ on $\mb{R}^+$ we have $\sigma_{n+1} \,\sigma_n <1/2$ which when 
multiplied on the both sides by $1/n - 1/(n+1)$ yields \eqref{semi_ortho_proof_1}. 
This proves the first part of  Lemma \ref{semi_ortho_lemma1}.  

\noindent {\em Part (b):} For a fixed $z\in\mb{R}^+$, it is easy to see that,  $\sigma_n \to -2A(\sqrt{z})$ as $n \to \infty$. This leads to 
\begin{eqnarray}\label{eq:lim}
\lim_{n\to \infty} \trd{\frac{1}{n} + \frac{1}{2\,\sigma_n}} = - \frac{1}{4A(\sqrt{z})}.
\end{eqnarray}
Multiplying \eqref{eq:lim} with $A^{\prime}(\sqrt{z})/\sqrt{z}$ for a fixed $z\in\mb{R}^+$, we get 
$\lim_{n \to \infty} h^{\prime}_n(z) = h^{\prime}(z) $.  Next we show that for $h^{\prime}(z)<1$ for any $z\in\mb{R}^+$. From Proposition \ref{sec_TTA_prop1},
\begin{align}
\frac{-\,A^{\prime}(\sqrt{z})}{16 \,\sqrt{z} \, A^3(\sqrt{z})} = \frac{\big(e^{2\sqrt{z}} - 2\sqrt{z}e^{\sqrt{z}} - 1\big) \big(1+ e^{\sqrt{z}}\big)}{\big(e^{\sqrt{z}}-1\big)^3} > 0.\label{eq:lim2}
\end{align}
Next, write $\psi(x) = 2(e^{x}-1) -x(e^{x}+1)$. Then $\psi'(x) = e^x - xe^x -1$ and $\psi''(x) = -xe^x$.  Hence $\psi(0) = 0$, $\psi'(0) = 0$ and $\psi'$ is decreasing, which entails $\psi$ is decreasing for $x > 0$ and $\psi(x) < 0$ for $x > 0$. Hence $2(e^{\sqrt{z}}-1) - \sqrt{z}(e^{\sqrt{z}}+1)<0$ for $z \in \mathbb{R}^+$ and the numerator of the right hand side of \eqref{eq:lim2} is 
$$
\big(e^{2\sqrt{z}} - 2\sqrt{z}e^{\sqrt{z}} - 1\big) \big(1+ e^{\sqrt{z}}\big) - \big(e^{\sqrt{z}}-1\big)^3 =2 e^{\sqrt{z}} \big\{2(e^{\sqrt{z}}-1) - \sqrt{z}(e^{\sqrt{z}}+1)\big\} < 0.
$$
This proves the second part of  Lemma \ref{semi_ortho_lemma1}.  
\end{proof}

\bibliographystyle{plainnat}
\bibliography{reference}


%
%
%
\end{document}